\documentclass[oneside,english]{amsart}
\usepackage[T1]{fontenc}
\usepackage[latin9]{inputenc}
\usepackage{mathtools}
\usepackage{amsbsy}
\usepackage{amstext}
\usepackage{amsthm}
\usepackage{amssymb}
\usepackage[all]{xy}

\makeatletter
\numberwithin{figure}{section}
\theoremstyle{plain}
\newtheorem{thm}{\protect\theoremname}
  \theoremstyle{remark}
  \newtheorem{rem}[thm]{\protect\remarkname}
  \theoremstyle{definition}
  \newtheorem{defn}[thm]{\protect\definitionname}
  \theoremstyle{plain}
  \newtheorem{prop}[thm]{\protect\propositionname}
  \theoremstyle{plain}
  \newtheorem{lem}[thm]{\protect\lemmaname}
  \theoremstyle{plain}
  \newtheorem{cor}[thm]{\protect\corollaryname}

\usepackage{amsfonts}\usepackage{amsthm}\usepackage{enumerate}\usepackage{mathrsfs}\usepackage{cancel}\usepackage{url}
\usepackage{MnSymbol}\usepackage{hyperref}
\usepackage{bbm}

\input xy\huge
\xyoption{all}

\DeclareFontFamily{OT1}{pzc}{}
\DeclareFontShape{OT1}{pzc}{m}{it}{<-> s * [1.10] pzcmi7t}{}
\DeclareMathAlphabet{\mathpzc}{OT1}{pzc}{m}{it}

\newcommand{\rr}{\mathbb{R}}

\newcommand{\pp}{\mathbb{P}}

\newcommand{\ff}{\mathcal{F}}
\newcommand{\gc}{\mathcal{G}}
\newcommand{\kk}{\mathcal{K}}

\newcommand{\cc}{\mathbb{C}}

\newcommand{\zz}{\mathbb{Z}}

\newcommand{\lc}{\mathcal{L}}
\newcommand{\xx}{\mathcal{X}}
\newcommand{\yy}{\mathcal{Y}}
\newcommand{\mm}{\mathcal{M}}
\newcommand{\bb}{\mathcal{B}}

\newcommand{\im}{\operatorname{Im}}
\newcommand{\coker}{\operatorname{coker}}

\newcommand{\cl}{\mathcal{C}}

\newcommand{\nn}{\mathbb{N}}

\newcommand{\bu}{\mathfrak{bu}}

\newcommand{\manc}{\textbf{Man}^c}

\newcommand{\tb}{\mathbb{T}}

\newcommand{\zc}{\mathcal{Z}}
\newcommand{\sfr}{\mathfrak{s}}
\newcommand{\basic}{\mathfrak{b}}
\newcommand{\tc}{\mathcal{T}}
\newcommand{\ts}{\mathscr{T}}
\newcommand{\aru}[2]{\ar[#1]^-{#2}}
\newcommand{\ard}[2]{\ar[#1]_-{#2}}

\newcommand{\lf}{\mathsf{l}}
\newcommand{\kf}{\mathsf{k}}

\newcommand{\wc}{\mathcal{W}}
\usepackage{tikz} 
\usetikzlibrary{matrix,arrows}
\usepackage{tensor}
\newcommand{\fibp}[2]{\tensor[_{#1\thinspace}]{\times}{_{\thinspace #2}}}
\newcommand{\ev}{\operatorname{ev}}
\newcommand{\evb}{\operatorname{evb}}
\newcommand{\evi}{\operatorname{evi}}
\newcommand{\Sym}{\operatorname{Sym}}
\newcommand{\For}{\operatorname{For}}
\newcommand{\Or}{\operatorname{Or}}
\newcommand{\ed}{\operatorname{ed}}
\newcommand{\id}{\operatorname{id}}
\newcommand{\pt}{\operatorname{pt}}
\newcommand{\pr}{\operatorname{pr}}

\newcommand{\sstar}{\smallstar}
\newcommand{\sgn}{\operatorname{sgn}}
\newcommand{\inv}{\operatorname{inv}}
\newcommand{\jc}{\mathcal{J}}
\newcommand{\no}{n_\text{odd}}
\newcommand{\cnt}{\operatorname{cnt}}

\makeatother

\usepackage{babel}
  \providecommand{\corollaryname}{Corollary}
  \providecommand{\definitionname}{Definition}
  \providecommand{\lemmaname}{Lemma}
  \providecommand{\propositionname}{Proposition}
  \providecommand{\remarkname}{Remark}
\providecommand{\theoremname}{Theorem}

\begin{document}

\author{Amitai Netser Zernik}\thanks{Hebrew University, \url{amitai.zernik@mail.huji.ac.il}}

\title{Equivariant Open Gromov-Witten Theory of $\rr\pp^{2m}\hookrightarrow\cc\pp^{2m}$}
\begin{abstract}
We define equivariant open Gromov-Witten invariants for $\rr\pp^{2m}\hookrightarrow\cc\pp^{2m}$
as sums of integrals of equivariant forms over resolution spaces,
which are blowups of products of moduli spaces of stable disc-maps
modeled on trees. These invariants encode the quantum deformation
of the equivariant cohomology of $\rr\pp^{2m}$ by holomorphic discs
in $\cc\pp^{2m}$ and, for $m=1$, specialize to give Welschinger's
signed count of real rational planar curves in the non-equivariant
limit. 

This paper prepares the ground for \cite{fp-loc-OGW} in which we
prove a fixed-point formula computing these invariants. 
\end{abstract}

\maketitle
\tableofcontents{}

\section{\label{sec:Main results}Introduction}

This paper is concerned with equivariant invariants obtained from
stable maps of holomorphic discs to $\cc\pp^{2m}$ with boundary on
the Lagrangian $\rr\pp^{2m}$. The pair $\left(X,L\right)=\left(\cc\pp^{2m},\rr\pp^{2m}\right)$
is a real homogeneous variety (see \cite[Definition 2]{mod2hom},
and Example 3 \emph{ibid.}). This alleviates many of the difficulties
involved with the construction of the virtual fundamental chain. In
fact, the moduli space of stable disc maps $\overline{\mm}_{0,k,l}\left(X,L,\beta\right)$
is an orbifold with corners, constructed from the moduli of closed
genus zero maps, $\overline{\mm}_{0,k+2l}\left(X,\beta\right)$. It
carries a relative orientation by a result of Solomon \cite{JT}.
Roughly speaking, this means $\overline{\mm}_{0,k,l}\left(X,L,\beta\right)$
defines a smooth singular chain $\mm$ in $L^{k}\times X^{l}$. 

Of course in general, the chain $\mm$ is not a cycle, $\partial\mm\neq0$,
so capping cohomology classes against $\mm$ will \emph{not }produce
invariants. In particular, if we try to count disc-map configurations
subject to constraints, as we modify the constraints through a cobordism
some configurations may ``fall off the boundary'', making the count
ill-defined.

This type of problem can be overcome in various ways. Open invariants
are defined and computed in the works of Katz and Liu \cite{katz+liu},
Pandharipande, Solomon and Walcher \cite{disc-enumeration}, Georgieva
and Zinger \cite{penka+zinger} and Tehrani and Zinger \cite{tehrani+zinger}\footnote{The last two works study \emph{real }invariants, but these are closely
related to disc invariants using the reflection principle, cf. Remark
\ref{rem:welschinger}.}.

We will focus on an approach taken by Fukaya \cite{fukaya-superpotential},
making use of the recursive structure of the boundary and extracting
invariants from the $A_{\infty}$ algebra associated with $L\hookrightarrow X$.
In a similar vein, Solomon and Tukachinsky \cite{jake+sara} consider
the potential $\Phi\left(w\right)$ of a weak bounding cochain $w$
for the Fukaya $A_{\infty}$ algebra of $L$, satisfying certain conditions.
In case $\dim L$ is odd and certain obstructions vanish, they prove
that such a $w$ exists and $\Phi\left(w\right)$ is independent of
all choices. 

The key result in this paper is the construction of a complex $\left(\Omega_{\basic},D\right)$
of \emph{extended equivariant forms}. These extended forms behave
as if they're equivariant differential forms on a closed manifold\emph{
}equipped with a torus action: there's an integration map
\[
\int_{\basic}:\Omega_{\basic}\to\rr\left[\vec{\lambda}\right]
\]
satisfying Stokes' theorem, $\int_{\basic}D\upsilon=0$. This allows
us to define equivariant open Gromov-Witten invariants by
\[
I\left(k,\vec{l},\beta\right)=\int_{\basic}\omega\left(k,\vec{l},\beta\right)
\]
for some $\omega\left(k,\vec{l},\beta\right)\in\Omega_{\basic}$.
Stokes' theorem also plays a key role in \cite{fp-loc-OGW} where
we obtain a fixed-point formula, simplifying considerably the computation
of $\int_{\basic}\omega$ for any $\omega\in\Omega_{\basic}$ and
specializing to give a closed formula for $\int_{\basic}\omega\left(k,\vec{l},\beta\right)$.

Although seemingly quite different, this integration approach to open
Gromov-Witten theory is in fact based on \cite{jake+sara}. The open
Gromov-Witten invariants $\int_{\basic}\omega\left(k,\vec{l},\beta\right)$
are also the coefficients of the potential $\Phi\left(w\right)$ of
a kind of weak bounding cochain $w$. We can use the homological perturbation
lemma (see \cite{twA8}) to reduce the choices involved in the construction
of $w$ to selecting a single homotopy retraction operator. Moreover,
we can interpret $\Phi\left(w\right)$ as encoding the quantum deformation
class of the equivariant cohomology of $\rr\pp^{2m}$. The relationship
between the $A_{\infty}$ perspective and the integration perspective
is explained in \cite[\S 1.6]{fp-loc-OGW}.

It is worth noting that for $m=1$, the $\left(\cc\pp^{2},\rr\pp^{2}\right)$
invariants we define are an equivariant extension of Welschinger's
signed count \cite{welschinger} of real rational planar curves passing
through $k$ points in $\rr\pp^{2}$ and $l$ conjugation-invariant
pairs of points in $\cc\pp^{2}$, see Remark \ref{rem:welschinger}.

Let us now provide a more detailed overview of the contents of this
paper.

\subsection{\label{subsec:resolutions and stokes}Resolutions and Stokes' theorem.}

Fix a positive integer $m$ and let\linebreak{}
${\left(X,L\right)=\left(\cc\pp^{2m},\rr\pp^{2m}\right)}$. Let $\tb=U\left(1\right)^{m}$
denote the rank $m$ torus group. Let $b=\left(k,l,\beta\right)$
be a 3-tuple of non-negative integers such that $k+2l+3\beta\geq3$
and $k+\beta=1\mod2$, and denote 
\[
\tb_{b}^{r}=\tb\times\Sym\left(k\right)\times\Sym\left(l\right)\times\Sym\left(r\right),
\]
where $\Sym\left(n\right)$ is the symmetric group on $n$ elements.
In Sections \ref{sec:Resolutions} and \ref{sec:Extended-forms} we
construct, for $r\geq0$ a diagram of $\tb_{b}^{r}$-orbifolds with
corners 
\[
\widetilde{\mm}_{b}^{r}\xrightarrow{\bu_{b}^{r}}\mm_{b}^{r}\xrightarrow{\For_{b}^{r}}\check{\mm}_{b}^{r},
\]
where $\For_{b}^{r}$ is a kind of forgetful map and $\bu_{b}^{r}$
is obtained from a spherical blow up construction. Let us first explain
the construction of $\mm_{b}^{r}$ and the forgetful map. We set
$\mm_{b}^{0}=\overline{\mm}_{0,k,l}\left(X,L,\beta\right)$. We construct
$\mm_{b}^{1}$ from $\mm_{b}^{0}$ by replacing some of the fiber
products in 
\[
\partial\mm_{b}^{0}=\coprod\overline{\mm}_{0,k_{1}+1,l_{1}}\left(X,L,\beta_{1}\right)\times_{L}\overline{\mm}_{0,k_{2}+1,l_{2}}\left(X,L,\beta_{2}\right)
\]
by products. This can be thought of as a kind of resolution, allowing
the two components of the stable disc to move independently of one
another. By Leibniz' rule, $\partial\mm_{\basic}^{1}$ is also a disjoint
union of fibered products, and $\mm_{\basic}^{2}$ is obtained by
replacing some of these fiber products by products. The spaces $\mm_{b}^{r}$
for $r\geq3$ are defined in a similar, recursive, fashion, so that
$\mm_{b}^{r}$ resolves a clopen component of $\partial^{r}\overline{\mm}_{0,k,l}\left(X,L,\beta\right)$,
see (\ref{eq:bd^r as pullback}), and is a disjoint union of products
of moduli spaces indexed by certain labeled trees (see Lemma \ref{lem:labeled trees alternative def}).
The group $\tb_{b}^{r}$ acts by translating maps and relabeling markings,
making $\partial^{r}\overline{\mm}_{0,k,l}\left(X,L,\beta\right)\to\mm_{b}^{r}$
a $\Sym\left(r\right)$-equivariant map. The spaces $\check{\mm}_{b}^{r}$
are obtained from $\mm_{b}^{r}$ by forgetting one of each of the
$r$ pairs of marked points corresponding to the nodes. We denote
the forgetful map by $\For_{b}^{r}:\mm_{b}^{r}\to\check{\mm}_{b}^{r}$. 

Let us now sketch the blow up construction, $\widetilde{\mm}_{b}^{r}\to\mm_{\basic}^{r}$.
Let $\bu_{\Delta}:\widetilde{L\times L}\to L\times L$ denote the
spherical blow up of $L\times L$ along the diagonal $\Delta\subset L\times L$
(cf. $\S$\ref{subsec:Resolution-blow-ups}). Evaluation at the node
markings produces a map $\ed_{b}^{r}:\mm_{b}^{r}\to\left(L\times L\right)^{r}$,
which is transverse to $\Delta^{r}\subset\left(L\times L\right)^{r}$,
so we can define the blow up of $\mm_{b}^{r}$ using the cartesian
square
\[
\xymatrix{\widetilde{\mm}_{b}^{r}\ar[r]^{\bu_{b}^{r}}\ar[d]_{\widetilde{\ed}_{b}^{r}} & \mm_{b}^{r}\ar[d]^{\ed_{b}^{r}}\\
\left(\widetilde{L\times L}\right)^{r}\ar[r]_{\bu_{\Delta}^{r}} & \left(L\times L\right)^{r}
}
.
\]

In Definition \ref{def:extended form} we define an \emph{extended
form} $\omega$ for $b$ as a collection of $\tb$-equivariant forms 

\[
\omega=\left\{ \check{\omega}_{r}\in\Omega\left(\check{\mm}_{b}^{r};\check{\mathcal{E}}_{b}^{r}\otimes_{\zz}\rr\left[\vec{\lambda}\right]\right)^{\tb}\right\} _{r\geq0}\mbox{ for }\check{\mathcal{E}}_{b}^{r}:=\bigotimes_{i=1}^{k}\left(\evb{}_{i}^{r}\right)^{*}\mbox{Or}\left(TL\right),
\]
satisfying a boundary compatibility condition and $\Sym\left(r\right)$-invariance.
The set of extended forms is a complex, $\left(\Omega_{b},D\right)$,
with $D$ acting level-wise by the Cartan-Weil differential $D=d-\sum_{j=1}^{m}\lambda_{j}\iota_{\xi_{j}}$.
Here $\left\{ \xi_{j}\right\} _{j=1}^{m}$ are vector fields generating
the $\tb$-action and $\iota_{\xi_{j}}$ denotes contraction with
$\xi_{j}$. In Definition \ref{def:integration} we introduce an
$\rr\left[\vec{\lambda}\right]$-linear, $\Sym\left(k\right)\times\Sym\left(l\right)$-invariant
integration map 
\[
\int_{b}:\Omega_{b}\to\rr\left[\vec{\lambda}\right],
\]
by the finite sum of integrals 
\begin{equation}
\int_{b}\omega=\sum_{r\geq0}\frac{1}{r!}\int_{\widetilde{\mm}_{b}^{r}}\left(\For_{b}^{r}\circ\bu_{b}^{r}\right)^{*}\check{\omega}_{r}\cdot\left(\widetilde{\ed}_{b}^{r}\right)^{*}\Lambda^{\boxtimes r}.\label{eq:extended integral}
\end{equation}
Here $\Lambda\in\Omega\left(\widetilde{L\times L};\tilde{\mbox{pr}}_{2}^{*}\left(Or\left(TL\right)\right)\otimes_{\zz}\rr\left[\vec{\lambda}\right]\right)^{\tb}$
is an \emph{equivariant homotopy kernel}, see $\S$\ref{subsec:Equivariant-homotopy-kernel}.
The integrals are defined using local system isomorphisms 
\begin{equation}
\widetilde{\mathcal{J}}_{b}^{r}:\Or\left(T\widetilde{\mm}_{b}^{r}\right)\to\check{\mathcal{E}}_{b}^{r}\otimes\left(\left(\tilde{\pr}_{2}\right)^{\times r}\circ\widetilde{ed}_{b}^{r}\right)^{-1}\Or\left(TL\right)^{\otimes r},\label{eq:local system map}
\end{equation}
see (\ref{eq:tilde J}).

Clearly, for $r=0$ we have $\widetilde{\mm}_{b}^{0}=\mm_{b}^{0}=\check{\mm}_{b}^{0}$,
so the first summand in (\ref{eq:extended integral}) is just

\[
\int_{\overline{\mm}_{0,k,l}\left(X,L,\beta\right)}\omega_{0}.
\]
One may think of the summands for $r\geq1$ as corrections accounting
for the boundary and corners of $\overline{\mm}_{0,k,l}\left(X,L,\beta\right)$.

A central result of this paper is Stokes' theorem, Theorem \ref{thm:stokes' theorem},
which states that

\begin{equation}
\int_{b}D\upsilon=0\label{eq:stokes}
\end{equation}
for any $\upsilon\in\Omega_{b}$. In the next subsection we will use
this version of Stokes' theorem to define equivariant open Gromov-Witten
invariants. Equation (\ref{eq:stokes}) also plays an important role
in the proof of the fixed-point formula in \cite{fp-loc-OGW}.

\subsection{\label{subsec:equiv open GW invts}Equivariant open Gromov-Witten
invariants.}

For any pair of non-negative integers $k,\beta$ and tuple of non-negative
integers $\vec{l}=\left(l_{0},...,l_{2m}\right)\in\zz_{\geq0}^{2m+1}$,
the\emph{ equivariant open Gromov-Witten invariant} 
\[
I\left(k,\vec{l},\beta\right)\in\rr\left[\vec{\lambda}\right]
\]
is defined as follows. Let $l=l_{0}+\cdots+l_{2m}$. If $k+2l+3\beta<3$
or $k+\beta=0\mod2$, we set $I\left(k,\vec{l},\beta\right)=0$, so
suppose this is not the case.

We define 

\emph{
\[
I\left(k,\vec{l},\beta\right)=\int_{b}\omega,
\]
}where the extended form $\omega$ is constructed as follows. Fix
an $l$-tuple $\left(d_{1},...,d_{l}\right)\in\left\{ 0,...,2m\right\} ^{l}$
so that for all $0\leq d\leq2m$ we have 
\[
\#\left\{ 1\leq i\leq l|d_{i}=d\right\} =l_{d}.
\]
We define an extended form $\omega=\left\{ \check{\omega}_{r}\right\} \in\Omega_{b}$
by 
\[
\check{\omega}_{r}=\prod_{1\leq j\leq l}\left(\evi{}_{j}^{r}\right)^{*}\eta^{d_{j}}\,\cdot\,\prod_{1\leq i\leq k}\left(\evb{}_{i}^{r}\right)^{*}\rho_{0}.
\]
where $\rho_{0}\in\Omega\left(L;\mbox{Or}\left(TL\right)\otimes\rr\left[\vec{\lambda}\right]\right)^{\tb}$
is a $D$-closed form Poincaré dual to the (unique) fixed point $p_{0}\in L$,
and $\eta\in\Omega\left(X;\rr\left[\vec{\lambda}\right]\right)^{\tb}$
is a $D$-closed form representing the equivariant hyperplane class
$H$ (cf. (\ref{eq:cohomology of CP^2m})). Clearly, $D\omega=0$,
and by (\ref{eq:stokes}), this is independent of the representatives
$\eta$ and $\rho$. Independence of the tuple $\left(d_{1},...,d_{l}\right)$
representing $l$ follows from $\Sym\left(l\right)$-equivariance
of the evaluation maps.
\begin{rem}
The same argument shows that, in fact, 
\begin{equation}
I\left(k,\vec{l},\beta\right)=\int_{b}\prod_{1\leq j\leq l}\left(\evi{}_{j}^{r}\right)^{*}\eta_{j}\,\cdot\,\prod_{1\leq i\leq k}\left(\evb{}_{i}^{r}\right)^{*}\rho_{i}\label{eq:more freedom}
\end{equation}
for any $\left\{ \rho_{i}\right\} ,\left\{ \eta_{j}\right\} $ with
$\left[\rho_{i}\right]=\mbox{pt}$ for $1\leq i\leq k$ and $\left[\eta_{j}\right]=H{}^{d_{j}}$
for $1\leq j\leq l$.
\end{rem}
\begin{rem}
\cite[Theorem 27]{fp-loc-OGW} can be restated as saying that for
every $\omega\in\Omega_{b}$ with $D\omega=0$ we have 
\[
\int_{b}\omega=\sum_{r,C}\operatorname{Cont}_{C}\left(\omega\right)
\]
where for $r\geq0$, $C\subset E^{\tb}$ ranges over the connected
components of $E^{\tb}$, the $\tb$-fixed-points of a clopen component
(or disjoint union of connected components) $E\subset\check{\mm}_{b}^{r}$,
and where $\operatorname{Cont}_{C}\left(\omega\right)$ is given by
a certain integral over $C$ which depends only on the cohomology
class of $\omega\in\Omega_{b}$. In this sense, the invariants can
be refined. Moreover, the explicit form of $\operatorname{Cont}_{C}\left(\omega\right)$
shows that it does not depend on the choice of equivariant homotopy
kernel $\Lambda$. It may be useful to find a direct proof of the
invariance of $\int_{b}\omega$ on the choice of $\Lambda$ that does
not require the fixed-point localization argument.
\end{rem}

We have 
\begin{multline}
\deg I\left(k,\vec{l},\beta\right)=\deg\omega-\dim\overline{\mm}_{0,k,l}\left(\beta\right)=\\
=\sum_{j=0}^{2m}\left(2j-1\right)l_{j}+\left(2m-1\right)\cdot k-\left(2m+1\right)\left(\beta+1\right)+4.\label{eq:invariant degree}
\end{multline}

\begin{rem}
\label{rem:geometric invariants}In Remark \ref{rem:non-equivariant}
we note that the same construction can be used to define \emph{non}-equivariant
extended forms and their integrals, which satisfy $\int d\upsilon=0$.
In particular, open Gromov-Witten invariants with $\deg I\left(k,\vec{l},\beta\right)=0$
can be computed using any $d$-closed forms $\eta_{i}$, $\rho_{j}$
representing their cohomology classes. This can be used to offer a
more geometric interpretation of the invariants, as follows.

Choose for $1\leq i\leq l$ a submanifold $W_{i}\subset X$ Poincaré
dual to $H^{d_{i}}$ and $k$ points $\left\{ z_{j}\right\} _{j=1}^{k}\subset L$
such that ${\left(\evb^{r}\times\evi^{r}\right)\circ\For_{b}^{r}\circ\bu_{b}^{r}}$
is b-transverse\footnote{See Definition \ref{def:properties of orbi-maps}. This means that
${\left(\evb^{r}\times\evi^{r}\right)\circ\For_{b}^{r}\circ\bu_{b}^{r}}\circ\iota_{\widetilde{\mm}_{b}^{r}}^{\partial^{c}}$
is transverse to $W_{1}\times\cdots\times W_{l}\times\left\{ z_{1}\right\} \times\cdots\times\left\{ z_{k}\right\} $
for all $c\geq0$, where $\iota_{\widetilde{\mm}_{b}^{r}}^{\partial^{c}}:\partial^{c}\widetilde{\mm}_{b}^{r}\to\widetilde{\mm}_{b}^{r}$
is the structure map.} to $W_{1}\times\cdots\times W_{l}\times\left\{ z_{1}\right\} \times\cdots\times\left\{ z_{k}\right\} \subset X^{l}\times L^{k}$
for all $r\geq0$, so the inverse image $\mathcal{W}^{r}\subset\widetilde{\mm}_{b}^{r}$
is an embedded suborbifold. We can then take $\eta_{i}$ and $\rho_{j}$
in (\ref{eq:more freedom}) to be Thom forms for $W_{i}$, $z_{j}$.
Shrinking the support of these forms we obtain, in the limit, 
\[
I\left(k,\vec{l},\beta\right)=\sum_{r\geq0}\frac{1}{r!}\int_{\mathcal{W}^{r}}\left(\widetilde{\ed}_{\basic}^{r}\right)^{*}\Lambda^{\boxtimes r}.
\]
In particular, the $r=0$ term is just $\int_{\mathcal{W}^{0}}1$:
a signed, isotropy-weighted count of the discs satisfying the constraints. 
\end{rem}
\begin{rem}
\label{rem:welschinger}When $m=1$, the invariants with $\deg I\left(k,\vec{l},\beta\right)=0$
for $\left(\cc\pp^{2},\rr\pp^{2}\right)$ overlap with those defined
in \cite{JT}. Indeed, the sign-of-conjugation argument used there
to cancel the contributions of the boundary, can be used to show that
the $r>0$ terms in (\ref{eq:extended integral}) vanish and $I\left(k,\vec{l},\beta\right)$
reduces to a disc count (see Remark \ref{rem:geometric invariants}).
As in \cite{JT}, Schwarz reflection associates to every stable disc-map
a real rational planar curve passing through $k$ real points and
$l$ conjugation-invariant pairs of complex points, and this can be
used to identify disc counts with twice Welschinger's signed count
of such curves \cite{welschinger}. This was verified with a computer
for all invariants of degree $\leq6$, using the fixed-point formula
\cite{fp-loc-OGW}. 
\end{rem}
\begin{rem}
Invariants with $\deg I\left(k,\vec{l},\beta\right)<0$ must vanish,
so the fixed-point formula \cite{fp-loc-OGW} produces non-trivial
relations involving descendent integrals on discs. See also Remark
\ref{rem:Pin factor verification} for a simple example of this.
\end{rem}

\textbf{Acknowledgments.} I am deeply grateful to my teacher, Jake
Solomon for contributing important ideas, advice, and encouragement.
I\textquoteright d also like to thank Mohammed Abouzaid, Rahul Pandharipande,
Ran Tessler and Lior Yanovski for engaging conversations and useful
suggestions. I was partially supported by ERC starting grant 337560,
ISF Grant 1747/13 and NSF grant DMS-1128155.

\section{\label{sec:Resolutions}Resolutions of Moduli Spaces}

Throughout the paper, $m$ will be some fixed positive integer, and
we will consider the open Gromov-Witten theory of $\left(X,L\right)=\left(\cc\pp^{2m},\rr\pp^{2m}\right)$. 

\subsection{The pair $\left(\cc\pp^{2m},\rr\pp^{2m}\right)$, torus actions and
moduli spaces.}

We begin by introducing notation and reviewing some results. We assume
the reader is familiar with the notion of a real homogeneous pair
\cite[Definition 1]{mod2hom}. 

Let $G_{X}=U\left(2m+1\right)$ denote the unitary group, with $\tb^{\cc}=U\left(1\right)^{2m+1}<U\left(2m+1\right)$
the subgroup of diagonal matrices. Projectivizing the standard action
on $V=\cc^{2m+1}$ defines a transitive group action 
\[
\alpha_{U,X}:U\left(2m+1\right)\times X\to X.
\]
Restricting to $\tb^{\cc}$ and decomposing $V=V_{0}\oplus\cdots\oplus V_{2m}$
into complex irreducible representations we find 
\[
X=\cc\pp^{2m}=\pp_{\cc}\left(V_{0}\oplus V_{1}\oplus\cdots\oplus V_{2m}\right).
\]
where $\left(u_{0},...,u_{2m}\right)\in\tb^{\cc}$ acts on $V_{i}$
by $z\mapsto u_{i}\cdot z$. 

$\mathbb{T}^{\cc}$-equivariant cohomology is defined over the ring
\[
H_{\mathbb{T}^{\cc}}^{\bullet}=H_{\mathbb{T}^{\cc}}^{\bullet}\left(\mbox{pt}\right)=\rr\left[\alpha_{0},...,\alpha_{2m}\right],\;\deg\alpha_{i}=2.
\]
The total space of the tautological line bundle $\tau$ on $\pp_{\cc}\left(V\right)$
is the blow up of $V$, so there's a natural lift of the $\tb^{\cc}$
action to $\tau$. Let $H=-c_{1}^{\tb^{\cc}}\left(\tau\right)\in H_{\mathbb{T}^{\cc}}^{2}\left(X\right)$
denote minus the equivariant Chern class of $\tau$. $H$ is an equivariant
extension of the hyperplane class. We have 
\begin{equation}
H_{\tb^{\cc}}^{\bullet}\left(X\right)=\rr\left[H,\alpha_{0},....,\alpha_{2m}\right]/\left(\prod_{i=0}^{2m}\left(H-\alpha_{i}\right)\right).\label{eq:cohomology of CP^2m}
\end{equation}
In particular, $H_{\tb^{\cc}}^{\bullet}\left(X\right)$ is generated
as an $H_{\mathbb{T}^{\cc}}^{\bullet}$-module by $H^{0},H^{1},...,H^{2m}$.

We now introduce a conjugation action. Let $c_{V}:V\to V$ be the
involution given by 
\begin{equation}
\left(z_{0},...,z_{2m}\right)\mapsto\left(\overline{z}_{0},\overline{z}_{2m},\overline{z}_{2m-1},...,\overline{z}_{1}\right),\label{eq:affine conjugation}
\end{equation}
let $c_{X}:X\to X$ denote its projectivization, and let ${c_{G}:U\left(2m+1\right)\to U\left(2m+1\right)}$
be defined by 
\begin{equation}
g\mapsto c_{V}\circ g\circ c_{V}^{-1}.\label{eq:G^C conjugation}
\end{equation}
Clearly $c_{G}\left(\tb^{\cc}\right)=\tb^{\cc}$ and $O\left(2m+1\right)$
and $\tb$ are the $c_{G}$-fixed subgroups of $U\left(2m+1\right)$
and $\tb^{\cc}$, respectively, where $\tb\simeq U\left(1\right)^{m}$
is the image of 
\[
U\left(1\right)^{m}\ni\left(u_{1},...,u_{m}\right)\mapsto\mbox{diag}\left(1,u_{1},....,u_{m},\overline{u}_{m},...,\overline{u}_{1}\right)\in\tb<\tb^{\cc}<U\left(2m+1\right).
\]
The monomorphism $\tb\to\tb^{\cc}$ corresponds to the map
\begin{equation}
H_{\tb^{\cc}}^{\bullet}=\rr\left[\alpha_{0},...,\alpha_{2m}\right]\overset{\rho_{\tb}}{\longrightarrow}H_{\tb}^{\bullet}=\rr\left[\lambda_{1},...,\lambda_{m}\right]\label{eq:weight loss}
\end{equation}
given by $\alpha_{0}\mapsto0$ and, for $1\leq i\leq m$, $\alpha_{i}\mapsto\lambda_{i}$
and $\alpha_{2m+1-i}\mapsto-\lambda_{i}$. Let $W_{0}=\rr$ denote
the trivial representation of $\tb$, and for $1\leq i\leq m$ let
$W_{i}=\rr^{2}$ denote the real representations of $\tb$, on which
$\left(e^{\sqrt{-1}t_{1}},...,e^{\sqrt{-1}t_{m}}\right)$ acts by
\[
\begin{pmatrix}\cos t_{i} & -\sin t_{i}\\
\sin t_{i} & \cos t_{i}
\end{pmatrix}.
\]
As $\zz/2\times\tb$ representations we have 
\[
V_{0}\simeq\cc\otimes W_{0}
\]
and for $1\leq i\leq m$,
\[
V_{i}\oplus V_{2m+1-i}\simeq\cc\otimes W_{i},
\]
where the $\zz/2$ action is the usual conjugation action on the $\cc$
factor. 

We find that
\[
\left(X,\omega,J,G_{X}=U\left(2m+1\right),\alpha_{U,X},c_{G},c_{X}\right)
\]
is a real homogeneous variety (cf. \cite[Definition 1]{mod2hom}),
where $\omega$ is the Fubini-Study symplectic form and $J$ is the
standard complex structure. Let $\left(X,L=\rr\pp^{2m}\right)$ be
the associated real homogeneous pair. As usual, the induced action
of $G_{X}^{\zz/2}=O\left(2m+1\right)$ on $L$ is transitive. It restricts
to a $\tb$-action on $L$ specified by the equivariant identification
\[
L=\pp_{\rr}\left(W_{0}\oplus W_{1}\oplus\cdots\oplus W_{m}\right).
\]
Recall the map
\begin{equation}
H_{2}\left(X,L\right)\to H_{2}\left(X\right)\label{eq:doubling in homology}
\end{equation}
sends a singular chain $\sigma\in C_{2}\left(X\right)$ with $\partial\sigma\in C_{1}\left(L\right)$
to the class represented by the cycle $c_{X*}\sigma+\sigma$. We identify
$H_{2}\left(X\right)=\zz$ using the complex structure and fix an
isomorphism $H_{2}\left(X,L\right)\simeq\zz$ so that $H_{2}\left(X\right)\to H_{2}\left(X,L\right)$
corresponds to multiplication by $+2$. The map (\ref{eq:doubling in homology})
then becomes $\id_{\zz}$. We see that an integer $\beta\in\zz$ can
either represent an element of $H_{2}\left(X,L\right)$ or its image
in $H_{2}\left(X\right)$ under (\ref{eq:doubling in homology}),
depending on the context. 

Let $\left(k,l,\beta\right)$ be non-negative integers with $k+2l+3\beta\geq3$,
and write 
\[
G_{k,l}=O\left(2m+1\right)\times\Sym\left(k\right)\times\Sym\left(l\right)
\]
for the product of the orthogonal group with the symmetric groups
on $k$ and $l$ elements. By \cite[Theorem 2]{mod2hom} the moduli
space
\[
\overline{\mm}_{0,k,l}\left(X,L,\beta\right)
\]
is a $G_{k,l}$-orbifold with corners admitting a $G_{k,l}$-equivariant
map to the moduli space $\overline{\mm}_{0,k+2l}\left(X,\beta\right)$,
and there's a $G_{k,l}$-equivariant evaluation map
\[
\overline{\mm}_{0,k,l}\left(X,L,\beta\right)\to L^{k}\times X^{l}
\]
induced from the closed evaluation map $\overline{\mm}_{0,k+2l}\left(\beta\right)\to X^{k+2l}$.

We will often use general sets as labels. For instance, we may consider
the moduli spaces $\overline{\mm}_{0,\kf,\lf}\left(\beta\right)$
where $\kf,\lf$ are finite disjoint sets.

\subsection{Moduli specifications and an overview.}

Let $\nn=\left\{ 1,2,...\right\} $. For any $S\subset\nn$, we denote
$\sstar'_{S}=\left\{ \sstar'_{i}\right\} _{i\in S}$ and similarly
for $\sstar''_{S}$ and $\star'_{S},\star''_{S},\star'''_{S},...$
Hollow stars will be used to denote markings related to boundary nodes,
and solid stars will be used for interior nodes. 
\begin{defn}
\label{def:moduli specifications}A \emph{pre-moduli specification}
$\mathfrak{b}$ is a 3-tuple $\left(\kf,\lf,\beta\right)$ where 

\begin{itemize}
\item $\kf\subset\nn\coprod\sstar''_{\nn}$ and $\lf\subset\nn$ are finite
subsets. Elements of $\kf$ and of $\lf$ are called \emph{orienting
labels }and \emph{interior labels, }respectively.
\item $\beta$ is a non-negative integer, \emph{the degree}, which we think
of as an element of $H_{2}\left(X,L\right)$.
\end{itemize}
A \emph{basic moduli specification }is a pre-moduli specification
$\mathfrak{b}=\left(\kf,\lf,\beta\right)$ that is 

\begin{itemize}
\item \emph{stable,} meaning $k+2l+3\beta\geq3$; henceforth we use standard
Roman letters to denote the sizes of sets labeled by the corresponding
Serif letters, so $k=\left|\kf\right|$ and $l=\left|\lf\right|$.
\item \emph{orientable}, meaning 
\begin{equation}
k+\beta=1\mod2.\label{eq:orientability of moduli specification}
\end{equation}
\end{itemize}
A \emph{moduli specification }$\sfr$ is a pair $\sfr=\left(\mathfrak{b},\sigma\right)$
where $\mathfrak{b}=\left(\kf,\lf,\beta\right)$ is an orientable
pre-moduli specification and $\sigma\subset\sstar'_{\nn}$ is a finite
subset such that $k+\left|\sigma\right|+2l+3\beta\geq3$. We call
$\sigma$ the \emph{superfluous (boundary) labels,} and $\tilde{\kf}=\kf\coprod\sigma$
the \emph{boundary labels}. 

A\emph{ }moduli specification\emph{ }$\sfr=\left(\mathfrak{b},\sigma\right)$
is called \emph{sturdy }if $\mathfrak{b}$ is stable and \emph{wobbly
}otherwise.
\end{defn}
Let $\sfr=\left(\mathfrak{b},\sigma\right)$ be a moduli specification.
If $\sfr$ is sturdy, then $\mathfrak{b}$ is a basic moduli specification;
if it is wobbly, it is necessarily of the form 
\begin{equation}
\left(\left(\kf,\emptyset,0\right),\sigma\right)\text{ with \ensuremath{\left|\kf\right|}=1 and \ensuremath{\left|\sigma\right|\geq}2.}\label{eq:wobbly spec}
\end{equation}
Either way, the \emph{combined moduli specification }$\tilde{\sfr}:=\left(\tilde{\kf},\lf,\beta\right)$
is a basic moduli specification. 

A 3-tuple of non-negative integers $b=\left(k,l,\beta\right)$ with
$k+2l+3\beta\geq3$ and $k+\beta=1\mod2$ may be used in place of
a basic moduli specification, taking $\basic=\left(\left[k\right],\left[l\right],\beta\right)$
where we denote $\left[k\right]=\left\{ 1,2,...,k\right\} $. 

The following proposition collects the main properties of resolutions
that are proved in this section, and which we will need when we discuss
extended forms and integration in Section \ref{sec:Extended-forms}.
We refer the reader to the Appendix, $\S$\ref{sec:appendix}, for
the definition of the category of orbifolds with corners and related
notions. 

Let $\basic=\left(\kf,\lf,\beta\right)$ be a basic moduli specification.
Let $\tb_{\basic}=\tb\times\Sym\left(\kf\right)\times\Sym\left(\lf\right)$
and $\tb_{\basic}^{r}=\tb_{\basic}\times\Sym\left(r\right)$. 
\begin{prop}
\label{thm:resolution summary}For every non-negative integer $r\geq0$
(the number of resolved nodes) and subset $S\subset\nn$ (the subset
of node tails we do \emph{not }forget) there exist 

\begin{itemize}
\item $\tb_{\basic}^{r}$-orbifolds with corners $\acute{\mm}_{\basic}^{r,S}$.
We set $\check{\mm}_{\basic}^{r}=\acute{\mm}_{\basic}^{r,\emptyset}$
and $\mm_{\basic}^{r}=\acute{\mm}_{\basic}^{r,\nn}$.
\end{itemize}

\begin{itemize}
\item $\tb_{\basic}^{r}$-equivariant b-fibrations (\emph{the forgetful
maps})
\[
\acute{\For}_{\basic}^{r,S}:\mm_{\basic}^{r}\to\acute{\mm}_{\basic}^{r,S}
\]
and 
\[
\grave{\For}_{\basic}^{r,S}:\acute{\mm}_{\basic}^{r,S}\to\check{\mm}_{\basic}^{r}.
\]
We set 
\[
\For_{\basic}^{r}=\acute{\For}_{\basic}^{r,\emptyset}:\mm_{\basic}^{r}\to\check{\mm}_{\basic}^{r}.
\]
$\For_{\basic}^{r}$ induces a $\tb_{\basic}^{r}$-equivariant decomposition
of the boundary, denoted 
\[
\partial\mm_{\basic}^{r}=\partial_{+}\mm_{\basic}^{r}\coprod\partial_{-}\mm_{\basic}^{r},
\]
see (\ref{eq:b-normal bdry decomposition}),
\item boundary evaluation maps
\begin{eqnarray*}
\acute{\ev}_{x}^{\basic,r,S}:\acute{\mm}_{\basic}^{r,S}\to L & \mbox{ for } & x\in\kf\coprod\left\{ \sstar''_{i}\right\} _{i=1}^{r}\coprod\left\{ \sstar'_{i}\right\} _{i\in\left[r\right]\cap S},\mbox{ and}
\end{eqnarray*}
\item interior evaluation maps
\begin{eqnarray*}
\acute{\ev}_{x}^{\basic,r,S}:\acute{\mm}_{\basic}^{r,S}\to X & \mbox{ for } & x\in\lf.
\end{eqnarray*}
\end{itemize}
\begin{itemize}
\item a $\tb_{\basic}^{r}$-equivariant involution $\inv_{\basic}^{r}:\partial_{+}\mm_{\basic}^{r}\to\partial_{+}\mm_{\basic}^{r}$,
and
\item local system maps 
\[
\ff_{\basic}^{r}:\Or\left(T\mm_{\basic}^{r}\right)\to\Or\left(T\check{\mm}_{\basic}^{r}\right)
\]
lying over $\For_{\basic}^{r}$ and 
\[
\check{\mathcal{J}}_{\basic}^{r}:Or\left(T\check{\mm}_{\basic}^{r}\right)\to Or\left(TL\right)^{\boxtimes\left(k+r\right)}
\]
lying over ${\prod_{x\in\kf\coprod\left(\sstar''_{i}\right)_{i=1}^{r}}\check{\ev}{}_{x}^{\basic,r}:\check{\mm}_{\basic}^{r}\to L^{k+r}}$.
We set 
\[
\jc_{\basic}^{r}:=\check{\jc}_{\basic}^{r}\circ\ff_{\basic}^{r}.
\]
\end{itemize}
These satisfy the following properties.

(a) $\mm_{\basic}^{0}=\overline{\mm}_{0,\kf,\lf}\left(\beta\right)$
and $\mm_{\basic}^{r}=\emptyset$ \textup{for sufficiently large $r$.} 

(b) there's a cartesian square 
\[
\xymatrix{\partial_{-}\mm_{\basic}^{r}\ar[d]\ar[r]^{g_{\basic}^{r+1}} & \mm_{\basic}^{r+1}\ar[d]^{\ev{}_{\sstar_{r+1}'}^{\basic,r+1}\times\ev{}_{\sstar_{r+1}''}^{\basic,r+1}}\\
L\ar[r]_{\Delta_{L}} & L\times L
}
\]
where the right and bottom maps are b-transverse (see Definition \ref{def:properties of orbi-maps}).
This induces a local system map 
\[
\mathcal{G}_{\basic}^{r+1}:\Or\left(T\partial_{-}\mm_{\basic}^{r}\right)\to\Or\left(T\mm_{\basic}^{r+1}\right)\otimes\left(ev_{\sstar''_{r+1}}^{\basic,r+1}\right)^{-1}\Or\left(TL\right)
\]
lying over $g_{\basic}^{r+1}$.

(c)
\[
\mathcal{J}_{\basic}^{r}\circ\left(\iota_{\mm_{\basic}^{r}}^{\partial}|_{\partial_{+}\mm_{\basic}^{r}}\right)\circ\Or\left(d\inv_{\basic}^{r}\right)=\left(-1\right)\mathcal{J}_{\basic}^{r}\circ\left(\iota_{\mm_{\basic}^{r}}^{\partial}|_{\partial_{+}\mm_{\basic}^{r}}\right)
\]

(d) For all $x\in\underline{\kf}\coprod\lf\coprod\left\{ \sstar'_{i},\sstar''_{i}\right\} _{i=1}^{r}$
we have 
\[
\ev_{x}^{\basic,r+1}\circ g_{\basic}^{r+1}=\ev_{x}^{\basic,r}\circ i_{\mm_{\basic}^{r}}^{\partial}
\]

(e) Let ${\left(\jc_{\basic}^{r+1}\right)^{\to}:\mm_{\basic}^{r+1}\otimes\left(ev_{\sstar''_{r+1}}^{\basic,r+1}\right)^{-1}Or\left(TL\right)^{\vee}\to Or\left(TL\right)^{\boxtimes\left(\left|k\right|+r\right)}}$
be the local system map derived from ${\mathcal{J}_{\basic}^{r+1}:\mm_{\basic}^{r+1}\to Or\left(TL\right)^{\boxtimes\left(\left|k\right|+\left(r+1\right)\right)}}$.
Then 
\[
\left(\jc_{\basic}^{r+1}\right)^{\to}\circ\mathcal{G}_{\basic}^{r+1}=\mathcal{J}_{\basic}^{r}\circ\iota_{\mm_{\basic}^{r}}^{\partial}|_{\partial_{-}\mm_{\basic}^{r}}.
\]

(f) The maps 
\[
\prod_{x\in\kf\coprod\lf\coprod\left\{ \sstar'_{i},\sstar''_{i}\right\} _{i=1}^{r}}\ev_{x}:\mm_{\basic}^{r}\to L^{\kf}\times X^{\lf}\times\left(L\times L\right)^{r}
\]
and 
\[
\prod_{x\in\kf\coprod\lf\coprod\left\{ \sstar''_{i}\right\} _{i=1}^{r}}\check{\ev}_{x}:\check{\mm}_{\basic}^{r}\to L^{\kf}\times X^{\lf}\times L^{r}
\]
are $\tb_{\basic}^{r}$-equivariant (see Remark \ref{rem:legal technicalities}).
In both cases, $\tb$ acts diagonally on the codomain and $\Sym\left(\kf\right)\times\Sym\left(\lf\right)\times\Sym\left(r\right)$
acts by shuffling factors. 

(g) $\mathcal{J}_{\basic}^{r}$ is $\tb\times\Sym\left(\kf\right)\times\Sym\left(\lf\right)<\tb_{\basic}^{r}$
invariant, but the action of $\tau\in\Sym\left(r\right)<\tb_{\basic}^{r}$
involves a sign: if we let $\tau.^{\mm_{\basic}^{r}}:\mm_{\basic}^{r}\to\mm_{\basic}^{r}$
denote the diffeomorphism induced by the action of $\tau$ then
\begin{equation}
\mathcal{J}_{\basic}^{r}\circ\Or\left(d\tau.^{\mm_{\basic}^{r}}\right)=\operatorname{sgn}\left(\tau\right)\cdot\mathcal{J}_{\basic}^{r},\label{eq:J commutes with Sym(r)}
\end{equation}
where $\operatorname{sgn}\left(\tau\right)\in\left\{ \pm1\right\} $
is the sign of $\tau$.

(h) We have $\acute{\ev}_{x}^{\basic,r,\nn}=\acute{\ev}_{x}^{\basic,r,S}\circ\acute{\For}_{\basic}^{r,S}$
and $\acute{\ev}_{y}^{\basic,r,S}=\acute{\ev}_{y}^{\basic,r,\emptyset}\circ\grave{\For}_{\basic}^{r,S}$
whenever both sides are defined.
\end{prop}
\begin{rem}
\label{rem:legal technicalities}Let $\xx$ be an orbifold and $M$
be a manifold. The maps $\xx\to M$ are equivalent to a set, so it
makes sense to say that two maps $f_{1},f_{2}:\xx\to M$ are equal.
If $\xx,M$ are equipped with an action of a group $H$, then the
$H$-equivariant maps form a subset of the set of all maps, so we
may treat $H$-equivariance as a property (in contrast, if $\yy$
is a general $H$-orbifold the forgetful functor between $H$-equivariant
maps $\xx\to\yy$ and ordinary maps is not full and faithful). Similar
remarks hold for maps of local systems whose codomain is a local system
over a manifold.
\end{rem}

\subsection{\label{subsec:wobbly involution}Wobbly boundary involution.}

\subsubsection{\label{subsec:b-normal boundary decomp}Boundary decomposition by
a b-normal map}

Following Joyce \cite{joyce-generalized}, we can use a b-normal map
$\xx\xrightarrow{f}\yy$ to decompose the boundary $\partial\xx$
into a horizontal $\partial_{+}^{f}\xx$ and a vertical $\partial_{-}^{f}\xx$
part.

As we discuss in the appendix, if $\xx\xrightarrow{f}\yy$ is a
smooth map of orbifolds with corners, we have an induced interior
map of l-orbifolds
\[
C\left(\xx\right):=\coprod_{k\geq0}C_{k}\left(\xx\right)\xrightarrow{C\left(f\right)}C\left(\yy\right)=\coprod_{l\geq0}C_{l}\left(\yy\right).
\]
It follows from \cite[Proposition 2.11]{joyce-generalized} that a
smooth map $f:\xx\to\yy$ is b-normal\emph{ }if and only if 
\[
C\left(f\right)\left(C_{r}\left(\xx\right)\right)\subset\coprod_{r'\leq r}C_{r'}\left(\yy\right)
\]
for all $r\geq0$. 

\begin{defn}
(cf. \cite{joyce-generalized}, below Definition 4.2). Let $f:\xx\to\yy$
be a b-normal map. We denote 
\[
\partial_{+}^{f}\xx:=C\left(f\right)|_{C_{1}\left(\xx\right)}^{-1}\left(C_{0}\left(\yy\right)\right)
\]
and 
\[
\partial_{-}^{f}\xx:=C\left(f\right)|_{C_{1}\left(\xx\right)}^{-1}\left(C_{1}\left(\yy\right)\right)
\]
so the boundary of $\xx$ is a disjoint union 
\begin{equation}
\partial\xx=C_{1}\left(\xx\right)=\partial_{+}^{f}\xx\coprod\partial_{-}^{f}\yy,\label{eq:b-normal bdry decomposition}
\end{equation}
and $C\left(f\right)|_{C_{1}\left(\xx\right)}$ is given by the b-normal
maps 

\[
f_{+}:\partial_{+}^{f}\xx\to\yy
\]
and
\[
f_{-}:\partial_{-}^{f}\xx\to\partial\yy.
\]
\end{defn}

\begin{lem}
\label{lem:composition decomposition}Let $h:\xx\to\zc$ denote the
composition $\xx\overset{f}{\longrightarrow}\yy\overset{g}{\longrightarrow}\zc$
of two b-normal maps. Then 

(a) $h$ is b-normal.

(b) we have 
\[
\partial_{+}^{h}\xx=\partial_{+}^{f}\xx\coprod f_{-}^{-1}\left(\partial_{+}^{g}\yy\right)\mbox{ and }\partial_{-}^{h}\xx=f_{-}^{-1}\left(\partial_{-}^{g}\xx\right).
\]

(c) we have
\[
h_{-}=g_{-}\circ f_{-}\mbox{, }h_{+}|_{\partial_{+}^{f}\xx}=g\circ f_{+}\mbox{ and }h_{+}|_{f_{-}^{-1}\left(\partial_{+}^{g}\yy\right)}=\left(g_{+}\circ f_{-}\right).
\]
\end{lem}
\begin{proof}
This follows from $C\left(h\right)=C\left(g\right)\circ C\left(f\right)$
(see \cite[Definition 2.10]{joyce-generalized}).
\end{proof}
More generally, given a chain 
\[
\xx_{r}\overbrace{\overset{f_{r}}{\longrightarrow}\cdots\overset{f_{j+1}}{\longrightarrow}}^{f_{>j}}\xx_{j}\underbrace{\overset{f_{j}}{\longrightarrow}\cdots\overset{f_{1}}{\longrightarrow}}_{f_{\leq j}}\xx_{0}
\]
of b-normal maps we define $f_{>j}:\xx_{r}\to\xx_{j}$ and $f_{\leq j}:\xx_{j}\to\xx_{0}$
for $j=1,...,r$ by the indicated compositions, and we have
\begin{equation}
\partial_{+}^{f_{\leq r}}\xx_{r}=\coprod_{j=1}^{r}\partial_{\leq j}\xx_{r}:=\coprod_{j=1}^{r}\left(f_{>j}\right)_{-}^{-1}\left(\partial_{+}^{f_{j}}\xx_{j}\right).\label{eq:multi composition decomposition}
\end{equation}

\subsubsection{\label{subsec:The-forgetful-map}The forgetful map boundary decomposition}

For a sturdy moduli specification $\sfr=\left(\basic,\sigma\right)=\left(\left(\kf,\lf,\beta\right),\sigma\right)$
and $S\subset\nn$, we define 
\[
\acute{\mm}_{\sfr}^{S}=\mm_{\left(\kf\coprod\left(\sigma\cap\sstar'_{S}\right),\lf,\beta\right)}.
\]
We will mostly be interested in $\mm_{\sfr}:=\acute{\mm}_{\sfr}^{\nn}$
and $\check{\mm}_{\sfr}:=\acute{\mm}_{\sfr}^{\emptyset}$. 

Let $\For_{\sfr}:\mm_{\sfr}\to\check{\mm}_{\sfr}$ be the map that
forgets the markings $\sstar'_{\sigma}$. For any $S\subset\nn$ we
have a decomposition 
\begin{equation}
\mm_{\sfr}\xrightarrow{\acute{\For}_{\sfr}^{S}}\acute{\mm}_{\sfr}^{S}\xrightarrow{\grave{\For}_{\sfr}^{S}}\check{\mm}_{\sfr}\label{eq:For decomposition}
\end{equation}
where $\acute{\For}_{\sfr}^{S}$ (respectively, $\grave{\For}_{\sfr}^{S}$)
is the map that forgets the markings $\sigma\backslash\sstar'_{S}$
(resp. $\sigma\cap\sstar'_{S}$). 
\begin{lem}
The maps $\For_{\sfr},\acute{\For}_{\sfr}^{S},\grave{\For}_{\sfr}^{S}$
are well-defined b-fibrations. 
\end{lem}
\begin{proof}
\cite[Lemma 8]{mod2hom} says that the map forgetting a single boundary
marked point is a b-fibration, whenever its codomain does not contain
any E-type nodes. Since each of $\For_{\sfr},\acute{\For}_{\sfr}^{S},\grave{\For}_{\sfr}^{S}$
can be written as a composition of maps forgetting a single boundary
marking, and b-fibrations are closed under composition, we reduce
to showing that for any moduli specification $\sfr=\left(\left(\kf,\lf,\beta\right),\sigma\right)$
and set $S$ the configurations parameterized by the space $\acute{\mm}_{\sfr}^{S}$
do not have any E-type nodes. 

To show this, note that the map ${D:=\id_{H_{2}\left(X\right)}+\left(c_{X}\right)_{*}:H_{2}\left(X\right)\to H_{2}\left(X\right)}$
is just multiplication by 2, so (by the orientability assumption,
$\left|\kf\right|+\beta=1\mod2$), we either have $\kf\coprod\left(\sigma\cap\sstar'_{S}\right)\neq\emptyset$
or $\beta$ is not in the image of $D$, which implies there are no
E-nodes by \cite[Remark 7]{mod2hom}.
\end{proof}
\begin{rem}
The forgetful map is \emph{not }a submersion; it is not even strongly
smooth.
\end{rem}
We write 
\[
\partial\mm_{\sfr}=\partial_{+}\mm_{\sfr}\coprod\partial_{-}\mm_{\sfr}
\]
with $\partial_{\pm}\mm_{\sfr}=\partial_{\pm}^{For_{\sfr}}\mm_{\sfr}$.
We call $\partial_{-}\mm_{\sfr}$ the \emph{sturdy boundary} and $\partial_{+}\mm_{\sfr}$
the \emph{wobbly boundary}. The boundary also decomposes as 

\begin{equation}
\partial\mm_{\sfr}=\coprod\mm_{\sfr'}\underset{\ev{}_{\sstar'_{r}}^{\sfr'}\qquad\ev{}_{\sstar''_{r}}^{\sfr''}}{\times}\mm_{\sfr''}\label{eq:boundary decomposition}
\end{equation}
where the disjoint union is taken over all pairs of moduli specifications
\[
{\sfr'=\left(\left(\kf',\lf',\beta'\right),\sigma'\coprod\sstar'_{r}\right)}\quad\text{and}\quad{\sfr''=\left(\left(\left(\sstar''_{r}\right)\coprod\kf'',\lf'',\beta''\right),\sigma''\right)}
\]
for
\begin{gather}
\sigma=\sigma'\coprod\sigma'',\;\kf=\kf'\coprod\kf'',\;\lf=\lf'\coprod\lf'',\;\mbox{and }\mbox{\ensuremath{\beta}=\ensuremath{\beta}'+\ensuremath{\beta}''},\label{eq:decomposing moduli specifications}
\end{gather}
and where $r$ is any sufficiently large integer, so $\sstar'_{r},\sstar_{r}''$
denote two new boundary markings (representing the special boundary
points identified by the node). We emphasize that the orientability
condition specifies the order of the fiber factors. Using the $O\left(2m+1\right)$
action, we see that the restriction of $d\ev_{\sstar'_{r}}^{\sfr'}$
to the codimension $k$ strata, 
\[
d\ev_{\sstar'_{r}}^{\sfr'}:TS^{k}\left(\mm_{\sfr'}\right)\to TZ,
\]
is surjective for every $k\geq0$, so the evaluation maps $\ev_{\sstar'_{r}}^{\sfr'},\ev_{\sstar''_{r}}^{\sfr''}$
are b-transverse (see the proof of Lemma \ref{lem:ev product is b-transverse to diag}
below, which generalizes this).

In terms of (\ref{eq:boundary decomposition}), the sturdy boundary
$\partial_{-}\mm_{\sfr}\subset\partial\mm_{\sfr}$ consists of those
components where $\sfr',\sfr''$ are both sturdy. The wobbly boundary
$\partial_{+}\mm_{\sfr}\subset\partial\mm_{\sfr}$ consists of those
components where precisely one of $\sfr'$ or $\sfr''$ is wobbly.

\subsubsection{Wobbly boundary involution}

We construct a fixed-point free involution 
\[
{\inv_{\sfr}:\partial_{+}\mm_{\sfr}\to\partial_{+}\mm_{\sfr}}
\]
as follows. If $\sfr_{0}=\left(\left(k_{0},l_{0},\beta_{0}\right),\sigma_{0}\right)$
is any moduli specification and $S\subset\sigma_{0}$ is a two element
subset, we abuse notation and denote by $S\in\Sym\left(\sigma_{0}\right)$
the permutation that swaps the elements in $S$, and by $S.:\mm_{\sfr_{0}}\to\mm_{\sfr_{0}}$
the induced diffeomorphism. We let $\inv_{\sfr}$ acts on ${\bb=\mm_{\sfr'}\underset{ev_{\sstar'}\qquad ev_{\sstar''}}{\times}\mm_{\sfr''}\subset\partial_{+}\mm_{\sfr}}$
as follows. Precisely one of $\sfr',\sfr''$ is a wobbly boundary
specification, and we denote it by $\sfr_{0}=\left(\left(\kf_{0},\lf_{0},\beta_{0}\right),\sigma_{0}\right)$.
It follows that $\left|\sigma_{0}\right|\geq2$. We take $S_{0}$
to be the first two elements of $\sigma_{0}$, where $\sigma'\coprod\sstar'_{r},\sigma''$
are ordered as subsets of $\sstar'_{\nn}$. Let 
\[
\inv_{\sfr}|_{\bb}=\begin{cases}
\left(S_{0}\right).\times\mbox{id} & \text{if }\sfr_{0}=\sfr'\\
\mbox{id}\times\left(S_{0}\right). & \text{otherwise}
\end{cases}.
\]

Clearly, 
\begin{equation}
\left(\For{}_{\sfr}\right)_{+}\circ\inv_{\sfr}=\left(\For{}_{\sfr}\right)_{+}.\label{eq:For commutes w tau}
\end{equation}
In the next subsection we will see that the wobbly boundary is inessential,
in the sense that $\inv_{\sfr}$ defines an orientation-reversing
involution on $\partial_{+}\mm_{\sfr}$. 

\subsection{\label{subsec:Orientations}Orienting moduli spaces.}

Our goal in this subsection is to prove the following.
\begin{prop}
\label{prop:relative orientation}Let $\sfr=\left(\left(k,l,\beta\right),\sigma\right)$
be a sturdy moduli specification

(a) There exists an isomorphism of local systems
\begin{equation}
\mathcal{J}_{\sfr}:Or\left(T\mm_{\sfr}\right)\to Or\left(TL\right)^{\boxtimes k},\label{eq:relative orientation}
\end{equation}
lying over $\prod_{x\in k}ev_{x}:\mm_{\sfr}\to L^{\times k}$.

(b) The local system isomorphism 
\[
\Or\left(d\inv_{\sfr}\right):\Or\left(T\partial_{+}\mm_{\sfr}\right)\to\Or\left(T\partial_{+}\mm_{\sfr}\right)
\]
satisfies
\begin{equation}
\mathcal{J}_{\sfr}\circ\iota_{\mm_{\sfr}}^{\partial_{+}}\circ\Or\left(d\inv_{\sfr}\right)=\left(-1\right)\mathcal{J}_{\sfr}\circ\iota_{\mm_{\sfr}}^{\partial_{+}},\label{eq:J anti commutes with tau}
\end{equation}
where $\iota_{\mm_{\sfr}}^{\partial_{+}}=\iota_{\mm_{\sfr}}^{\partial}|_{\partial_{+}\mm_{\sfr}}$
is the boundary local system map (see (\ref{eq:boundary ls iso})).
\end{prop}
The proof of this proposition will occupy the remainder of this subsection.
By Lemma \ref{lem:ls extensions}, it suffices to construct local
systems and maps of local systems on interior (depth zero) points
of orbifolds with corners. We will make repeated use of this fact.

\subsubsection{\label{subsec:J Construction}Construction of $\mathcal{J}_{\sfr}$}

If $m$ is even (respectively odd), there are two $Pin^{+}$ (resp.,
$Pin^{-}$) structures on $\rr\pp^{2m}$. Provisionally, let $\mathfrak{p}$
denote one of them. Consider some sturdy moduli specification $\sfr=\left(\basic,\sigma\right)=\left(\left(\kf,\lf,\beta\right),\sigma\right)$.
We will shortly define $\mathcal{J}_{\sfr}$ to be the composition
\[
\Or\left(T\mm_{\sfr}\right)\xrightarrow{\ff_{\sfr}}\Or\left(T\check{\mm}_{\sfr}\right)\xrightarrow{\check{\mathcal{J}}_{\sfr}}\Or\left(TL\right)^{\boxtimes k}
\]
where $\check{\mathcal{J}}_{\sfr}$ is the local system map constructed
by Solomon \cite{JT} (we will review the definition below) and $\ff_{\sfr}:\Or\left(T\mm_{\sfr}\right)\to\Or\left(T\check{\mm}_{\sfr}\right)$
is defined using the natural orientation of the fibers of $\For_{\sfr}$.
Reversing the choice of $\mathfrak{p}$ corresponds to replacing $\check{\jc}_{\sfr}$
by $-\check{\jc}_{\sfr}$ (see \cite[Lemma 2.10]{JT}), and so we
may fix $\mathfrak{p}$ by requiring that $\check{\jc}_{\sfr_{1}=\left(\left(\left\{ 1,2\right\} ,0,1\right),\emptyset\right)}$
gives positive orientation to both points of any generic fiber of
$\ev_{1}^{\sfr_{1}}\times\ev_{2}^{\sfr_{1}}$.

We set

\begin{equation}
\mathcal{F}_{\sfr}:\Or\left(T\mm_{\sfr}\right)\to\Or\left(T\check{\mm}_{\sfr}\right)\label{eq:ls iso forget}
\end{equation}
to be the local system map over $\For_{\sfr}$, which is defined over
the interior of $\left(\check{\mm}_{\sfr}\right)^{\circ}$ using the
ordered direct sum decomposition 
\[
T\mm_{\sfr}=T\check{\mm}_{\sfr}\oplus\ker\left(d\For_{\sfr}\right).
\]
The fiber of $\For_{\sfr}|_{\left(\mm_{\sfr}\right)^{\circ}}$ over
an interior point of $\left(\check{\mm}_{\sfr}\right)^{\circ}$, represented
by $\left(\Sigma,\nu,\kappa,\lambda,u\right)$, is naturally identified
with an open subset of $\left(\partial\Sigma\right)^{\sigma}$ which
is oriented; here we use the order on $\sigma\subset\sstar'_{\nn}$
to order the product.

For any $S\subset\mathbb{N}$ we have a decomposition 
\begin{equation}
\ff_{\sfr}=\grave{\ff}_{\sfr}^{S}\circ\acute{\ff}_{\sfr}^{S}\label{eq:F ls decomposition}
\end{equation}
lying over $\For_{\sfr}=\grave{\For}_{\sfr}^{S}\circ\acute{\For}_{\sfr}^{S}$;
the maps $\grave{\ff}_{\sfr}^{S},\acute{\ff}_{\sfr}^{S}$ are defined
similarly to $\ff_{\sfr}$, using the orientation on $\left(\partial\Sigma\right)^{\sigma\cap\sstar'_{S}}$
and $\left(\partial\Sigma\right)^{\sigma\backslash\sstar'_{S}}$ respectively,
except we twist $\acute{\ff}_{\sfr}^{S}$ by a suitable shuffle sign
so that (\ref{eq:F ls decomposition}) holds.

Next we review the construction of $\check{\mathcal{J}}_{\sfr}$ following
\cite{JT}, beginning with $\check{\mathcal{J}}_{\sfr}^{main}=\check{\mathcal{J}}_{\sfr}|_{\check{\mm}_{\sfr}^{main}}$
where $\check{\mm}_{\sfr}^{main}\subset\check{\mm}_{\sfr}$ is the
clopen component which is the closure of points represented by $\left(\Sigma=D^{2},\nu,\kappa,\lambda,u\right)$
in which $\kappa:\kf\to S^{1}$ preserves the cyclic order; henceforth
we consider $\kf\subset\nn\coprod\sstar''_{\nn}$ as ordered by putting
the elements of $\nn$ first, in order, then the elements of $\sstar''_{\nn}$,
in order.

Let $\widetilde{\mm}^{\text{reg}}\left(\beta\right)$ denote the space
of holomorphic maps $\left(D^{2},\partial D^{2}\right)\to\left(X,L\right)$
of degree $\beta$, where $D^{2}\subset\cc$ is the standard unit
disc. Let $\delta$ be a tuple of $\max\left(0,2-\left|\kf\right|\right)$
dummy markings, so that $\delta\coprod\kf$ has length $k_{+}\geq2$,
and set $\hat{\kf}$ to be the tuple obtained by omitting the first
two elements of $\delta\coprod\kf$. Consider the subspace

\[
{\widehat{\mm}_{\basic}\subset\widetilde{\mm}^{\mbox{reg}}\left(\beta\right)\times\left(\partial D^{2}\right)^{\hat{\kf}}\times\left(D^{2}\right)^{\lf}}
\]
in which the components $\left(z_{3},...,z_{k_{+}}\right)\in\left(\partial D^{2}\right)^{\hat{\kf}}$
are such that 
\[
\left(z_{1},z_{2},...,z_{k_{+}}\right):=\left(+1,-1,z_{3},...,z_{k_{+}}\right)
\]
is cyclically ordered. Using similar notation to \cite[Chapter 8]{FOOO},
our orientation convention is summarized by the following equality
of oriented bases for $T_{p}\widehat{\mm}_{\basic}$: 
\begin{equation}
\mm_{\basic_{+}}\times\rr_{\beta}=\widetilde{\mm}^{\mbox{reg}}\left(\beta\right)\times\left(\partial D^{2}\right)^{\hat{\kf}}\times\left(D^{2}\right)^{\lf}.\label{eq:group orientation convention}
\end{equation}
Here $D^{2}$ and $\partial D^{2}$ are oriented using the complex
structure.  Letting $\basic_{+}=\left(\left(\kf,\lf,\beta\right),\delta\right)$,
$\mm_{\basic_{+}}$ stands for the pullback of a local oriented base
for $T\mm_{\basic_{+}}$ under the map $q_{1,2}:\widehat{\mm}_{\sfr}\to\mm_{\basic_{+}}$
\[
q_{1,2}\left(u,z_{3},...,z_{k_{+}},\boldsymbol{w}\right)\mapsto\left[\Sigma=D^{2},\nu=\emptyset,\kappa,\lambda,u\right]
\]
for $\kappa\left(\delta\coprod\kf\right)=\left(+1,-1,z_{3},...,z_{k_{+}}\right)$
and $\lambda\left(\lf\right)=\boldsymbol{w}$. $\rr_{\beta}$ is
an oriented real line bundle on $\widetilde{\mm}^{\text{reg}}\left(\beta\right)$
representing the action of the subgroup of $PSL_{2}\left(\rr\right)$
fixing the points $\pm1$, with the positive direction corresponding
to the flow from $+1$ to $-1$ (cf. \cite[Convention 8.3.1]{FOOO}).

Now assume without loss of generality that $\kf=\left(1,2,...,k\right)$.
Let ${u:\widehat{\mm}_{\basic}\to\widetilde{\mm}^{\mbox{reg}}\left(\beta\right)}$
and $z\left(j\right):\widehat{\mm}_{\basic}\to\partial D^{2}$, for
$1\leq j\leq k$, denote the projections. Let $u\left(\cdot\right):\widehat{\mm}_{\basic}\times D^{2}\to X$
be the corresponding evaluation map. By (\ref{eq:group orientation convention}),
to construct $\check{\mathcal{J}}_{\sfr}$ it suffices to produce
a section $u^{-1}\Or\left(T\widetilde{\mm}^{\mbox{reg}}\left(\beta\right)\right)\otimes\bigotimes_{j=1}^{k}u\left(z\left(j\right)\right)^{-1}\Or\left(TL\right)^{\vee}$.
We choose an arbitrary orientation for $\left(u\left(z\left(1\right)\right)\right)^{-1}TL$
and transport it along $u|_{\partial D^{2}}$, obtaining orientations
for $\left(u\left(z\left(j\right)\right)\right)^{-1}TL$ for $2\leq j\leq k$
and an orientation for $u|_{\partial D^{2}}^{-1}TL$ if orientable
(that is, if $\beta=0\mod2$). Using the $Pin^{\pm}$ structure $\mathfrak{p}$
and the orientation for $u|_{\partial D^{2}}^{*}TL$ if $\beta=0\mod2$
we obtain an orientation for $\Or\left(T\widetilde{\mm}^{\text{reg}}\left(\beta\right)\right)$
using \cite[Proposition 2.8]{JT}. Reversing the initial choice of
orientation for $\left(u\left(z\left(1\right)\right)\right)^{-1}TL$
reverses the orientations $\left(u\left(z\left(j\right)\right)\right)^{-1}TL$
for all $1\leq j\leq k$ and, if $\beta=0$, also the orientation
for $\Or\left(T\widetilde{\mm}^{\text{reg}}\left(\beta\right)\right)$
(see \cite[Lemma 2.9]{JT}). Since $k+\beta=1\mod2$ it follows that
the section of the tensor product is independent of the initial choice
of orientation. This completes the definition of $\check{\mathcal{J}}_{\sfr}^{main}:\Or\left(T\check{\mm}_{\sfr}^{main}\right)\to\Or\left(TL\right)^{\otimes\kf}$. 

Let $\tau.:\check{\mm}_{\sfr}\to\check{\mm}_{\sfr}$ denote the action
of the permutation $\tau\in\Sym\left(\kf\right)$. Let $\tau_{0}$
denote the cyclic shift, which generates the stabilizer subgroup $\left\{ \tau|\tau.\check{\mm}_{\sfr}^{main}\subset\check{\mm}_{\sfr}^{main}\right\} $.
We have
\begin{equation}
\check{\mathcal{J}}_{\sfr}^{main}\circ\Or\left(d\tau_{0}.\right)=\check{\mathcal{J}}_{\sfr}^{main}.\label{eq:J cyclic invariance}
\end{equation}
Indeed, the shift introduces a sign of $\left(-1\right)^{k-1}=\left(-1\right)^{\beta}$.
On the other hand, we need to compare the orientation transport $OT\left(1\right)$
beginning at $z\left(1\right)$, with the orientation transport $OT\left(2\right)$
beginning at $z\left(2\right)$. We can choose the initial orientation
for $OT\left(2\right)$ so it agrees with $OT\left(1\right)$ at all
of the points $z\left(2\right),z\left(3\right),...,z\left(k\right)$.
With this choice, $OT\left(1\right)$ and $OT\left(2\right)$ also
agree at $z\left(1\right)$ if and only if $\beta=0\mod2$. Eq (\ref{eq:J cyclic invariance})
follows. We extend the definition of $\check{\mathcal{J}}_{\sfr}$
to $\check{\mm}_{\sfr}$ by setting 
\[
\check{\mathcal{J}}_{\sfr}|_{\tau.^{-1}\left(\check{\mm}_{\sfr}^{main}\right)}:=\check{\mathcal{J}}_{\sfr}^{main}\circ\Or\left(d\tau.\right).
\]
This is well defined by (\ref{eq:J cyclic invariance}), and shows
that 
\begin{equation}
\check{\mathcal{J}}_{\sfr}\circ\Or\left(d\tau.\right)=\check{\mathcal{J}}_{\sfr}\label{eq:J Sym(k)-invariance}
\end{equation}
for all $\tau\in\Sym\left(\kf\right)$. Using this it is straightforward
to check that 

\begin{equation}
\mathcal{J}_{\sfr}\circ\Or\left(d\tau.\right)=\mathcal{J}_{\sfr}\label{eq:sign of k permutation}
\end{equation}
for $\tau\in\Sym\left(\kf\right)$, whereas for $\tau\in\Sym\left(\sigma\right)$
we have 
\begin{equation}
\ff_{\sfr}\circ\Or\left(d\tau.\right)=\sgn\left(\tau\right)\cdot\ff_{\sfr}\label{eq:sign of sigma permutation on F}
\end{equation}
and hence
\begin{equation}
\mathcal{J}_{\sfr}\circ\Or\left(d\tau.\right)=\operatorname{sgn}\left(\tau\right)\cdot\mathcal{J}_{\sfr}.\label{eq:sign of sigma permutation}
\end{equation}

\subsubsection{Checking $\inv_{\sfr}$ reverses the orientation.}

We turn to proving part (b) of Proposition \ref{prop:relative orientation}.
We will prove that 
\begin{equation}
\ff_{\sfr}\circ\iota_{\mm_{\sfr}}^{\partial_{+}}\circ\Or\left(d\inv_{\sfr}\right)=\left(-1\right)\ff_{\sfr}\circ\iota_{\mm_{\sfr}}^{\partial_{+}}.\label{eq:F anti commutes with tau_sfr}
\end{equation}
which clearly implies (\ref{eq:J anti commutes with tau}).

Consider a wobbly boundary component 
\[
{\bb=\mm_{\sfr'}\underset{ev_{\sstar'}\qquad ev_{\sstar''}}{\times}\mm_{\sfr''}\subset\partial_{+}\mm_{\sfr}}.
\]
Recall $\inv_{\sfr}|_{\bb}$ was defined by swapping $S_{0}$, the
first two elements of $\sigma'\coprod\left\{ \sstar'_{r}\right\} $
or the first two elements of $\sigma''$, depending on which side
is wobbly. If $S_{0}\subseteq\sigma=\sigma'\coprod\sigma''$, $\inv_{\sfr}|_{\bb}$
is the restriction of $\left(S_{0}\right).:\mm_{\sfr}\to\mm_{\sfr}$,
\begin{equation}
i_{\mm_{\sfr}}^{\partial_{+}}\circ\inv_{\sfr}|_{\bb}=\left(S_{0}\right).\circ i_{\mm_{\sfr}}^{\partial_{+}}|_{\bb}\label{eq:tau extends}
\end{equation}
so (\ref{eq:F anti commutes with tau_sfr}) holds over $\bb$ by (\ref{eq:sign of sigma permutation on F}). 

In case $\sstar_{r}'\in S_{0}$ we have 
\[
\mm_{\sfr'}=\mm_{\left(\left(\left(x\right),\emptyset,0\right),\left\{ s,\sstar_{r}'\right\} \right)}=\left\{ \left(\sstar'_{r},x,s\right),\left(\sstar_{r}',s,x\right)\right\} 
\]
for some $x\in\kf,s=\sstar'_{i}\in\sigma'$, $i<r$. On the right,
each 3-tuple specifies the cyclic order of the markings on the boundary
of the disc. In this case, 
\[
\inv_{\sfr}|_{\bb}=\mbox{sw}\times\mbox{id}_{\mm_{\sfr''}}:\bb\to\bb
\]
where $\mbox{sw}$ swaps the two tuples. By (\ref{eq:F ls decomposition})
we can write
\begin{equation}
\ff_{\sfr}=\grave{\ff}_{\sfr}^{S}\circ\acute{\ff}_{\sfr}^{S}\label{eq:F factors through forgetful map}
\end{equation}
for $S=\sigma\backslash\left\{ i\right\} $. In other words, we forget
$s=\sstar'_{i}$ first, then the rest of $\sigma$.

Since the local system map $\iota_{\mm_{\sfr}}^{\partial}:\Or\left(T\partial\mm_{\sfr}\right)\to\Or\left(T\mm_{\sfr}\right)$
was defined using the outward normal orientation convention, we have
\begin{equation}
\acute{\ff}_{\sfr}^{S}\circ\iota_{\mm_{\sfr}}^{\partial}\circ\Or\left(d\inv_{\sfr}\right)|_{\bb}=-\acute{\ff}_{\sfr}^{S}\circ\iota_{\mm_{\sfr}}^{\partial}\label{eq:upsilon anti commutes with tau}
\end{equation}
(over each fiber of $\acute{\For}_{\sfr}^{S}$, $\inv_{\sfr}|_{\bb}$
swaps two interval endpoints which are identified by $\acute{\For}_{\sfr}^{S}\circ i_{\mm_{\sfr}}^{\partial}|_{B}$).
So (\ref{eq:F anti commutes with tau_sfr}) holds over $\bb$ by (\ref{eq:F factors through forgetful map})
and (\ref{eq:upsilon anti commutes with tau}).

\subsection{\label{subsec:resolutions}Resolutions}

In this subsection we construct resolutions of moduli spaces, which
are orbifolds with corners $\mm_{\basic}^{\rho}$ where $\sfr$ is
a sturdy moduli specification and $\rho\subset\nn$ is a finite subset.
We will see that these are simply products of moduli spaces modeled
on trees. Ultimately, we will be interested in $\rho=\left[r\right]$
and $\sfr=\left(\basic,\emptyset\right)$ for some basic moduli specification
$\basic$, setting 
\[
\mm_{\basic}^{r}=\mm_{\basic}^{\left[r\right]}=\mm_{\left(\basic,\emptyset\right)}^{\left[r\right]}.
\]
The use of sets is needed to write cleaner recursive definitions.
The discussion in this section may seem pedantic; the added detail
is needed only for precise orientation computations.
\begin{defn}
\label{def:recursive definitions}Given a sturdy moduli specification
\[
\sfr=\left(\basic,\sigma\right)=\left(\left(\kf,\lf,\beta\right),\sigma\right)\text{ and }S\subset\nn
\]
we define recursively for every finite subset $\rho\subset\nn$ such
that 
\[
\sstar''_{\rho}\cap\kf=\sstar'_{\rho}\cap\sigma=\emptyset
\]
\begin{itemize}
\item An orbifold with corners $\acute{\mm}_{\sfr}^{\rho,S}$,
\item a pair of forgetful maps 
\[
\mm_{\sfr}^{\rho}:=\acute{\mm}_{\sfr}^{\rho,\nn}\xrightarrow{\acute{\For}_{\sfr}^{\rho,S}}\acute{\mm}_{\sfr}^{\rho,S}\xrightarrow{\grave{\For}_{\sfr}^{\rho,S}}\check{\mm}_{\sfr}^{\rho}=\check{\mm}_{\sfr}^{\rho,\emptyset},
\]
\item a finite set of\emph{ sturdy $\left(\rho,\sfr\right)$-labeled trees,}
$\ts_{\sfr}^{\rho}$, and
\item a locally constant map $\check{\pi}_{\sfr}^{\rho}:\check{\mm}_{\sfr}^{\rho}\to\ts_{\sfr}^{\rho}$,
\end{itemize}
as follows. For $\rho=\emptyset$ we set ${\acute{\mm}_{\sfr}^{\emptyset,S}=\mm_{\left(\basic,\sigma\cap\sstar'_{S}\right)}}$,
$\acute{\For}{}_{\sfr}^{\emptyset,S}=\acute{\For}_{\sfr}^{S}$ and
$\grave{\For}{}_{\sfr}^{\emptyset,S}=\grave{\For}_{\sfr}^{S}$ (cf.
$\S$\ref{subsec:The-forgetful-map}). We take $\ts_{\sfr}^{\emptyset}$
to be a set with one element, which should be thought of as representing
a tree with a single vertex labeled $\sfr$. We set $\check{\pi}_{\sfr}^{\emptyset}$
to be the unique map to a point.

If $\rho\neq\emptyset$, let $r$ denote the maximal element of $\rho$
and let $\hat{\rho}=\rho\backslash\left\{ r\right\} $. We then define
\begin{eqnarray}
\acute{\mm}_{\sfr}^{\rho,S} & = & \coprod'''_{\sfr,\rho}\acute{\mm}_{\sfr'}^{\rho',S}\times\acute{\mm}_{\sfr''}^{\rho'',S}\label{eq:resolution recursively}\\
\acute{\For}{}_{\sfr}^{\rho,S} & = & \coprod'''_{\sfr,\rho}\acute{\For}{}_{\sfr'}^{\rho',S}\times\acute{\For}{}_{\sfr''}^{\rho'',S}\\
\grave{\For}{}_{\sfr}^{\rho,S} & = & \coprod'''_{\sfr,\rho}\grave{\For}{}_{\sfr'}^{\rho',S}\times\grave{\For}{}_{\sfr''}^{\rho'',S}\\
\ts_{\sfr}^{\rho} & = & \coprod'''_{\sfr,\rho}\ts_{\sfr'}^{\rho'}\times\ts_{\sfr''}^{\rho''}\label{eq:sturdy trees recursion}\\
\check{\pi}_{\sfr}^{\rho} & = & \coprod'''_{\sfr,\rho}\pi_{\sfr'}^{\rho'}\times\pi_{\sfr''}^{\rho''}.
\end{eqnarray}
Henceforth, $\coprod'''_{\sfr,\rho}$ denotes a disjoint union over
all partitions $\hat{\rho}=\rho'\coprod\rho''$ and pairs of \emph{sturdy}
moduli specifications $\sfr'=\left(\left(\kf',\lf',\beta'\right),\sigma'\coprod\left\{ \sstar_{r}'\right\} \right)$
and $\sfr''=\left(\left(\kf''\coprod\left\{ \sstar''_{r}\right\} ,\lf'',\beta''\right),\sigma''\right)$
such that 
\[
\sigma=\sigma'\coprod\sigma'',\;\kf=\kf'\coprod\kf'',\;\lf=\lf'\coprod\lf'',\;\mbox{and }\mbox{\ensuremath{\beta}=\ensuremath{\beta}'+\ensuremath{\beta}''}.
\]

We introduce some more notation. Define $\For{}_{\sfr}^{\rho}=\acute{\For}_{\sfr}^{\rho,\emptyset}$,
$\pi_{\sfr}^{\rho}:\mm_{\sfr}^{\rho}\to\ts_{\sfr}^{\rho}$ by 
\[
\pi_{\sfr}^{\rho}=\check{\pi}_{\sfr}^{\rho}\circ\For{}_{\sfr}^{\rho},
\]
and set $\check{\mm}_{\tc}=\left(\check{\pi}_{\sfr}^{\rho}\right)^{-1}\left(\tc\right)$
and $\mm_{\tc}=\left(\pi_{\sfr}^{\rho}\right)^{-1}\left(\tc\right)$.
\end{defn}

For 
\[
x\in\kf\coprod\sstar''_{\rho}\coprod\left(\left(\sigma\coprod\sstar'_{\rho}\right)\cap\sstar'_{S}\right)
\]
we have boundary evaluation maps
\[
\acute{\ev}_{x}^{\sfr,\rho,S}:\acute{\mm}_{\sfr}^{\rho,S}\to L
\]
defined in the obvious way. Similarly, for $x\in\lf$ there are interior
evaluation maps $\acute{\evi}_{x}^{\sfr,\rho,S}:\acute{\mm}_{\sfr}^{\rho,S}\to X$.

For $j\in\rho\cap S$ we define $\acute{\ed}{}_{j}^{\sfr,\rho,S}:\acute{\mm}_{\sfr}^{\rho,S}\to L\times L$
by 
\[
\acute{\ed}{}_{j}^{\sfr,\rho,S}=\acute{\ev}{}_{\sstar'_{j}}^{\sfr,\rho,S}\times\acute{\ev}{}_{\sstar''_{j}}^{\sfr,\rho,S}
\]
and $\ed{}_{\sfr}^{\rho}:=\prod_{j\in\rho}\ed{}_{j}^{\sfr,\rho}:\mm_{\sfr}^{\rho}\to\left(L\times L\right)^{\rho}$,
where the product is ordered by the natural order on $\rho\subset\nn$.

We will mostly be interested in $\ev_{x}^{\sfr,\rho}=\acute{\ev}_{x}^{\sfr,\rho,\nn}$
and $\check{\ev}_{x}^{\sfr,\rho}=\acute{\ev}_{x}^{\sfr,\rho,\emptyset}$.
We have 
\[
\ev_{x}^{\sfr,\rho}=\acute{\ev}_{x}^{\sfr,\rho,S}\circ\acute{\For}_{\sfr}^{\rho,S}
\]
for any $x$ such that both sides of the equality make sense, and
similarly
\[
{\acute{\ev}_{x}^{\sfr,\rho,S}=\check{\ev}_{x}^{\sfr,\rho}\circ\grave{\For}_{\sfr}^{\rho,S}.}
\]

All the spaces are equipped with natural actions of the group ${O_{\basic}=O\left(2m+1\right)\times\Sym\left(\kf\right)\times\Sym\left(\lf\right)}$
making the maps $O_{\basic}$-equivariant. We now want to extend this
to an action of $O_{\basic}^{\rho}:=O_{\basic}\times\Sym\left(\rho\right)$.
To define this, we prove the following lemma, which provides an alternative
definition of resolution using trees.
\begin{lem}
\label{lem:labeled trees alternative def}Let $\sfr=\left(\left(k,l,\beta\right),\sigma\right)$
be a sturdy moduli specification and $\rho$ be a tuple.

(a) The set $\ts_{\sfr}^{\rho}$ is in natural bijection with isomorphism
types of \emph{$\left(\sfr,\rho\right)$-labeled trees} $\tc$. These
are trees with set of oriented edges $\tc_{1}=\rho$, set of vertices
$\tc_{0}$, and maps $\sfr_{\tc}$ assigning to each vertex a sturdy
moduli specification,
\[
\sfr_{\tc}\left(v\right)=\left(\left(\kf_{\tc}\left(v\right),\lf_{\tc}\left(v\right),\beta_{\tc}\left(v\right)\right),\sigma_{\tc}\left(v\right)\right),
\]
such that:

(i) The head (respectively, the tail) of the edge $j\in\tc_{1}$ is
the vertex $v$ if and only if $\sstar''_{j}\in k_{\tc}\left(v\right)$
(resp. $\sstar'_{j}\in\sigma_{\tc}\left(v\right)$)

(ii) We have
\[
\coprod_{v\in\tc_{0}}\kf_{\tc}\left(v\right)=\kf\coprod\left\{ \sstar''_{j}\right\} _{j\in\rho},\coprod\lf_{\tc}\left(v\right)=\lf,\sum\beta_{\tc}\left(v\right)=\beta
\]
and $\coprod\sigma_{\tc}\left(v\right)=\sigma\coprod\left\{ \sstar'_{j}\right\} _{j\in\rho}$

(b) There's a natural order on the vertex set $\tc_{0}$ of each $\tc\in\ts_{\sfr}^{\rho}$
and we have, for any $S\subset\nn$,
\[
\acute{\mm}_{\sfr}^{\rho,S}=\coprod_{\tc\in\ts_{\sfr}^{\rho}}\acute{\mm}_{\tc}^{S}\text{ for }\acute{\mm}_{\tc}^{S}=\prod_{v\in\tc_{0}}\acute{\mm}_{\sfr_{\tc}\left(v\right)}^{S};
\]
and 
\[
\acute{\For}{}_{\sfr}^{\rho,S}=\coprod_{\tc\in\ts_{\sfr}^{\rho}}\prod_{v\in\tc_{0}}\acute{\For}{}_{\sfr_{\tc}\left(v\right)}^{S},\;\grave{\For}{}_{\sfr}^{\rho,S}=\coprod_{\tc\in\ts_{\sfr}^{\rho}}\prod_{v\in\tc_{0}}\grave{\For}{}_{\sfr_{\tc}\left(v\right)}^{S}
\]
and the map $\check{\pi}_{\sfr}^{\rho}$ is defined by $\check{\mm}_{\tc}:=\left(\check{\pi}_{\sfr}^{\rho}\right)^{-1}\left(\tc\right)\simeq\prod_{v\in\tc_{0}}\check{\mm}_{\sfr_{\tc}\left(v\right)}$.

(c) $\Sym\left(\rho\right)$ acts on $\mm_{\sfr}^{\rho},\check{\mm}_{\sfr}^{\rho},\ts_{\sfr}^{\rho}$
in a way which commutes with the $O_{\basic}$ action and makes all
the maps $O_{\basic}^{\rho}$-equivariant in the obvious sense.
\end{lem}
\begin{proof}
We prove (a) by induction on $\rho$. For $\rho=\emptyset$ the claim
is trivial. If $r=\max\rho$ and $\hat{\rho}=\rho\backslash r$, it
suffices to give a bijection of isomorphism types
\[
\left\{ \left(\sfr,\rho\right)\mbox{ trees}\right\} \simeq\coprod'''_{\sfr,\rho}\left\{ \left(\sfr',\rho'\right)\mbox{ trees}\right\} \times\left\{ \left(\sfr'',\rho''\right)\mbox{ trees}\right\} .
\]
One direction is immediate: given a sturdy $\left(\sfr',\rho'\right)$-labeled
tree $\tc'$ and a sturdy $\left(\sfr'',\rho''\right)$-labeled tree
$\tc''$, we construct $\tc$ by connecting the unique vertex $v'\in\tc_{0}'$
such that $\sstar'_{r}\in\sfr_{\tc'}\left(v\right)$ with the unique
vertex $v''\in\tc_{0}''$ such that $\sstar''_{r}\in\sfr_{\tc''}\left(v\right)$.

For the other direction, given a sturdy $\left(\sfr,\rho\right)$-labeled
tree $\tc$, let $\tc_{-}$ denote $\tc$ after we remove the interior
of the edge $e_{r}$ labeled by $r$. We obtain an ordered pair of
trees $\left(\tc',\tc''\right)$ with $\tc'$ (respectively, $\tc''$)
corresponding to the connected component of $\tc_{-}$ containing
the tail $v'$ (resp. the head $v''$) of the edge $e_{r}$. Let $\rho'\subset\rho$
and $\tc_{0}'\subset\tc_{0}$ be the subset of edges and vertices
belonging to $\tc'$, and set $\sfr_{\tc'}=\sfr_{\tc}|_{\tc_{0}'}$.
 We claim 
\[
\sfr'=\left(\left(\coprod_{v\in\tc_{0}'}\kf_{\tc'}\left(v\right)\backslash\left\{ \sstar''_{z}\right\} _{z\in\rho'},\coprod_{v\in\tc_{0}'}\lf_{\tc'}\left(v\right),\sum_{v\in\tc_{0}'}\beta_{\tc'}\left(v\right)\right),\coprod_{v\in\tc_{0}'}\sigma_{\tc'}\left(v\right)\backslash\left\{ \sstar_{z}'\right\} _{z\in\rho'}\right).
\]
is a sturdy moduli specification. To see this, write $\sfr'=\left(\left(\kf',\lf',\beta'\right),\sigma'\right)$.
We check stability: any oriented tree has a vertex with no incoming
edges, and we let $u$ be such a vertex for $\tc'$, so $\kf_{\tc'}\left(u\right)\subset\kf'$.
Since we always have $\lf_{\tc'}\left(u\right)\subset\lf'$ and $\beta_{\tc'}\left(u\right)\leq\beta'$,
stability of $\left(\kf',\lf',\beta'\right)$ follows from stability
of $\left(\kf{}_{\tc'}\left(u\right),\lf_{\tc'}\left(u\right),\beta_{\tc'}\left(u\right)\right)$.
Orientability follows from the following computation mod 2:
\[
\left|\coprod_{v\in\tc_{0}'}\kf_{\tc'}\left(v\right)\backslash\left\{ \sstar''_{z}\right\} _{z\in\rho'}\right|=\sum\left|\kf_{\tc'}\left(v\right)\right|-\left|\rho'\right|=\sum\beta_{\tc'}\left(v\right)+\left|\tc_{0}'\right|-\left|\rho'\right|=\beta'+1.
\]
Clearly, $\tc'$ is an $\left(\sfr',\rho'\right)$-labeled tree. Similarly,
labeling $\tc''$ by $\sfr_{\tc}|_{\tc''_{0}}$ we obtain a $\left(\sfr'',\rho''\right)$-labeled
tree for the appropriate $\left(\sfr'',\rho''\right)$, and one checks
that $\sfr',\rho',\sfr'',\rho''$ satisfy (\ref{eq:boundary decomposition}).
The result follows. 

The proof of (b) is straightforward; the order on $\tc_{0}$ is defined
recursively, so 
\[
\tc_{0}=\tc_{0}'\coprod\tc_{0}''
\]
(that is, the vertices in $\tc_{0}'$ appear before those in $\tc_{0}''$). 

We prove part (c). It is clear that $\Sym\left(\rho\right)$ acts
on $\left(\sfr,\rho\right)$-labeled trees, and this defines an action
on $\ts_{\sfr}^{\rho}$ using part (a). Suppose $\tau\in\Sym\left(\rho\right)$
sends $\tc^{1}\in\ts_{\basic}^{\rho}$ to $\tc^{2}\in\ts_{\basic}^{\rho}$.
We define a diffeomorphism $\tau.:\mm_{\tc^{1}}\to\mm_{\tc^{2}}$
by first relabeling and permuting the markings so $\sstar'_{j}$ maps
to $\sstar'_{\tau\left(j\right)}$ and $\sstar''_{j}$ by $\sstar''_{\tau\left(j\right)}$,
and then permuting factors so that the order of the factors agrees
with the order specified in part (b). Consider for example the following
component $\mm_{\tc}\subset\mm_{\sfr}^{\rho}$ for $\rho=\left(1,2\right)$:
\[
\mm_{\tc}=\mm_{\left(\left(\kf_{1},\lf_{1},\beta_{1}\right),\left(\sstar'_{2}\right)\right)}\times\left(\mm_{\left(\left(\kf_{2},\lf_{2},\beta_{2}\right),\left(\sstar_{1}'\right)\right)}\times\mm_{\left(\left(\kf_{3}\coprod\left(\sstar''_{1},\sstar''_{2}\right),\lf_{3},\beta_{3}\right),\emptyset\right)}\right)
\]
the non-trivial element $\tau\in\Sym\left(\rho\right)$ acts by first
applying a product of three diffeomorphisms, with codomain
\[
\mm_{\left(\left(\kf_{1},\lf_{1},\beta_{1}\right),\left(\sstar'_{1}\right)\right)}\times\left(\mm_{\left(\left(\kf_{2},\lf_{2},\beta_{2}\right),\left(\sstar_{2}'\right)\right)}\times\mm_{\left(\left(\kf_{3}\coprod\left(\sstar''_{1},\sstar''_{2}\right),\lf_{3},\beta_{3}\right),\emptyset\right)}\right).
\]
where the diffeomorphism on the right involves a permutation of the
tuple $\left(\partial D^{2}\right)^{\kf_{3}\coprod\left(\sstar''_{1},\sstar''_{2}\right)}$
(the left and middle diffeomorphism are just relabeling of values
in this case). Then we apply the associator and commutator for the
product to identify this with $\mm_{\tilde{\tc}}\subset\mm_{\sfr}^{\rho}$
for 
\[
\mm_{\tilde{\tc}}=\mm_{\left(\left(\kf_{2},\lf_{2},\beta_{2}\right),\left(\sstar_{2}'\right)\right)}\times\left(\mm_{\left(\left(\kf_{1},\lf_{1},\beta_{1}\right),\left(\sstar'_{1}\right)\right)}\times\mm_{\left(\left(\kf_{3}\coprod\left(\sstar''_{1},\sstar''_{2}\right),\lf_{3},\beta_{3}\right),\emptyset\right)}\right).
\]
We define a diffeomorphism $\check{\mm}_{\tc^{1}}\to\check{\mm}_{\tc^{2}}$
similarly by mapping $\sstar''_{j}$ to $\sstar''_{\tau\left(j\right)}$
and permuting factors. Clearly these maps commute with the $O_{\basic}$
action and the other maps.
\end{proof}
This lemma gives a manifestly $Sym\left(\rho\right)$-symmetric alternative
to Definition \ref{def:recursive definitions}.
\begin{lem}
\label{lem:ev product is b-transverse to diag}For every $S\subset\rho$
the product $\ed_{S}^{\sfr,\rho}:=\prod_{j\in S}\ed_{j}^{\sfr,\rho}:\mm_{\sfr}^{\rho}\to\left(L\times L\right)^{S}$
is b-transverse to the diagonal 
\[
\Delta_{L}^{\times S}:L^{\times S}\to\left(L\times L\right)^{\times S},
\]
so the fiber product
\[
\left(\ed_{S}^{\sfr,\rho}\right)^{-1}\left(\Delta\right):=L^{\times S}\fibp{\Delta_{L}^{S}}{\prod_{j\in S}\ed_{j}^{\sfr,\rho}}\mm_{\sfr}^{\rho}
\]
exists, and its corners are described by Eq (\ref{eq:corners as fibp in orb}).
\end{lem}
\begin{proof}
$O\left(2m+1\right)^{\tc_{0}}$ acts on $\mm_{\tc}=\prod_{v\in\tc_{0}}\mm_{\sfr_{\tc}\left(v\right)}$,
and the maps $\ev_{x}^{\sfr,\rho}|_{\mm_{\tc}}:\mm_{\tc}\to L$ are
naturally $O\left(2m+1\right)^{\tc_{0}}\to O\left(2m+1\right)$ equivariant
with respect to the appropriate projection. The map linearization
of the action 
\[
\mathfrak{o}\left(2m+1\right)\to T_{y}L
\]
is surjective at every $y\in L$. 

Let $p\in S^{k}\left(\mm_{\sfr}^{\rho}\right)$ be a point of depth
$k$ with $\ed_{S}^{\sfr,\rho}\left(p\right)=\Delta_{L}^{\times S}\left(y\right)$
for $y=\left(y_{1},...,y_{s}\right)\in L^{S}$. Let $v=\left(v_{1},...,v_{s}\right)$,
$v_{i}\in T_{y_{i}}L$ represent a normal vector to $\Delta^{\times S}$
at $y$. More precisely, we use the isomorphism $N_{\Delta}^{L\times L}\simeq\pr_{2}^{*}TL$
where $\pr_{2}$ denotes projection on the head of the edge $i$.
Let $\tc\backslash S$ denote the tree with the edges corresponding
to $S$ removed. For each connected component $C_{j}\subset\tc\backslash S$
we fix an $O\left(2m+1\right)$-fundamental vector field $\theta_{j}$
on $L$ in such a way that $\theta_{j_{2}}|_{y_{i}}-\theta_{j_{1}}|_{y_{i}}=v_{i}$
whenever the tail and head of $i$ are incident to $C_{j_{1}}$ and
$C_{j_{2}}$ respectively. The corresponding lie algebra elements
define a lift of $v$ to $w\in T_{p}S^{k}\left(\mm_{\sfr}^{\rho}\right)$.
Now use Remark \ref{rem:easy b-transversality} and Lemma \ref{lem:fibp in orb}.
\end{proof}
\begin{rem}
For $\left|S\right|=1$, $\ed_{S}^{\sfr,\rho}$ is in fact a b-submersion. 
\end{rem}

\subsubsection{Resolutions and the boundary.}

As $\rho$ varies, the spaces $\mm_{\sfr}^{\rho}$ form a resolution
of the sturdy boundary in the sense that for $r>\max\rho$ there's
a map $g_{\sfr}^{\rho\coprod\left\{ r\right\} }:\partial_{-}\mm_{\sfr}^{\rho}\to\mm_{\sfr}^{\rho\coprod\left\{ r\right\} }$,
which we shall now define.

Comparing (\ref{eq:boundary decomposition}) and (\ref{eq:resolution recursively})
we see that for any $r\in\nn$ there's a map
\[
g!_{\sfr}^{\left\{ r\right\} }:\partial_{-}\mm_{\sfr}^{\emptyset}\to\mm_{\sfr}^{\left\{ r\right\} }
\]
sitting in a cartesian square
\[
\xymatrix{\partial_{-}\mm_{\sfr}^{\emptyset}\ar[r]^{g!_{\sfr}^{\left\{ r\right\} }}\ar[d] & \mm_{\sfr}^{\left\{ r\right\} }\ar[d]^{\ed_{r}^{\sfr,\left\{ r\right\} }}\\
L\ar[r]_{\Delta} & L\times L
}
\]

More generally for $r>\max\rho$ we define an $O_{\sfr}^{\rho}$-equivariant
map
\[
g!_{\sfr}^{\rho\coprod\left\{ r\right\} }:\partial_{-}\mm_{\sfr}^{\rho}\to\mm_{\sfr}^{\rho\coprod\left\{ r\right\} }
\]
by recursion on $\rho$. For $\rho=\emptyset$ we've already defined
it. For $\left|\rho\right|\geq1$ we set $g!_{\sfr}^{\rho\coprod\left\{ r\right\} }$
to be the composition of the following maps 
\begin{multline*}
\partial_{-}\mm_{\sfr}^{\rho}=\coprod'''_{\sfr,\hat{\rho}}\left(\left(\partial_{-}\mm_{\sfr'}^{\rho'}\times\mm_{\sfr''}^{\rho''}\right)\coprod\left(\mm_{\sfr'}^{\rho'}\times\partial_{-}\mm_{\sfr''}^{\rho''}\right)\right)\xrightarrow{\left(1\right)}\\
\xrightarrow{\left(1\right)}\coprod'''_{\sfr,\hat{\rho}}\left(\left(\mm_{\sfr'}^{\rho'\coprod\left\{ r\right\} }\times\mm_{\sfr''}^{\rho''}\right)\coprod\left(\mm_{\sfr'}^{\rho'}\times\mm_{\sfr''}^{\rho''\coprod\left(r\right)}\right)\right)=\mm_{\sfr}^{\rho\coprod\left(r\right)}.
\end{multline*}
The map (1) is $\coprod'''_{\sfr,\hat{\rho}}\left(g!_{\sfr'}^{\rho'\coprod\left\{ r\right\} }\times\id\coprod\id\times g!_{\sfr''}^{\rho''\coprod\left\{ r\right\} }\right)$. 

We define the map $g_{\sfr}^{\rho\coprod\left\{ r\right\} }$ as the
composition
\begin{equation}
\partial_{-}\mm_{\sfr}^{\rho\coprod\left\{ r\right\} }\xrightarrow{g!_{\sfr}^{\rho\coprod\left\{ r\right\} }}\mm_{\sfr}^{\rho\coprod\left\{ r\right\} }\xrightarrow{\tau.}\mm_{\sfr}^{\rho\coprod\left\{ r\right\} }\label{eq:def of g}
\end{equation}
where $\tau\in\Sym\left(\rho\coprod\left\{ r\right\} \right)$ is
the cyclic permutation that sends $r$ to $\min\rho$. $g_{\sfr}^{\rho\coprod r}$
is an $O_{\sfr}^{\rho}=O\left(2m+1\right)\times\Sym\left(\kf\right)\times\Sym\left(\lf\right)\times\Sym\left(\sigma\right)\times\Sym\left(\rho\right)$
equivariant map, which sits in a cartesian square 
\begin{equation}
\xymatrix{\partial_{-}\mm_{\sfr}^{\rho}\ar[r]^{g_{\sfr}^{\rho\coprod\left\{ r\right\} }}\ar[r]\ar[d] & \mm_{\sfr}^{\rho}\ar[d]^{\ed_{r}^{\sfr,\rho}}\\
L\ar[r]_{\Delta_{L}} & L\times L
}
\label{eq:boundary as pullback}
\end{equation}

Since $\For{}_{\sfr}^{\rho}$ is $\Sym\left(\rho\right)$ equivariant,
the decomposition $\partial\mm_{\sfr}^{\rho}=\partial_{-}\mm_{\sfr}^{\rho}\coprod\partial_{+}\mm_{\sfr}^{\rho}$
is $\Sym\left(\rho\right)$-equivariant. The following lemma will
allow us to compute the sign of the $\Sym\left(\rho\right)$ action
on local systems more easily.
\begin{lem}
\label{cor:bd^r as pullback}Let $\sfr$ be a sturdy moduli specification,
$\rho\subset\nn$ finite, and $r$ a non-negative integer. We denote
$\rho_{+r}=\rho\coprod\left\{ \max\rho+1,...,\max\rho+r\right\} $. 

(a) There exists a clopen component $\partial_{-}^{r}\mm_{\sfr}^{\rho}\subset\partial^{r}\mm_{\sfr}^{\rho}$
and an $O_{\sfr}^{\rho}$-equivariant map 
\[
g_{\sfr}^{\rho_{+r},r}:\partial_{-}^{r}\mm_{\sfr}^{\rho}\to\mm_{\sfr}^{\rho_{+r}}
\]
which sits in a cartesian square 
\begin{equation}
\xymatrix{\partial_{-}^{r}\mm_{\sfr}^{\rho}\ar[r]^{g_{\sfr}^{\rho_{+r},r}}\ar[d] & \mm_{\sfr}^{\rho_{+r}}\ar[d]^{\ed{}_{\left(1,...,r\right)}^{\sfr,\rho_{+r}}}\\
L^{r}\ar[r]_{\Delta_{L}^{\times r}} & \left(L\times L\right)^{r}
}
\label{eq:bd^r as pullback}
\end{equation}
where $\ed{}_{\left(1,...,r\right)}^{\sfr,\rho_{+r}}:=\ed{}_{1}^{\sfr,\rho_{+r}}\times\cdots\times\ed{}_{r}^{\sfr,\rho_{+r}}$
is b-transverse to $\Delta_{L}^{r}$.

(b) As a subgroup of $\Sym\left(\rho_{+r}\right)$, $\Sym\left(r\right)$
acts on $\mm_{\sfr}^{\rho_{+r}}$. $\Sym\left(r\right)$ also acts
on $\Delta_{L}^{r}$ and $\left(L\times L\right)^{r}$ by permuting
factors. The induced action on the fibered product $\partial_{-}^{r}\mm_{\sfr}^{\rho}$
is the restriction of the $Sym\left(r\right)$ action on $\partial^{r}\mm_{\sfr}^{\rho}$
permuting the local boundary components.
\end{lem}
\begin{proof}
We prove part (a), by induction on $r$. For $r=0$ we take $\partial_{-}^{0}\mm_{\sfr}^{\rho}=\partial^{0}\mm_{\sfr}^{\rho}=\mm_{\sfr}^{\rho}$
and $g_{\sfr}^{\rho,0}=\id$ so the claim is trivial.

For $r\geq1$, use the inductive hypothesis to obtain a cartesian
square 
\begin{equation}
\xymatrix{\partial\partial_{-}^{\left(r-1\right)}\mm_{\sfr}^{\rho}\ar[d]\ar[r]^{\partial g_{\sfr}^{\rho,r-1}} & \partial\mm_{\sfr}^{\rho_{+\left(r-1\right)}}\ar[d]^{\ed{}_{\left(z_{1},...,z_{r-1}\right)}^{\sfr,\rho_{+\left(r-1\right)}}\circ i_{\mm_{\sfr}^{\rho_{+\left(r-1\right)}}}^{\partial}}\\
\Delta^{r-1}\ar[r] & \left(L\times L\right)^{r-1}
}
.\label{eq:boundary of bd^(r-1)}
\end{equation}
We take 
\[
\partial_{-}^{r}\mm_{\sfr}^{\rho}:=\left(\partial g_{\sfr}^{\rho,r-1}\right)^{-1}\left(\partial_{-}\mm_{\sfr}^{\rho_{+\left(r-1\right)}}\right),
\]
where $\partial_{-}\mm_{\sfr}^{\rho_{+\left(r-1\right)}}:=\partial_{-}^{\For_{\sfr}^{\rho_{+\left(r-1\right)}}}\mm_{\sfr}^{\rho_{+\left(r-1\right)}}$
is the horizontal clopen component as in $\S$\ref{subsec:b-normal boundary decomp}.
We obtain a pair of cartesian squares 
\begin{equation}
\xymatrix{\partial_{-}^{r}\mm_{\sfr}^{\rho}\ar[d]\ar[r] & \partial_{-}\mm_{\sfr}^{\rho_{+\left(r-1\right)}}\ar[d]^{\ed{}_{\left[r-1\right]}\circ i^{\partial}}\\
\Delta^{r-1}\ar[r] & \left(L\times L\right)^{r-1}
}
\;\xymatrix{\partial_{-}\mm_{\sfr}^{\rho_{+\left(r-1\right)}}\ar[d]\ar[r] & \mm_{\sfr}^{\rho_{+\left(r-1\right)}\coprod\left\{ r\right\} }\ar[d]^{\ed_{r}}\\
\Delta\ar[r] & L\times L
}
.\label{eq:two squares}
\end{equation}
Note that $\ed{}_{\left[r-1\right]}=\ed{}_{\left(1,...,r-1\right)}^{\sfr,\rho_{+\left(r-1\right)}}$
factors through $\partial_{-}\mm_{\sfr}^{\rho_{+\left(r-1\right)}}\to\mm_{\sfr}^{\rho_{+\left(r-1\right)}\coprod\left(r\right)}=\mm_{\sfr}^{\rho_{+r}}$
and in each of the squares the bottom and right maps are b-transverse
to one another by Lemma \ref{lem:ev product is b-transverse to diag},
so the b-transverse cartesian square (\ref{eq:bd^r as pullback})
is obtained from (\ref{eq:two squares}). More precisely, a simple
diagram chase shows that given two cartesian squares
\[
\xymatrix{P\ar[r]\ar[d] & Q\ar[d]^{f}\\
A\ar[r]_{g} & B
}
\;\xymatrix{Q\ar[r]\ar[d] & E\ar[d]^{a}\\
C\ar[r]_{b} & D
}
\]
with $df\oplus dg:TQ\oplus TA\to TB$ and $da\oplus db:TE\oplus TC\to TD$
surjective, the square 
\[
\xymatrix{P\ar[r]\ar[d] & E\ar[d]^{r\times a}\\
A\times C\ar[r]_{g\times b} & B\times D
}
\]
is cartesian too and $d\left(r\times a\right)\oplus d\left(g\times b\right)$
is surjective. The result follows from this by Remark \ref{rem:easy b-transversality}.

Part (b) is easy to see.
\end{proof}
We now want discuss an analogous relationship between the spaces $\check{\mm}_{\sfr}^{\rho}$
for various $\rho$, but first we introduce some more notation. Let
$r>\max\rho$. We define a map 
\begin{equation}
\cnt=\cnt_{\sfr}^{\rho\coprod\left\{ r\right\} }:\ts_{\sfr}^{\rho\coprod\left\{ r\right\} }\to\ts_{\sfr}^{\rho}\label{eq:tree contraction}
\end{equation}
which one may think of as contracting the edge $e_{r}$ to some vertex
$v_{i}\in\tc_{0}=\cnt\left(\tc_{+}\right)_{0}$ which carries the
sum of the degrees and the disjoint union of the labels of the two
incident vertices, except for $\sstar'_{r},\sstar''_{r}$ which are
discarded. More precisely, the map sends a sturdy tree $\tc_{+}\in\ts_{\sfr}^{\rho\coprod\left\{ r\right\} }$
to the unique $\tc\in\ts_{\sfr}^{\rho}$ such that 
\[
\left(g_{\sfr}^{\rho\coprod\left\{ r\right\} }\right)^{-1}\left(\mm_{\tc_{+}}\right)\subset\partial\mm_{\tc}.
\]
Setting 
\[
\partial^{\tc_{+}}\mm_{\tc}:=\left(g_{\sfr}^{\rho\coprod\left\{ r\right\} }\right)^{-1}\left(\mm_{\tc_{+}}\right),
\]
we have 
\[
\partial_{-}\mm_{\tc}=\coprod_{\left\{ \tc_{+}|\cnt\tc_{+}=\tc\right\} }\partial^{\tc_{+}}\mm_{\tc}.
\]
Now consider some $S\subset\nn$, $r\not\in S$. We define\textbf{
}
\begin{equation}
\partial^{\tc_{+}}\acute{\mm}_{\tc}^{S}=\left(\acute{\For}_{\sfr}^{\rho,S}\right)_{-}\left(\partial^{\tc_{+}}\mm_{\tc}\right)\subset\partial\acute{\mm}_{\tc}^{S}.\label{eq:partial tc+ def}
\end{equation}
If $\tc_{1},\tc_{2}\in\ts_{\sfr}^{\rho}$ differ by moving $\sstar'_{i}$,
$i\not\in S$ from the head of $e_{r}$ to the tail of $e_{r}$ or
vice-versa, then $\partial^{\tc_{1}}\acute{\mm}_{\tc}^{S}=\partial^{\tc_{2}}\acute{\mm}_{\tc}^{S}$.
There's a map 
\begin{equation}
\acute{g}_{\tc_{+}}^{S}:\partial^{\tc_{+}}\acute{\mm}_{\tc}^{S}\to\acute{\mm}_{\tc_{+}}^{S\coprod\left\{ r\right\} }\label{eq:acute g_tc_+}
\end{equation}
sitting in cartesian squares 
\begin{equation}
\xymatrix{\partial^{\tc_{+}}\mm_{\tc}\ar[d]_{\left(\acute{\For}_{\sfr}^{\rho,S}\right)_{-}}\ar[r]^{g_{\basic}^{r+1}} & \mm_{\tc_{+}}\ar[d]^{\acute{\For}_{\sfr}^{\rho\coprod r,S\coprod r}}\\
\partial^{\tc_{+}}\check{\mm}_{\tc}\ar[r]_{\acute{g}_{\tc_{+}}^{S}}\ar[d] & \acute{\mm}_{\tc_{+}}^{S\coprod\left\{ r\right\} }\ar[d]^{\ed_{r}^{\sfr,\rho\coprod r,S\coprod r}}\\
L\ar[r]_{\Delta} & L\times L
}
\label{eq:acute g cartesian squares}
\end{equation}
whose composition is the restriction of (\ref{eq:boundary as pullback})
to $\mm_{\tc_{+}}$.

We define 
\begin{equation}
\check{g}_{\tc_{+}}:\partial^{\tc_{+}}\check{\mm}_{\tc}\to\check{\mm}_{\tc_{+}}\label{eq:check g definition}
\end{equation}
by $\check{g}_{\tc_{+}}=\grave{\For}_{\sfr}^{\rho,\left\{ r\right\} }\circ\acute{g}_{\tc_{+}}^{\emptyset}$.

\subsection{\label{subsec:Orienting-resolutions}Orienting resolutions}

\subsubsection{Overview}

Our goal in this section is to construct local system maps
\[
\ff_{\sfr}^{\rho}:\Or\left(T\mm_{\sfr}^{\rho}\right)\to\Or\left(T\check{\mm}_{\sfr}^{\rho}\right)\text{ over }\For_{\sfr}^{\rho}
\]
and
\[
\check{\mathcal{J}}_{\sfr}^{\rho}:\Or\left(T\check{\mm}_{\sfr}^{\rho}\right)\to\Or\left(TL\right)^{\boxtimes\kf\coprod\sstar''_{\rho}}\text{ over }\prod_{x\in\sstar''_{\rho}\coprod\kf}\check{\ev}_{x}^{\sfr,\rho}
\]
extending $\ff_{\sfr}=\ff_{\sfr}^{\emptyset}$ and $\check{\mathcal{J}}_{\sfr}=\check{\mathcal{J}}_{\sfr}^{\emptyset}$
so that $\mathcal{J}_{\sfr}^{\rho}:=\check{\mathcal{J}}_{\sfr}^{\rho}\circ\ff_{\sfr}^{\rho}$
can be used to integrate forms on $\mm_{\sfr}^{\rho}$. We will show
that the following two properties, essential to the proof of Stokes'
theorem, hold:

\begin{align*}
\text{(coherence)}\qquad & \left(\mathcal{J}_{\sfr}^{\rho\coprod\left(r\right)}\right)^{\to}\circ\mathcal{G}_{\sfr}^{\rho\coprod\left(r\right)}=\mathcal{J}_{\sfr}^{\rho}\circ\iota_{\mm_{\sfr}^{\rho}}^{\partial_{-}}.\\
\text{(}\Sym\left(\rho\right)\text{-equivariance)}\qquad & \mathcal{J}_{\sfr}^{\rho}\circ\Or\left(d\tau.\right)=\sgn\left(\tau\right)\cdot\mathcal{J}_{\sfr}^{\rho}
\end{align*}
Here the local system map
\[
\left(\mathcal{J}_{\sfr}^{\rho\coprod\left(r\right)}\right)^{\to}:\Or\left(T\mm_{\sfr}^{\rho\coprod\left(r\right)}\right)\otimes\left(\ev_{\sstar''_{r}}^{\rho\coprod\left(r\right)}\right)^{-1}\Or\left(TL\right)^{\vee}\to\Or\left(TL\right)^{\boxtimes\kf\coprod\sstar''_{\rho}}
\]
is obtained from $\mathcal{J}_{\sfr}^{\rho\coprod\left(r\right)}$
by the tensor-hom adjunction, and
\begin{equation}
\mathcal{G}_{\sfr}^{\rho\coprod\left(r\right)}:\Or\left(T\partial_{-}\mm_{\sfr}^{\rho}\right)\to\Or\left(T\mm_{\sfr}^{\rho\coprod\left(r\right)}\right)\otimes\left(\ev_{\sstar''_{r}}^{\rho\coprod\left(r\right)}\right)^{-1}\Or\left(TL\right)^{\vee}\label{eq:fiber square ls map}
\end{equation}
is induced from the cartesian square (\ref{eq:boundary as pullback}).
Orienting fiber products involves some convention. To this end,
we use the short exact sequences
\[
0\to\Or\left(T\partial_{-}\mm_{\sfr}^{\rho}\right)\to\left(g_{\sfr}^{\rho\coprod\left(r\right)}\right)^{-1}T\mm_{\sfr}^{\rho\coprod\left(r\right)}\oplus\ev_{\sstar'_{r}}^{-1}TL\to\ed_{r}^{-1}T\left(L\times L\right)\to0
\]
and
\[
0\to\ev_{\sstar'_{r}}^{-1}TL\to\ed_{r}^{-1}T\left(L\times L\right)\to\ev_{\sstar''_{r}}^{-1}TL\to0,
\]
in conjunction with (\ref{eq:ls map from ses}).

We will conclude this subsection with an explicit computation of $\mathcal{J}_{\sfr}^{\rho}|_{\mm_{\tc}}$
for $\tc\in\ts_{\sfr}^{\rho}$ a special kind of tree, which will
come up in the fixed point localization computation in \cite{fp-loc-OGW}.

\subsubsection{Construction of $\ff_{\sfr}^{\rho},\check{\mathcal{J}}_{\sfr}^{\rho}$}

Let $\sfr$ be a sturdy moduli specification, $\rho\subset\nn$ a
finite subset and $S\subset\nn$ any subset. Using the order on vertices
in Lemma \ref{lem:labeled trees alternative def}(b) we obtain an
isomorphism of local systems 
\begin{equation}
\Or\left(T\acute{\mm}_{\sfr}^{\rho,S}\right)\simeq\coprod_{\tc\in\ts_{\sfr}^{\rho}}\boxtimes_{v\in\tc_{0}}\Or\left(T\acute{\mm}_{\sfr_{\tc}\left(v\right)}^{S}\right)\label{eq:tangent direct sum decomposition}
\end{equation}
Using this and
\[
\coprod_{\tc\in\ts_{\sfr}^{\rho}}\boxtimes\ff_{\sfr_{\tc}\left(v\right)},\coprod\boxtimes\acute{\ff}_{\sfr_{\tc}\left(v\right)}^{S},\coprod\boxtimes\grave{\ff}_{\sfr_{\tc}\left(v\right)}^{S},\coprod\boxtimes\check{\mathcal{J}}_{\sfr_{\tc}\left(v\right)},\coprod\boxtimes\mathcal{J}_{\sfr_{\tc}\left(v\right)}
\]
 we obtain local system maps 

\[
\left(\prod\ff\right){}_{\sfr}^{\rho}:\Or\left(T\mm_{\sfr}^{\rho}\right)\to\Or\left(T\check{\mm}_{\sfr}^{\rho}\right)\text{ over }\For_{\sfr}^{\rho},
\]

\[
\left(\prod\acute{\ff}\right){}_{\sfr}^{\rho}:\Or\left(T\mm_{\sfr}^{\rho}\right)\to\Or\left(T\acute{\mm}_{\sfr}^{\rho,S}\right)\text{ over }\acute{\For}_{\sfr}^{\rho,S},
\]
\[
\left(\prod\grave{\ff}\right){}_{\sfr}^{\rho}:\Or\left(T\acute{\mm}_{\sfr}^{\rho,S}\right)\to\Or\left(T\check{\mm}_{\sfr}^{\rho}\right)\text{ over }\grave{\ff}_{\sfr}^{\rho,S},
\]
\[
\left(\prod\check{\mathcal{J}}\right)_{\sfr}^{\rho}:\Or\left(T\check{\mm}_{\sfr}^{\rho}\right)\to\Or\left(TL\right)^{\boxtimes\kf\coprod\sstar''_{\rho}}\text{ over }\prod_{x\in\kf\coprod\sstar''_{\rho}}\check{\ev}_{x}^{\sfr,\rho}
\]
and 

\[
\left(\prod\mathcal{J}\right)_{\sfr}^{\rho}:\Or\left(T\mm_{\sfr}^{\rho}\right)\to\Or\left(TL\right)^{\boxtimes\kf\coprod\sstar''_{\rho}}\text{ over }\prod_{x\in\kf\coprod\sstar''_{\rho}}\ev_{x}^{\sfr,\rho}
\]
with $\left(\prod\mathcal{J}\right)_{\sfr}^{\rho}=\left(\prod\check{\mathcal{J}}\right)_{\sfr}^{\rho}\circ\left(\prod\ff\right){}_{\sfr}^{\rho}$
and $\left(\prod\ff\right){}_{\sfr}^{\rho}=\left(\prod\grave{\ff}\right){}_{\sfr}^{\rho}\circ\left(\prod\acute{\ff}\right){}_{\sfr}^{\rho}$.

However, $\left(\prod\jc\right)_{\sfr}^{\rho}$ does \emph{not} satisfy
the coherence condition above. The plan is as follows. We will twist
$\left(\prod\ff\right){}_{\sfr}^{\rho},\left(\prod\check{\mathcal{J}}\right)_{\sfr}^{\rho}$
by certain signs to make them coherent (in a sense which we will make
precise shortly), and prove that this makes $\left(\prod\jc\right)_{\sfr}^{\rho}$
coherent. We will see $\Sym\left(\rho\right)$ equivariant follows
from coherence.
\begin{lem}
\label{lem:boundary of ls map}Let $\ff:\lc_{1}\to\lc_{2}$ be a local
system map lying over a b-normal map of orbifolds $f:\xx\to\yy$.
There's an induced local system map $\ff_{-}:\left(i_{\xx}^{\partial_{-}^{f}}\right)^{-1}\lc_{1}\to\left(i_{\yy}^{\partial}\right)^{-1}\lc_{2}$
over $f_{-}$. In case $\lc_{1}=Or\left(T\xx\right)$ and $\lc_{2}=Or\left(T\yy\right)$
we can use $\iota_{\xx}^{\partial_{-}^{f}},\iota_{\yy}^{\partial}$
to define a local system map 
\[
\partial\ff:Or\left(T\partial_{-}^{f}\xx\right)\to Or\left(T\partial\yy\right)
\]
over $f_{-}$, which satisfies $\ff\circ\iota_{\xx}^{\partial_{-}^{f}}=\iota_{\yy}^{\partial}\circ\partial\ff$.
\end{lem}
\begin{proof}
straightforward.
\end{proof}

Fix some $r>\max\rho$. For every $\tc_{+}\in\ts_{\sfr}^{\rho\coprod\left\{ r\right\} }$
we define 

\[
\acute{\gc}_{\tc_{+}}:\Or\left(\partial^{\tc_{+}}\check{\mm}_{\cnt\tc_{+}}\right)\to\Or\left(\acute{\mm}_{\tc_{+}}^{\left\{ r\right\} }\right)\otimes\left(\acute{\ev}_{\sstar''_{r}}^{\sfr,\rho\coprod r,\left\{ r\right\} }\right)^{-1}\Or\left(TL\right)
\]
over $\acute{g}_{\tc_{+}}^{\emptyset}$ by the bottom cartesian square
in (\ref{eq:acute g cartesian squares}), using the same convention
as in (\ref{eq:fiber square ls map}). We define a local system map
over $\check{g}_{\tc_{+}}=\grave{\For}_{\sfr}^{\rho,\left\{ r\right\} }\circ\acute{g}_{\tc_{+}}^{\left\{ r\right\} }$
by 
\[
\check{\mathcal{G}}_{\tc_{+}}=\left(\left(\prod\grave{\ff}\right)_{\sfr}^{\rho,\left\{ r\right\} }\otimes\alpha\right)\circ\acute{\gc}_{\tc_{+}}.
\]
Here 
\[
\alpha:\left(\acute{\ev}_{\sstar''_{r}}^{\sfr,\rho\coprod r,\left\{ r\right\} }\right)^{-1}\Or\left(TL\right)\to\left(\acute{\ev}_{\sstar''_{r}}^{\sfr,\rho\coprod r,\emptyset}\right)^{-1}\Or\left(TL\right)
\]
is the associator local system map lying over $\grave{\For}_{\sfr}^{\rho,\left\{ r\right\} }$,
associated with the factorization of the evaluation map ${\acute{\ev}_{\sstar''_{r}}^{\sfr,\rho\coprod r,\left\{ r\right\} }=\acute{\ev}_{\sstar''_{r}}^{\sfr,\rho,\emptyset}\circ\grave{\For}_{\sfr}^{\rho,\left\{ r\right\} }}$.
In what follows we will abuse notation and use $\alpha$ to denote
any such associator. We will also avoid excessive decorations and
write just $\ev_{\sstar''_{r}}$ if the domain is clear from the context.

The following proposition defines $\ff_{\sfr}^{\rho}$ and the coherence
condition that it satisfies.
\begin{prop}
\label{prop:F definition and coherence}There exists some function
$\delta_{\sfr}^{\rho}:\ts_{\sfr}^{\rho}\to\left\{ \pm1\right\} $
such that if we define
\[
\ff_{\sfr}^{\rho}=\left(\delta_{\sfr}^{\rho}\circ\pi_{\sfr}^{\rho}\right)\cdot\left(\prod\ff\right)_{\sfr}^{\rho},
\]
then we have for any $r>\max\rho$ and $\tc_{+}\in\ts_{\sfr}^{\rho\coprod r}$,
\begin{equation}
\left[\ff_{\sfr}^{\rho\coprod r}\otimes\alpha\right]\circ\mathcal{G}_{\sfr}^{\rho\coprod r}|_{\partial^{\tc_{+}}\mm_{\tc}}=\check{\gc}_{\tc_{+}}\circ\partial^{\tc_{+}}\ff_{\sfr}^{\rho}\label{eq:F coherence}
\end{equation}
where we denote $\tc=\cnt\tc_{+}$ and $\partial^{\tc_{+}}\ff_{\sfr}^{\rho}:=\partial\ff_{\sfr}^{\rho}|_{\partial^{\tc_{+}}\mm_{\tc}}.$
\end{prop}
The proof relies on the following lemma, which says that the failure
of $\prod\ff$ to satisfy coherence is measured by a function $\epsilon$
which is constant on components of the form $\partial^{\tc_{+}}\mm_{\cnt\tc_{+}}$.
We will then define the correction function $\delta$ recursively
by a kind of difference equation which depends on $\epsilon$.
\begin{lem}
\label{lem:prod F sign difference}There's some function $\epsilon_{\sfr}^{\rho\coprod r}:\ts_{\sfr}^{\rho\coprod r}\to\left\{ \pm1\right\} $
such that for every $\tc_{+}\in\ts_{\sfr}^{\rho\coprod r}$, $\tc=\cnt\tc_{+}$,
we have
\end{lem}
\begin{equation}
\acute{\mathcal{G}}_{\tc_{+}}\circ\partial^{\tc_{+}}\left(\prod\ff\right)_{\sfr}^{\rho}=\epsilon_{\sfr}^{\rho\coprod r}\left(\tc_{+}\right)\cdot\left[\left(\prod\acute{\ff}\right)_{\sfr}^{\rho\coprod r,\left\{ r\right\} }\otimes\alpha\right]\circ\mathcal{G}_{\sfr}^{\rho\coprod r}|_{\partial^{\tc_{+}}\mm_{\tc}}\label{eq:prod F sign diff}
\end{equation}

\begin{proof}
Consider some $\tc_{+}\in\ts_{\sfr}^{\rho\coprod r}$. Let $\tc_{+}^{!}=\tau.^{-1}\left(\tc_{+}\right)$
where $\tau\in\Sym\left(\rho\coprod r\right)$ is the cyclic shift
that sends $r$ to $\min\rho$. The four local system maps in (\ref{eq:prod F sign diff})
form the four outer edges of the following diagram
\begin{equation}
\xymatrix{\Or\left(T\partial^{\tc_{+}}\mm_{\tc}\right)\aru{r}{\mathcal{G}{}_{\tc_{+}^{!}}^{!}}\ar[d]^{\partial\left(\prod\ff\right)_{\sfr}^{\rho}} & \Or\left(T\mm_{\tc_{+}^{!}}\right)_{\otimes}\aru{r}{\Or\left(d\tau.\right)}\ar[d]^{\left[\left(\prod\acute{\ff}!\right)_{\sfr}^{\rho\coprod r,\left\{ r\right\} }\otimes\alpha\right]} & \Or\left(T\mm_{\tc_{+}}\right)_{\otimes}\ar[d]^{\left[\left(\prod\acute{\ff}\right)_{\sfr}^{\rho\coprod r,\left\{ r\right\} }\otimes\alpha\right]}\\
\Or\left(T\partial^{\tc_{+}}\check{\mm}_{\tc}\right)\ard{r}{\acute{\mathcal{G}}_{\tc_{+}^{!}}^{!}} & \Or\left(T\acute{\mm}_{\tc_{+}^{!}}^{\left\{ r\right\} }\right)_{\otimes}\ard{r}{\Or\left(d\tau.\right)} & \Or\left(T\acute{\mm}_{\sfr}^{\rho\coprod r,\left\{ r\right\} }\right)_{\otimes}
}
.\label{eq:G commutator in two steps}
\end{equation}
Here $\mathcal{G}{}_{\tc_{+}^{!}}^{!},\acute{\mathcal{G}}_{\tc_{+}^{!}}^{!}$
are induced from the cartesian squares for $g!_{\sfr}^{\rho}$ and
$\acute{g}!_{\tc_{+}}^{\left\{ r\right\} }$, respectively, and we
set 
\[
\Or\left(T\mm_{\sfr}^{\rho\coprod r}\right)_{\otimes}=\Or\left(T\mm_{\sfr}^{\rho\coprod r}\right)\otimes\ev_{r}^{-1}\Or\left(TL\right),
\]
and similarly for the other occurrences of subscript $\otimes$. The
map $\left(\prod\acute{\ff}!\right)_{\sfr}^{\rho\coprod r,\left\{ r\right\} }$
is defined similarly to $\left(\prod\acute{\ff}\right){}_{\sfr}^{\rho\coprod r,\left\{ r\right\} }$,
except we now treat $\sstar'_{r}$ as the first element of any $\sigma_{\tc}\left(v\right)$
containing it, so there's no sign twist in this case (compare to
(\ref{eq:F ls decomposition})). The compositions of the top (respectively,
bottom) row of arrows in this diagram give $\mathcal{G}_{\tc_{+}}$
(resp., $\acute{\mathcal{G}}_{\tc_{+}}$), and we reduce to showing
the two squares commute up to signs which depend only on $\tc_{+}$.

Consider the left square. We have
\begin{equation}
\iota_{\mm_{\sfr}^{\rho}}^{\partial}|_{\mm_{\sfr_{\tc}\left(v_{1}\right)}\times\cdots\times\partial\mm_{\sfr_{\tc}\left(v_{i}\right)}\times\cdots}=\left(-1\right)^{\sum_{j<i}\dim\mm_{\sfr_{\tc}\left(v_{j}\right)}}\cdot\id\otimes\iota_{\mm_{\sfr_{\tc}}\left(v_{i}\right)}^{\partial}\otimes\id\label{eq:sign ON shift}
\end{equation}
 and similarly for $\check{\mm}_{\sfr}^{\rho}$ in place of $\mm_{\sfr}^{\rho}$,
so that
\[
\partial\left(\prod\ff\right){}_{\sfr}^{\rho}|_{\mm_{\sfr_{\tc}\left(v_{1}\right)}\times\cdots\times\partial\mm_{\sfr_{\tc}\left(v_{i}\right)}\times\cdots}=\left(-1\right)^{\bullet}\ff_{\sfr_{\tc}\left(v_{1}\right)}\times\cdots\times\partial\ff_{\sfr_{\tc}\left(v_{i}\right)}\times\cdots\times\ff_{\sfr_{\tc}\left(v_{r+1}\right)}
\]
for 
\begin{equation}
\bullet=\sum_{j<i}\dim\mm_{\sfr_{\tc}\left(v_{j}\right)}=\sum_{j<i}\left|\sigma_{\tc}\left(v_{j}\right)\right|,\label{eq:sign of leibnitz}
\end{equation}
which depends only on $\tc_{+}$. Thus, we reduce to computing the
sign for $\rho=\emptyset$. So let $\tc_{+}^{!}$ be some tree with
a single edge labeled by $r$, and $\tc$ be the unique $\left(\sfr,\emptyset\right)$-labeled
tree. Then the oriented base for $T\partial^{\tc_{+}}\mm_{\tc}$ picked
out by going up then left starting from $\Or\left(T\acute{\mm}_{\tc_{+}^{!}}^{\left\{ r\right\} }\right)_{\otimes}$
at the bottom center in (\ref{eq:G commutator in two steps}) is defined
by the equation
\[
\partial^{\tc_{+}}\mm_{\tc}\times L=\mm_{\left(\left(\kf',\lf',\beta'\right),\sstar'_{r}\right)}\times\left(\partial D^{2}\right)^{\sigma'\backslash\sstar'_{r}}\times\check{\mm}_{\sfr''}\times\left(\partial D^{2}\right)^{\sigma''}
\]
Where we let $\xx$ stand for a local oriented base for $T\xx$. On
the other hand, if we go left then up, we obtain an oriented base
defined by the equation
\[
\partial_{-}\mm_{\sfr}^{\emptyset}\times L=\mm_{\left(\left(\kf',\lf',\beta'\right),\sstar'_{r}\right)}\times\check{\mm}_{\sfr''}\times\left(\partial D^{2}\right)^{\sigma}.
\]
Since $\kf''+\beta''=1\mod2$ we find that $\dim\check{\mm}_{\sfr''}=0\mod2$
so the sign discrepancy for the square is simply 
\begin{equation}
\sgn\left(\sigma'\backslash\sstar'_{r},\sigma''\right)\label{eq:sign of shuffle}
\end{equation}
where $\left(\sigma'\backslash\sstar'_{r},\sigma''\right)$ denotes
the shuffle permutation $\left(\sigma'\backslash\sstar'_{r},\sigma''\right)\to\sigma$,
which depends only on $\tc_{+}$.

It is not hard to check that the right square of (\ref{eq:G commutator in two steps})
also  depends only on $\tc_{+}$. We will not need the explicit
formula for this sign.
\end{proof}
\begin{proof}
[Proof of Proposition \ref{prop:F definition and coherence}]We define
$\delta_{\sfr}^{\rho}$ recursively. We set $\delta_{\sfr}^{\emptyset}\equiv1$.
Suppose $\delta_{\sfr}^{\rho}$ is given. Using (\ref{eq:tree contraction})
we can write
\[
\delta_{\sfr}^{\rho}\circ\pi_{\sfr}^{\rho}\circ i_{\mm_{\sfr}^{\hat{\rho}}}^{\partial}=\upsilon\circ\pi_{\sfr}^{\rho\coprod r}\circ g_{\sfr}^{\rho\coprod r}
\]
for some $\upsilon:\ts_{\sfr}^{\rho\coprod r}\to\left\{ \pm1\right\} $.
We then compute, using Lemma \ref{lem:prod F sign difference}, for
every $\tc_{+}\in\ts_{\sfr}^{\rho\coprod r}$ and $\tc=\cnt\tc_{+}$,

\begin{multline*}
\check{\gc}_{\tc_{+}}\circ\partial\ff_{\sfr}^{\rho}=\left(\delta_{\sfr}^{\rho}\left(\tc\right)\right)\cdot\check{\gc}_{\tc_{+}}\circ\partial\left(\prod\ff\right)_{\sfr}^{\rho}=\\
=\left(\upsilon\left(\tc_{+}\right)\right)\cdot\check{\mathcal{G}}_{\tc_{+}}\circ\partial\left(\prod\ff\right)_{\sfr}^{\rho}=\\
=\left(\left(\upsilon\cdot\epsilon_{\sfr}^{\rho\coprod r}\right)\left(\tc_{+}\right)\right)\cdot\left(\left(\prod\grave{\ff}\right)_{\sfr}^{\rho\coprod r}\otimes\alpha\right)\circ\left(\left(\prod\acute{\ff}\right){}_{\sfr}^{\rho\coprod r}\otimes\alpha\right)\circ\mathcal{G}_{\sfr}^{\rho\coprod r}=\\
=\left[\left(\left(\upsilon\cdot\epsilon_{\sfr}^{\rho\coprod r}\right)\left(\tc_{+}\right)\right)\cdot\left(\left(\prod\ff\right){}_{\sfr}^{\rho_{+}}\otimes\alpha\right)\right]\circ\mathcal{G}_{\sfr}^{\rho\coprod r},
\end{multline*}
so if we take $\delta_{\sfr}^{\rho\coprod r}=\upsilon\cdot\epsilon_{\sfr}^{\rho\coprod r}$
Eq (\ref{eq:F coherence}) holds for $\rho\coprod r$.
\end{proof}

\begin{prop}
\label{prop:check J def and coherence}There exists some function
$\zeta_{\sfr}^{\rho}:\ts_{\sfr}^{\rho}\to\left\{ \pm1\right\} $ so
that, setting
\[
\check{\mathcal{J}}_{\sfr}^{\rho}:=\left(\zeta_{\sfr}^{\rho}\circ\pi_{\sfr}^{\rho}\right)\cdot\left(\prod\check{\mathcal{J}}\right)_{\sfr}^{\rho},
\]
we have, for $\tc_{+}\in\ts_{\sfr}^{\rho\coprod r},\tc=\cnt\tc_{+}$,
\begin{equation}
\left(\check{\mathcal{J}}_{\sfr}^{\rho\coprod r}\right)^{\to}\circ\check{\gc}_{\tc_{+}}=\check{\mathcal{J}}_{\sfr}^{\rho}\circ\iota_{\check{\mm}_{\sfr}^{\rho}}^{\partial}.\label{eq:check J coherence}
\end{equation}
\end{prop}
\begin{proof}
This is similar to the proof of Proposition \ref{prop:F definition and coherence}
above, except we use Lemma \ref{lem:prod J sign difference} below,
which computes the failure of coherence for $\prod\check{\jc}$.
\end{proof}
\begin{cor}
\label{cor:J coherence}$\mathcal{J}_{\sfr}^{\rho}:=\check{\jc}_{\sfr}^{\rho}\circ\ff_{\sfr}^{\rho}$
satisfies \emph{(coherence)}:
\[
\left(\mathcal{J}_{\sfr}^{\rho\coprod r}\right)^{\to}\circ\gc_{\sfr}^{\rho\coprod r}=\mathcal{J}_{\sfr}^{\rho}\circ\iota_{\mm_{\sfr}^{\rho}}^{\partial}.
\]
\end{cor}
\begin{proof}
Consider some $\tc_{+}\in\ts_{\sfr}^{\rho\coprod r}$ and $\tc=\cnt\tc_{+}$.
Restricting our attention to $\partial^{\tc_{+}}:=\partial^{\tc_{+}}\mm_{\tc}$
we have
\begin{multline*}
\left(\mathcal{J}_{\sfr}^{\rho\coprod r}\right)^{\to}\circ\gc_{\sfr}^{\rho\coprod r}|_{\partial^{\tc_{+}}\mm_{\tc}}=\left(\check{\mathcal{J}}_{\sfr}^{\rho\coprod r}\right)^{\to}\circ\left(\ff_{\sfr}^{\rho\coprod r}\otimes\alpha\right)\circ\gc_{\sfr}^{\rho\coprod r}|_{\partial^{\tc_{+}}}=\\
=\left(\check{\mathcal{J}}_{\sfr}^{\rho\coprod r}\right)^{\to}\circ\left(\grave{\ff}_{\sfr}^{\rho\coprod r,\left\{ r\right\} }\otimes\alpha\right)\circ\left(\acute{\ff}_{\sfr}^{\rho\coprod r,\left\{ r\right\} }\otimes\alpha\right)\circ\gc_{\sfr}^{\rho\coprod r}|_{\partial^{\tc_{+}}}=\\
=\left(\check{\mathcal{J}}_{\sfr}^{\rho\coprod r}\right)^{\to}\circ\left(\grave{\ff}_{\sfr}^{\rho\coprod r,\left\{ r\right\} }\otimes\alpha\right)\circ\acute{\gc}_{\tc_{+}}\circ\partial\ff_{\sfr}^{\rho}|_{\partial^{\tc_{+}}}=\\
\check{\mathcal{J}}_{\sfr}^{\rho}\circ\iota_{\check{\mm}_{\sfr}^{\rho}}^{\partial}\circ\partial\ff_{\sfr}^{\rho}|_{\partial^{\tc_{+}}}=\check{\mathcal{J}}_{\sfr}^{\rho}\circ\partial\ff_{\sfr}^{\rho}\circ\iota_{\mm_{\sfr}^{\rho}}^{\partial}=\jc_{\sfr}^{\rho}\circ\iota_{\mm_{\sfr}^{\rho}}^{\partial}.
\end{multline*}
\end{proof}
\begin{lem}
\label{lem:prod J sign difference}There exists a function $\eta_{\sfr}^{\rho\coprod r}:\ts_{\sfr}^{\rho\coprod r}\to\left\{ \pm1\right\} $
such that 
\[
\left(\prod\check{\jc}\right)_{\sfr}^{\rho}\circ\iota_{\check{\mm}_{\sfr}^{\rho}}^{\partial}|_{\partial^{\tc_{+}}\check{\mm}_{\tc}}=\left(\eta_{\sfr}^{\rho\coprod r}\left(\tc_{+}\right)\right)\left(\left(\prod\check{\jc}\right)_{\sfr}^{\rho\coprod r}\right)^{\to}\circ\check{\gc}_{\tc_{+}}
\]
for every $\tc_{+}\in\ts_{\sfr}^{\rho\coprod r},\tc=\cnt\tc_{+}$.
\end{lem}
\begin{proof}
Consider the diagram
\begin{equation}
\xymatrix{\Or\left(T\partial^{\tc_{+}}\check{\mm}_{\tc}\right)\aru{r}{\acute{\gc}_{\tc_{+}^{!}}^{!}}\ar[d]_{\iota_{\check{\mm}_{\sfr}^{\rho}}^{\partial}} & \Or\left(T\acute{\mm}_{\tc_{+}^{!}}^{\left\{ r\right\} }\right)_{\otimes}\aru{r}{\Or\left(d\tau.\right)}\ar[d]^{\left(\prod\grave{\ff}!\right)_{\sfr}^{\rho\coprod r,\left\{ r\right\} }} & \Or\left(T\acute{\mm}_{\sfr}^{\rho\coprod r,\left\{ r\right\} }\right)_{\otimes}\ar[d]^{\left(\prod\grave{\ff}\right)_{\sfr}^{\rho\coprod r,\left\{ r\right\} }}\\
\Or\left(T\check{\mm}_{\sfr}^{\rho}\right)\ar[dr]_{\left(\prod\check{\jc}\right)_{\sfr}^{\rho}} & \ar[d]\Or\left(T\check{\mm}_{\sfr}^{\rho\coprod r}\right)_{\otimes}\ar[r]^{\Or\left(d\tau.\right)}\ar[d] & \Or\left(T\check{\mm}_{\sfr}^{\rho\coprod r}\right)_{\otimes}\ar[dl]^{\left(\prod\check{\jc}\right)_{\sfr}^{\rho}}\\
 & \Or\left(TL\right)^{\boxtimes\left(\kf\coprod\sstar''_{\rho}\right)}
}
\label{eq:diagram for check J commutator}
\end{equation}
where $\left(\prod\grave{\ff}!\right)_{\sfr}^{\rho\coprod r,\left\{ r\right\} }$
is defined similarly to $\left(\prod\grave{\ff}\right)_{\sfr}^{\rho\coprod r,\left\{ r\right\} }$,
with no sign twist, and we have
\[
\left(\prod\ff\right)_{\sfr}^{\rho\coprod r}=\left(\prod\grave{\ff}!\right)_{\sfr}^{\rho\coprod r,\left\{ r\right\} }\circ\left(\prod\acute{\ff}!\right)_{\sfr}^{\rho\coprod r,\left\{ r\right\} }.
\]
Using (\ref{eq:sign ON shift}) we reduce the computation of the sign
of the pentagon on the left to the case $\rho=\emptyset$; we may
assume $\sfr=\left(\left(\kf,\lf,\beta\right),\emptyset\right)$ for
$\kf=\left(1,...,k\right)$, so we're working over 
\[
\partial\check{\mm}_{\sfr}\to\mm_{\sfr'}\times\mm_{\sfr''}\subset\acute{\mm}_{\sfr}^{\left(1\right)}
\]
where $\sfr'=\left(\left(\kf',\lf',\beta'\right),\left\{ \sstar'_{1}\right\} \right)$
and $\sfr''=\left(\left(\kf'',\lf'',\beta''\right),\emptyset\right)$.
Using (\ref{eq:J Sym(k)-invariance}) for both the domain and codomain
of the map we may further assume (without effecting the sign) that
$\kf'=\left(1,...,k'\right)$ and $\kf''=\left(\sstar''_{1},k'+1,...,k\right)$
and that the image in $\mm_{\sfr'}\times\mm_{\sfr''}$ is represented
by 
\[
\left(\Sigma'=D^{2},\nu'=\id,\kappa',\lambda',u'\right),\left(\Sigma''=D^{2},\nu''=\id,\kappa'',\lambda'',u''\right)
\]
so that $\kappa'$ and $\kappa''$ both preserve the cyclic order.
We can also use the same initial orientation for $\mm_{\sfr'}$ and
$\mm_{\sfr}$ (both of the orientation transports begin at $1\in\underline{\kf'}\subset\underline{\kf}$;
see $\S$\ref{subsec:J Construction}), and continue the orientation
transport of $\mm_{\sfr'}$ to $\sstar'_{1}$ and use the orientation
of $\left(u'_{\sstar'_{1}}\right)^{-1}TL=\left(u''_{\sstar''_{1}}\right)^{-1}TL$
as the initial orientation for the orientation transport of $\mm_{\sfr''}$.
In particular, the orientations for the real boundary conditions $u|_{\partial\Sigma'}^{-1}TL$,
$u|_{\partial\Sigma''}^{-1}TL$ and $u|_{\partial\Sigma}^{-1}TL$,
if orientable, are compatible, so that by \cite{jake-wdvv} the corresponding
orientations for 
\[
T\left(\widetilde{\mm}\left(\beta'\right)\fibp{u\left(+1\right)}{u\left(-1\right)}\widetilde{\mm}\left(\beta''\right)\right),T\partial\widetilde{\mm}\left(\beta\right)
\]
differ by the sign 
\[
W_{m}\left(\beta',\beta''\right)=\begin{cases}
\left(-1\right)^{\beta'\beta''} & m=1\mod2\\
+1 & m=0\mod2
\end{cases}
\]
(recall the parity of $m$ affects whether $L=\rr\pp^{2m}$ is $Pin^{\pm}$).

From here the computation is a variation on the proof of \cite[Proposition 8.3.3]{FOOO},
and we try to use similar notation. We assume that $\delta'=\delta''=\emptyset$
and $\lf=\emptyset$, the computation in the other cases being similar
and giving the same result. We use subscript $L$ to denote the place
where we \emph{remove} (the pullback along $\ev_{\sstar'_{1}}=\ev_{\sstar''_{1}}$
of) an oriented base for $TL$ (this is picked out by the orientation
transport as discussed above; since $\dim L=2m$ we can shift this
without an additional sign). We let $\rr_{\beta'},\rr_{\beta''},\rr_{\beta'+\beta''}$
denote copies of the 1-parameter subgroup of $PSL_{2}\left(\rr\right)$
fixing $\pm1\in\partial D^{2}$. $\rr_{out}$ stands for the outward
normal boundary vector, whose positive generator is the image of 
\[
\left(1,-1\right)\in\mathfrak{lie}\left(\rr_{\beta'}\times\rr_{\beta''}\right)=\rr_{\beta'}\oplus\rr_{\beta''},
\]
see \cite{FOOO}. With this in place, the computation runs as follows.
\begin{align*}
 & \rr_{out}\times\mm_{\sfr'}\times_{L}\mm_{\sfr''}\times\rr_{\beta'+\beta''}=\\
=\left(-1\right)^{d'} & \left(\mm_{\sfr'}\times\rr_{\beta'}\right)\times_{L}\left(\mm_{\sfr''}\times\rr_{\beta''}\right)\\
=\left(-1\right)^{d'} & \left[\widetilde{\mm}\left(\beta'\right)\times\left(\partial D^{2}\right)^{\left(\widehat{1}^{+1},\widehat{2}^{-1},3,...,k',\sstar'_{r}\right)}\times\widetilde{\mm}\left(\beta''\right)\times\left(\partial D^{2}\right)^{\widehat{\kf''}}\right]_{L}\\
=\left(-1\right)^{d'+k'} & \left[\widetilde{\mm}\left(\beta'\right)\times\left(\partial D^{2}\right)^{\left(\widehat{1}^{-1},2,...,k',\widehat{\sstar'_{r}}^{+1}\right)}\times\widetilde{\mm}\left(\beta''\right)\times\left(\partial D^{2}\right)^{\widehat{\kf''}}\right]_{L}\\
=\left(-1\right)^{d'+k'+\tilde{d}''\cdot\left(k'-1\right)} & \left[\widetilde{\mm}\left(\beta'\right)\times\widetilde{\mm}\left(\beta''\right)\times\left(\partial D^{2}\right)^{\left(\widehat{1}^{-1},2,...,k',\widehat{k'+1}^{+1},k'+2,...,k\right)}\right]_{L}\\
=\left(-1\right)^{d'+\tilde{d}''\cdot\left(k'-1\right)} & \widetilde{\mm}\left(\beta'\right)\times_{L}\widetilde{\mm}\left(\beta''\right)\times\left(\partial D^{2}\right)^{\left(\widehat{1}^{+1},\widehat{2}^{-1},...,k\right)}\\
=\left(-1\right)^{d'+\tilde{d}''\left(k'-1\right)}\cdot W_{m}\left(\beta',\beta''\right) & \widetilde{\mm}\left(\beta'+\beta''\right)\times\left(\partial D^{2}\right)^{\left(\hat{1},\hat{2},...,k\right)}\\
=\left(-1\right)^{d'+\tilde{d}''\left(k'-1\right)}\cdot W_{m}\left(\beta',\beta''\right) & \mm_{\sfr}\times\rr_{\beta}
\end{align*}
This should be read as a string of equalities of oriented bases for
the tangent space at a point of ${\widetilde{\mm}\left(\beta'\right)\times_{L}\widetilde{\mm}\left(\beta'\right)\times\left(\partial D^{2}\right)^{k-2}}$.
For instance, $\mm_{\sfr'}$ denotes the pullback of an oriented base
for $\mm_{\sfr'}$ along the map 
\begin{align*}
q_{1,2}:\widetilde{\mm}\left(\beta'\right)\times\left(\partial D^{2}\right)^{\left(k'+1\right)-2} & \to\mm_{\sfr'}\\
\left(u',z_{1},...,z_{k'-1}\right) & \mapsto\left[D^{2},\emptyset,\kappa'=\left(+1,-1,z_{1},...,z_{k'-1}\right),\lambda'=\emptyset,u'\right]
\end{align*}
On the fourth line, we switch to a different point of $\widetilde{\mm}\left(\beta'\right)\times\left(\partial D^{2}\right)^{\left(k'+1\right)-2}$,
mapping to the same point of the moduli space but now using the map
\[
q_{1,k'+1}:\left(\tilde{u}',\tilde{z}_{1},...,\tilde{z}_{k'-1}\right)\mapsto\left[D^{2},\emptyset,\kappa'=\left(-1,\tilde{z}_{1},...,\tilde{z}_{k'-1},+1\right),\lambda'=\emptyset,\tilde{u}'\right].
\]
Consider $g\in PSL\left(2,\rr\right)$ which takes $\left(+1,-1,z_{1},...,z_{k'-1}\right)$
to $\left(-1,\tilde{z}_{1},...,\tilde{z}_{k'-1},+1\right)$. The action
of $g$ preserves the orientation of $\widetilde{\mm}\left(\beta'\right)$
by \cite[Proposition 2.8]{JT}. It is easy to see that the differential
of the action of $g$ is in the same connected component of $GL\left(k'-1,\rr\right)$
as the map 
\[
\begin{pmatrix} &  &  & -1\\
1 &  &  & -1\\
 & \ddots &  & \vdots\\
 &  & 1 & -1
\end{pmatrix}
\]
which has determinant $\left(-1\right)^{k'}$. A similar change occurs
on the sixth line, as indicated by the shifting of hats $\widehat{}$.
Finally, we have

\begin{align*}
d':=\dim\mm_{\sfr'} & =2m+\left(2m+1\right)\cdot\beta'-3+\left(k'+1\right)+2l'\overset{\mod2}{\equiv}\left(k'+1\right)+\beta'+1\equiv1\\
\tilde{d}'':=\dim\widetilde{\mm}\left(\beta''\right) & =\left(2m+1\right)\beta''\overset{\mod2}{\equiv}\beta'',
\end{align*}
so in sum, for $\hat{\rho}=\emptyset$ the left pentagon in (\ref{eq:diagram for check J commutator})
commutes up to 
\begin{equation}
\left(-1\right)^{k'+\beta''\left(k'-1\right)}W_{m}\left(\beta',\beta''\right)=\left(-1\right)^{k'+\beta'\cdot\beta''}W_{m}\left(\beta',\beta''\right)=\left(-1\right)^{k'+\left(1+m\right)\beta'\beta''},\label{eq:=00005Bcheck J, G!=00005D}
\end{equation}
which clearly factors through $\pi_{\sfr}^{\left(1\right)}$.

The square in (\ref{eq:diagram for check J commutator}) commutes
(with $+1$ sign), since the moduli factors appearing in $\check{\mm}_{\sfr}^{\rho\coprod r}$
are all even dimensional, so there's at most one odd dimensional moduli
factor in the spaces lying above them, $\mm_{\tc_{+}^{!}}^{!}$ and
$\acute{\mm}_{\sfr}^{\rho\coprod r,\left\{ r\right\} }$. The triangle
commutes up to a sign which factors through $\pi_{\sfr}^{\rho}\circ g_{\sfr}^{\rho}$
(we will not need an explicit formula for it).
\end{proof}

\subsubsection{$\Sym\left(\rho\right)$ equivariance}

\begin{prop}
For $\tau\in\Sym\left(\rho\right)$ we have
\begin{align}
\jc_{\sfr}^{\rho}\circ\Or\left(d\tau.\right) & =\sgn\left(\tau\right)\cdot\jc_{\sfr}^{\rho}\label{eq:J_s^rho equivariance}\\
\check{\jc}_{\sfr}^{\rho}\circ\Or\left(d\tau.\right) & =\sgn\left(\tau\right)\cdot\check{\jc}_{\sfr}^{\rho}\label{eq:check J equivariance}\\
\ff_{\sfr}^{\rho}\circ\Or\left(d\tau.\right) & =\Or\left(d\tau.\right)\circ\ff_{\sfr}^{\rho}\label{eq:ff equivariance}
\end{align}
\end{prop}
\begin{proof}
We prove (\ref{eq:J_s^rho equivariance}). Assume without loss of
generality $\rho=\left\{ 1,...,r\right\} $. Let 
\[
\gc_{\sfr}^{\left[r\right],r}:\Or\left(T\partial_{-}^{r}\mm_{\sfr}^{\emptyset}\right)\to\Or\left(T\mm_{\sfr}^{\left\{ 1,...,r\right\} }\right)\otimes\bigotimes_{i=1}^{r}\ev_{\sstar''_{i}}^{-1}\Or\left(TL\right)
\]
be the local system map over $g_{\sfr}^{\left[r\right],r}$ associated
with the cartesian square (\ref{eq:bd^r as pullback}). If we equip
the domain with the $\Sym\left(r\right)$-action that permutes boundary
faces, we have 
\begin{equation}
\gc_{\sfr}^{\left[r\right],r}\circ\Or\left(d\tau.\right)=\Or\left(d\tau.\right)\circ\gc_{\sfr}^{\left[r\right],r}.\label{eq:G^=00005Br=00005D,r equivariance}
\end{equation}

Iteratively applying coherence we find that 
\begin{equation}
\left(\mathcal{J}_{\sfr}^{\left(1,...,r\right)}\right)^{\twoheadrightarrow}\circ\gc_{\sfr}^{\left[r\right],r}=\left(-1\right)^{\binom{r}{2}}\jc_{\sfr}^{\emptyset}\circ\iota_{\mm_{\sfr}^{\emptyset}}^{\partial^{r}}.\label{eq:iterated J coherence}
\end{equation}
Here the sign $\left(-1\right)^{\binom{r}{2}}$ comes from 
\[
\gc_{\sfr}^{\left[r\right],r}=\left(-1\right)^{\binom{r}{2}}\gc_{\sfr}^{\left(1,...,r\right)}\circ\partial\gc_{\sfr}^{\left(1,...,r-1\right)}\circ\cdots\circ\partial^{r-1}\gc_{\sfr}^{\left(1\right)}
\]
(the cartesian natural local system maps between pullbacks of $\Or\left(TL\right)$
are suppressed in this equation), which reflects a reversal of the
order of the outward normal vectors; at any rate all we need is that
this sign factor is constant on $\mm_{\sfr}^{\left[r\right]}$. Clearly,
$\jc_{\sfr}^{\emptyset}$ is $\Sym\left(r\right)$ invariant and $\iota_{\mm_{\sfr}^{\emptyset}}^{\partial^{r}}\circ\Or\left(d\tau.\right)=\sgn\left(\tau\right)\iota_{\mm_{\sfr}^{\emptyset}}^{\partial^{r}}$.
Combining this with (\ref{eq:iterated J coherence}) and (\ref{eq:G^=00005Br=00005D,r equivariance})
we deduce that (\ref{eq:J_s^rho equivariance}) holds on the image
of $g_{\sfr}^{\left[r\right],r}$. By the long exact sequence of
a fibration, $g_{\sfr}^{\left[r\right],r}$ visits every connected
component of $\mm_{\sfr}^{\left(1,...,r\right)}$, so (\ref{eq:J_s^rho equivariance})
holds everywhere and the proof of (\ref{eq:J_s^rho equivariance})
is complete. The proof of (\ref{eq:check J equivariance}) is similar,
and (\ref{eq:ff equivariance}) follows directly from (\ref{eq:J_s^rho equivariance})
and (\ref{eq:check J equivariance}).
\end{proof}

\subsubsection{Explicit signs for odd-even trees.}

We conclude the discussion of orientations with a result that will
be used in deriving the explicit formula for the fixed point contributions
in \cite{fp-loc-OGW}.

First we state a general lemma which aids the computation of $\jc$.
For a labeled tree $\tc\in\ts_{\basic}^{\left[r\right]}=\ts_{\basic}^{\left(1,...,r\right)}$,
we have 
\[
\jc_{\basic}^{\left[r\right]}|_{\mm_{\tc}}=\theta\left(\tc\right)\cdot\left(\prod\jc\right){}_{\basic}^{\left[r\right]}
\]
for $\theta\left(\tc\right)=\delta\left(\tc\right)\cdot\zeta\left(\tc\right)$
(see Propositions \ref{prop:F definition and coherence} and \ref{prop:check J def and coherence}).
The \emph{smoothing sequence }of $\tc$ is
\[
\tc=\tc^{\left(0\right)},\tc^{\left(1\right)},...,\tc^{\left(r\right)}
\]
where, for $1\leq j\leq r$, $\tc^{\left(j\right)}\in\ts_{\basic}^{\left(j+1,...,r\right)}$
is obtained from $\tc^{\left(j-1\right)}$ by contracting the edge
$j$; more precisely it is uniquely specified by requiring 
\[
\partial\mm_{\tc^{\left(j\right)}}\subset\left(g_{\basic}^{\left(j,...,r\right)}!\right)^{-1}\left(\mm_{\tc^{\left(j-1\right)}}\right).
\]

\begin{lem}
\label{lem:smoothing sign difference}We have 
\[
\theta\left(\tc\right)=\prod_{a=1}^{r}\xi\left(\tc^{\left(a-1\right)}\right)
\]
and 
\[
\zeta\left(\tc\right)=\prod_{a=1}^{r}\check{\xi}\left(\tc^{\left(a-1\right)}\right)
\]
where for $1\leq a\leq r$, $\xi\left(\tc^{\left(a-1\right)}\right),\check{\xi}\left(\tc^{\left(a-1\right)}\right)$
are computed as follows. We denote by $v',v''$ the tail and head,
respectively, of the edge $a$ in $\tc^{\left(a-1\right)}$, and let
\[
k'=\left|\kf_{\tc^{\left(a-1\right)}}\left(v'\right)\right|,\beta'=\beta_{\tc^{\left(a-1\right)}}\left(v'\right),\beta''=\beta_{\tc^{\left(a-1\right)}}\left(v''\right)
\]
Let $\tc_{0}^{\left(a\right)}=\left(v_{1},...,v_{r+1-a}\right)$ and
let $v_{i}$ be the vertex obtained from contracting the edge $a$.
We have 
\begin{align*}
\xi_{\tc}\left(\tc^{\left(a-1\right)}\right) & =\left(-1\right)^{r-a}\cdot\left(-1\right)^{k'+\left(1+m\right)\beta'\beta''}\cdot\left(-1\right)^{\sum_{j<i}\left|\sigma_{\tc^{\left(a\right)}}\left(v_{j}\right)\right|}\cdot\sgn\left(\sigma'\backslash\sstar'_{a},\sigma''\right)\\
\check{\xi}_{\tc}\left(\tc^{\left(a-1\right)}\right) & =\left(-1\right)^{r-a}\cdot\left(-1\right)^{k'+\left(1+m\right)\beta'\beta''}
\end{align*}
\end{lem}
\begin{proof}
We consider $\xi_{\tc}$ first. We have $\theta\left(\tc^{\left(r\right)}\right)=1$,
so it suffices to prove that $\theta\left(\tc^{\left(a-1\right)}\right)=\xi\left(\tc^{\left(a\right)}\right)\cdot\theta\left(\tc^{\left(a\right)}\right)$.
Without loss of generality, assume $a=1$. Let $\tau:\left(1,...,r\right)\mapsto\left(r,1,...,r-1\right)$
be the label-preserving bijection. We have
\begin{multline*}
\left(-1\right)^{r-1}\left(\jc_{\basic}^{\left[r\right]}\right)^{\to}\circ\acute{\gc}_{\tc^{\left(0\right)!}}^{!}=\left(\jc_{\basic}^{\left[r\right]}\right)^{\to}\circ\Or\left(d\tau.\right)\circ\acute{\gc}_{\tc^{\left(0\right)!}}^{!}=\\
=\left(\jc_{\basic}^{\left[r\right]}\right)^{\to}\circ\gc_{\basic}^{\left[r\right]}=\jc_{\basic}^{\left[r-1\right]}\circ\iota_{\mm_{\basic}^{\left[r-1\right]}}^{\partial}=\\
=\theta\left(\tc^{\left(1\right)}\right)\cdot\left(\prod\jc\right)_{\basic}^{\left[r-1\right]}\circ\iota_{\mm_{\basic}^{\left[r-1\right]}}^{\partial}=\left(-1\right)^{\blacklozenge}\theta\left(\tc^{\left(1\right)}\right)\cdot\left(\left(\prod\jc\right)_{\basic}^{\left[r\right]}\right)^{\to}\circ\acute{\gc}_{\tc^{\left(0\right)!}}^{!}
\end{multline*}
where the sign
\[
\left(-1\right)^{\blacklozenge}=\left(-1\right)^{k'+\left(1+m\right)\beta'\beta''}\cdot\left(-1\right)^{\sum_{j<i}\left|\sigma_{\tc^{\left(a\right)}}\left(v_{j}\right)\right|}\cdot\sgn\left(\sigma'\backslash\sstar'_{r},\sigma''\right)
\]
is given by (\ref{eq:sign of leibnitz},\ref{eq:sign of shuffle},\ref{eq:=00005Bcheck J, G!=00005D})
(this is the product of signs of the left square in (\ref{eq:G commutator in two steps})
and the left trapezoid in (\ref{eq:diagram for check J commutator})).
The result follows.

The computation of $\check{\xi}_{\tc}$ is similar:
\begin{multline*}
\left(-1\right)^{r-1}\left(\check{\jc}_{\basic}^{\left[r\right]}\right)^{\to}\circ\check{\gc}{}_{\tc^{\left(0\right)!}}^{!}=\left(\check{\jc}_{\basic}^{\left[r\right]}\right)^{\to}\circ\Or\left(d\tau.\right)\circ\check{\gc}{}_{\tc^{\left(0\right)!}}^{!}=\\
=\left(\check{\jc}_{\basic}^{\left[r\right]}\right)^{\to}\circ\check{\gc}{}_{\tc^{\left(0\right)}}=\check{\jc}_{\basic}^{\left[r-1\right]}\circ\iota_{\mm_{\basic}^{\left[r-1\right]}}^{\partial}=\\
=\zeta\left(\tc^{\left(1\right)}\right)\cdot\left(\prod\check{\jc}\right)_{\basic}^{\left[r-1\right]}\circ\iota_{\mm_{\basic}^{\left[r-1\right]}}^{\partial}=\left(-1\right)^{\check{\blacklozenge}}\zeta\left(\tc^{\left(1\right)}\right)\cdot\left(\left(\prod\check{\jc}\right)_{\basic}^{\left[r\right]}\right)^{\to}\circ\check{\gc}{}_{\tc^{\left(0\right)!}}^{!}.
\end{multline*}
Here $\check{\gc}{}_{\tc^{\left(0\right)!}}^{!}:=\Or\left(d\tau.\right)^{-1}\circ\check{\gc}{}_{\tc^{\left(0\right)}}$
(justifying the equality between the first and second line above).
The sign $\left(-1\right)^{\check{\blacklozenge}}$ is the sign of
the parallelogram on the bottom left of the following diagram (compare
(\ref{eq:diagram for check J commutator}))
\[
\xymatrix{\Or\left(T\partial^{\tc_{+}}\check{\mm}_{\tc}\right)\ar[dr]_{\check{\gc}_{\tc_{+}^{!}}^{!}}\aru{r}{\acute{\gc}_{\tc_{+}^{!}}^{!}}\ar[d]_{\iota^{\partial}}\ar@{-->}[drr] & \Or\left(T\acute{\mm}_{\tc_{+}^{!}}^{\left\{ r\right\} }\right)_{\otimes}\aru{r}{\Or\left(d\tau.\right)}\ar[d]^{\left(\prod\grave{\ff}!\right)_{\sfr}^{\rho\coprod r,\left\{ r\right\} }} & \Or\left(T\acute{\mm}_{\sfr}^{\rho\coprod r,\left\{ r\right\} }\right)_{\otimes}\ar[d]^{\left(\prod\grave{\ff}\right)_{\sfr}^{\rho}}\\
\Or\left(T\check{\mm}_{\sfr}^{\hat{\rho}}\right)\ar[dr]_{\left(\prod\check{\jc}\right)_{\sfr}^{\hat{\rho}}} & \Or\left(T\check{\mm}_{\sfr}^{\rho\coprod r}\right)_{\otimes}\ar[d]\ar[r]^{\Or\left(d\tau.\right)} & \Or\left(T\check{\mm}_{\sfr}^{\rho\coprod r}\right)_{\otimes}\ar[dl]^{\left(\prod\check{\jc}\right)_{\sfr}^{\rho}}\\
 & \Or\left(TL\right)^{\boxtimes\left(\kf\coprod\sstar''_{\rho}\right)}
}
\]
The dashed arrow is $\check{\gc}_{\tc_{+}}$, which by definition
is the map that makes the triangle directly above it commute. By the
definition of $\check{\gc}_{\tc_{+}^{!}}^{!}$, the triangle directly
below this map also commutes. We've argued above that the square commutes,
so $\check{\gc}{}_{\tc_{+}^{!}}^{!}=\left(\prod\grave{\ff}!\right)_{\sfr}^{\rho\coprod r,\left\{ r\right\} }\circ\acute{\gc}_{\tc_{+}^{!}}^{!}$,
so that 
\[
\left(-1\right)^{\check{\blacklozenge}}=\left(-1\right)^{k'+\left(1+m\right)\beta'\beta''}
\]
by (\ref{eq:=00005Bcheck J, G!=00005D}).
\end{proof}
\begin{defn}
(a) A labeled tree $\tc\in\ts_{\basic}^{\left[r\right]}$ will be
called an\emph{ odd-even tree }if all the moduli specifications $\left(\left(\kf,\lf,\beta\right),\sigma\right)=\sfr_{\tc}\left(v\right)$
for a vertex $v\in\tc_{0}$ satisfy the following condition: if $\beta=0\mod2$
then $\sigma=\emptyset$ and if $\beta=1\mod2$ then $\kf=\emptyset$
(in particular this means the tree is bipartite with respect to the
partition into odd degree and even degree vertices, which explains
the name). 

(b) An odd-even tree $\tc\in\ts_{\basic}^{\left[r\right]}$ will be
called \emph{sorted }if the graph spanned by the edges $1,...,a$
is connected for every $1\leq a\leq r$, and such that for any $v\in\tc_{0}$,
if ${\left\{ \sstar'_{i},\sstar'_{j}\right\} \subset\sigma_{\tc}\left(v\right)}$
for some $i<j$, then $\sstar'_{a}\in\sigma_{\tc}\left(v\right)$
for all $i\leq a\leq j$.
\end{defn}
For every odd-even tree $\tc\in\ts_{\basic}^{\left[r\right]}$ there
exists at least one $\tau\in\Sym\left(r\right)$ such that $\tau.\tc$
is sorted. The following facts are readily verified:
\begin{itemize}
\item Ordering the vertices of $\tc$, $\tc_{0}=\left(v_{1},...,v_{r+1}\right)$
as in Lemma \ref{lem:labeled trees alternative def}(b), if $\beta_{\tc}\left(v_{i}\right)=1\mod2$
and $\beta_{\tc}\left(v_{j}\right)=0\mod2$ then $i<j$, so the odd
vertices appear before the even vertices.
\item Let $\no$ denote the number of odd vertices. There are integers $r\geq s_{1}\geq s_{2}\geq\cdots\geq s_{\no}=0$
such that for $1\leq i\leq\no$ we have
\[
\sigma_{\tc}\left(v_{i}\right)=\left(\sstar'_{s_{i}+1},...,\sstar'_{s_{\left(i-1\right)}}\right)
\]
and $\sigma_{\tc}\left(v_{i}\right)=\emptyset$ for $\no+1\leq i\leq r+1$.
\item $\tc^{\left(j\right)}$ is sorted odd-even for every $\tc^{\left(j\right)}$
in the smoothing sequence of $\tc$.
\end{itemize}
\begin{prop}
Let $\basic=\left(\kf,\lf,\beta\right)$ be a basic moduli specification.
For a sorted odd-even\textup{ $\tc\in\ts_{\basic}^{\left[r\right]}:=\ts_{\left(\basic,\emptyset\right)}^{\left[r\right]}$
we have
\begin{align}
\jc_{\basic}^{\left[r\right]}|_{\mm_{\tc}} & =\left(-1\right)^{\left(1+m\right)\cdot\binom{o}{2}}\left(\prod\jc\right){}_{\basic}^{\left[r\right]}\label{eq:jc vs product}\\
\check{\jc}_{\basic}^{\left[r\right]}|_{\mm_{\tc}} & =\left(-1\right)^{\binom{r}{2}}\left(-1\right)^{\left(1+m\right)\cdot\binom{o}{2}}\left(\prod\check{\jc}\right){}_{\basic}^{\left[r\right]}\label{eq:check jc vs product}\\
\ff_{\basic}^{\left[r\right]}|_{\mm_{\tc}} & =\left(-1\right)^{\binom{r}{2}}\left(\prod\ff\right)_{\basic}^{\left[r\right]}\label{eq:ff vs product}
\end{align}
}
\end{prop}
where $o$ is the number of odd vertices.
\begin{proof}
We prove (\ref{eq:jc vs product}). Using the notation of Lemma
\ref{lem:smoothing sign difference}, and the properties of sorted
odd-even trees listed above, we have 
\[
\theta\left(\tc\right)=\prod_{a=1}^{r}\xi\left(\tc^{\left(a-1\right)}\right)
\]
for 
\[
\xi\left(\tc^{\left(a-1\right)}\right)=\underbrace{\left(-1\right)^{\left(1+m\right)\beta'\beta''}}_{f_{1}\left(a\right)}\underbrace{\left(-1\right)^{\left(r-a\right)}}_{f_{2}\left(a\right)}\underbrace{\left(-1\right)^{\sum_{j<i}\left|\sigma_{\tc^{\left(a\right)}}\left(v_{j}\right)\right|}}_{f_{3}\left(a\right)}\underbrace{\left(\sigma'\backslash\sstar'_{a},\sigma''\right)}_{f_{4}\left(a\right)}.
\]
It is easy to see that
\[
\prod_{a=1}^{r}f_{1}\left(a\right)=\left(-1\right)^{\left(1+m\right)\cdot\left(\sum_{i<j}\beta_{\tc}\left(v_{i}\right)\cdot\beta_{\tc}\left(v_{j}\right)\right)}=\left(-1\right)^{\left(1+m\right)\binom{o}{2}}.
\]
Fix some $1\leq a\leq r$. Let $1\leq j\leq o$. be such that $\sstar'_{a}\in\left(\sstar'_{s_{j}+1},...,\sstar'_{s_{\left(j-1\right)}}\right)=\sigma_{\tc}\left(v_{j}\right)$.
For every $1\leq b<s_{j}+1$, $\sstar'_{a}$ contributes to $f_{3}\left(b\right)$.
For $b=s_{j}+1$, $\sstar'_{a}$ contributes $\left(-1\right)^{a-\left(s_{j}+1\right)}$
to the sign of the shuffle $\left(\sigma'\backslash\sstar'_{s_{j}+1},\sigma''\right)$,
and for $b>s_{j}+1$ it contributes nothing (with the convention that
an element contributes to the shuffle sign when it is commuted past
elements preceding it, but not when elements succeeding it are commuted
past it). This shows that the total contribution of $\sstar'_{a}$
to $\prod_{b=1}^{r}\left(f_{3}\left(b\right)\cdot f_{4}\left(b\right)\right)$
is $\left(-1\right)^{\left(r-a\right)}$ which cancels $f_{2}\left(a\right)$.

The proof of (\ref{eq:check jc vs product}) is simpler, since $f_{3}\left(a\right)$
and $f_{4}\left(a\right)$ terms are absent. Equation (\ref{eq:ff vs product})
follows from the previous two equations since $\jc_{\basic}^{\left[r\right]}=\check{\jc}_{\basic}^{\left[r\right]}\circ\ff_{\basic}^{\left[r\right]}$
and ${\left(\prod\jc\right)_{\basic}^{\left[r\right]}=\left(\prod\check{\jc}\right)_{\basic}^{\left[r\right]}\circ\left(\prod\ff\right)_{\basic}^{\left[r\right]}}$.
\end{proof}
\begin{rem}
\label{rem:Pin factor verification}The factor $\left(-1\right)^{\left(1+m\right)\cdot\binom{o}{2}}$
was verified using relations derived from fixed-point localization,
see \cite[Example 35]{fp-loc-OGW}.
\end{rem}

\subsection{Proof of Proposition \ref{thm:resolution summary}}

Most of the theorem is obtained directly from the results of the previous
subsections, by specializing to $\rho=\left[r\right]$ and $\sfr=\left(\basic,\emptyset\right)$.
We now fill in the few remaining gaps.

By Lemma \ref{lem:labeled trees alternative def} and the stability
condition on the specification $\sfr_{\tc}\left(v\right)$ of the
vertices, $\ts_{\basic}^{r}=\emptyset$ for sufficiently large $r$,
and then $\mm_{\basic}^{r}=\emptyset$ too. 

We construct the involution $\inv_{\sfr}^{\rho}:\partial_{+}\mm_{\sfr}^{\rho}\to\partial_{+}\mm_{\sfr}^{\rho}$
recursively. We have 
\[
\partial_{+}^{For_{\sfr}^{\rho}}\mm_{\sfr}^{\rho}=\coprod'''\left(\left(\partial_{+}^{For_{\sfr'}^{\rho'}}\mm_{\sfr'}^{\rho'}\times\mm_{\sfr''}^{\rho''}\right)\coprod\left(\mm_{\sfr'}^{\rho'}\times\partial_{+}^{For_{\sfr''}^{\rho''}}\mm_{\sfr''}^{\rho''}\right)\right),
\]
so we can define recursively $\inv_{\sfr}^{\rho}$ by setting $\inv_{\sfr}^{\emptyset}=\inv_{\sfr}$
and 
\[
\inv_{\sfr}^{\rho}=\coprod'''\left(\left(\inv_{\sfr'}^{\rho'}\times\mbox{id}\right)\coprod\left(\mbox{id}\times\inv_{\sfr''}^{\rho''}\right)\right).
\]
A simple inductive proof based on (\ref{eq:F anti commutes with tau_sfr})
shows that $\left(\prod\ff\right)_{\sfr}^{\rho}\circ\iota_{\mm_{\sfr}^{\rho}}^{\partial_{+}}\circ Or\left(d\inv_{\sfr}^{\rho}\right)=\left(-1\right)\left(\prod\ff\right)_{\sfr}^{\rho}\circ\iota_{\mm_{\sfr}^{\rho}}^{\partial_{+}}$
and that $\pi_{\sfr}^{\rho}\circ\inv_{\sfr}^{\rho}=\pi_{\sfr}^{\rho}$.
Since $\left(\prod\ff\right)_{\sfr}^{\rho}$ and $\ff_{\sfr}^{\rho}$
differ by a sign which factors through $\pi_{\sfr}^{\rho}$ we find
that 
\[
\ff_{\sfr}^{\rho}\circ\iota_{\mm_{\sfr}^{\rho}}^{\partial_{+}}\circ Or\left(d\inv_{\sfr}^{\rho}\right)=\left(-1\right)\ff_{\sfr}^{\rho}\circ\iota_{\mm_{\sfr}^{\rho}}^{\partial_{+}}.
\]

This completes the proof of Proposition \ref{thm:resolution summary}.

\section{\label{sec:Extended-forms}Extended forms and Stokes' Theorem}

In this section we define an integration map and prove Stokes' theorem.
Throughout this section, we fix some stable basic moduli specification
$\basic=\left(\kf,\lf,\beta\right)$, and we may omit it from the
notation to avoid clutter. 

The \emph{internal }and \emph{external }local systems on ${\check{\mm}^{r}=\check{\mm}_{\basic}^{r}=\check{\mm}_{\basic}^{\left[r\right]}}$
are defined, respectively, by 
\[
\check{\mathcal{I}}_{\basic}^{r}=\bigotimes_{x\in\sstar''_{\left[r\right]}}\left(\ev{}_{x}^{\basic,r}\right)^{-1}Or\left(TL\right)\quad\mbox{and}\quad\check{\mathcal{E}}_{\basic}^{r}=\bigotimes_{x\in\kf}\left(\ev{}_{x}^{\basic,r}\right)^{-1}Or\left(TL\right).
\]
The internal and external local systems on $\mm_{\basic}^{r}$ are
given by $\mathcal{I}_{\basic}^{r}=\left(\For{}_{\basic}^{r}\right)^{-1}\check{\mathcal{I}}_{\basic}^{r}$
and $\mathcal{E}_{\basic}^{r}=\left(\For{}_{\basic}^{r}\right)^{-1}\check{\mathcal{E}}_{\basic}^{r}$,
respectively. 

\subsection{Statement of Stokes' theorem}

\begin{defn}
\label{def:extended form}An \emph{extended form} $\omega$ \emph{for}
$\basic$ is a sequence 
\[
\left\{ \check{\omega}_{r}\in\Omega\left(\check{\mm}_{\basic}^{r};\check{\mathcal{E}}_{\basic}^{r}\otimes_{\zz}\rr\left[\vec{\alpha}\right]\right)^{\tb}\right\} _{r\geq0},
\]
such that

(a) the following \emph{coherence property }holds for all $r\geq0$
and $\tc_{+}\in\ts_{\basic}^{r+1}$, 
\begin{equation}
\left(\left(i_{\check{\mm}_{\basic}^{r}}^{\partial}\right)^{*}\check{\omega}_{r}\right)|_{\partial^{\tc_{+}}\check{\mm}_{\tc}}=\left(\check{g}_{\tc_{+}}\right)^{*}\left(\check{\omega}_{r+1}|_{\check{\mm}_{\tc_{+}}}\right)\label{eq:check omega coherence}
\end{equation}
(see (\ref{eq:partial tc+ def},\ref{eq:check g definition})).

(b) $\check{\omega}_{r}$ is $Sym\left(r\right)$-invariant.

We denote by $\Omega_{\basic}$ the differential graded $\rr\left[\lambda_{1},...,\lambda_{m}\right]$-module
of all extended forms; the degree $\left(+1\right)$ differential
$D:\Omega_{\basic}\to\Omega_{\basic}$ is given by 
\[
D\left\{ \check{\omega}_{r}\right\} _{r\geq0}=\left\{ D\check{\omega}_{r}\right\} .
\]
\end{defn}
Our goal in this section is to define an $\rr\left[\vec{\alpha}\right]$-module
map 
\[
\int_{\basic}:\Omega_{\basic}\to\rr\left[\vec{\lambda}\right]=\rr\left[\lambda_{1},...,\lambda_{m}\right]
\]
of degree 
\[
-\left[\left(2m+1\right)\beta+2m+\left(k-3\right)+2l\right]=-\dim\mm_{\basic}^{0}
\]
 and prove 
\begin{thm}
\label{thm:stokes' theorem}(Stokes' Theorem) We have $\int_{\basic}D\omega=0$
for any $\omega\in\Omega_{\basic}$.
\end{thm}

\begin{rem}
\label{rem:non-equivariant}One can also consider the complex of non-equivariant
extended forms $\Omega_{\basic}^{\text{ne}}$, consisting of sequences
of forms $\check{\omega}_{r}\in\Omega\left(\check{\mm}_{\basic}^{r};\check{\mathcal{E}}_{\basic}^{r}\otimes_{\zz}\rr\right)$
satisfying the same two conditions as in Definition \ref{def:extended form}
above; the differential is given by the level-wise action of the exterior
derivative $d$, integration is defined using the same formula as
in Definition \ref{def:integration} below, except we replace $\Lambda$
by 
\[
\Lambda^{\text{ne}}:=\Lambda\mod\left(\lambda_{1},...,\lambda_{m}\right).
\]
The proof of Stokes' theorem carries through and we find 
\begin{equation}
\int_{\basic}d\omega=0.\label{eq:non-equiv stokes}
\end{equation}
The map $\rr\left[\vec{\lambda}\right]\to\rr$ lifts to a map $R_{\basic}:\Omega_{\basic}\to\Omega_{\basic}^{\text{ne}}$
commuting with integration.

We mention this for two reasons. First, we do not know whether the
induced map $H\left(R_{\basic}\right):H\left(\Omega_{\basic}\right)\to H\left(\Omega_{\basic}^{\text{ne}}\right)$,
is surjective, so there may be non-equivariant invariants which do
not admit an equivariant extension. Note, however, that $X^{l}\times L^{k}$
has cohomology in even degrees so the map 
\[
H_{\tb}^{\bullet}\left(X^{l}\times L^{k}\right)\to H^{\bullet}\left(X^{l}\times L^{k}\right)
\]
is\emph{ }surjective (cf. \cite[Proposition 32]{twA8}). This means
the equivariant open Gromov-Witten invariants we considered in $\S$\ref{subsec:equiv open GW invts}
exhaust both equivariant and non-equivariant invariants defined by
pull back along the evaluation maps. 

The second reason to consider non-equivariant invariants is that (\ref{eq:non-equiv stokes})
represents a stronger invariance property, which can be used to offer
a geometric interpretation of the invariants with ${\deg I\left(k,l,\beta\right)=0}$,
as we discussed in Remark \ref{rem:geometric invariants}.
\end{rem}

\subsection{\label{subsec:Resolution-blow-ups}Resolution blow ups}

In this subsection we write $G$ for the group $O\left(2m+1\right)$.
Fix some $G$-invariant Riemannian metric on $L\times L$, and construct
a $G$-equivariant tubular neighborhood 
\begin{equation}
N_{\Delta}\overset{j_{\Delta}}{\supset}V_{\Delta}\overset{\gamma_{\Delta}}{\longrightarrow}L\times L.\label{eq:tube of diagonal}
\end{equation}
Let $\widetilde{L\times L}\overset{\bu_{\Delta}}{\longrightarrow}L\times L$
denote the blow up of the diagonal $L\overset{\Delta_{L}}{\subset}L\times L$
(cf. \cite[Definition 37]{twA8}). It is a $G$-equivariant map of
manifolds with corners. Explicitly, 
\[
\widetilde{L\times L}=S\left(N_{\Delta}\right)\times[0,\epsilon)\bigcup_{S\left(N_{\Delta}\right)\times(0,\epsilon)}\left(L\times L\backslash\Delta_{L}\right).
\]

For each $\tc\in\ts_{\basic}^{r}$, $G^{r+1}$ acts on $\mm_{\tc}$
where the $i$'th factor of $G$ acts on the $i$'th factor of ${\mm_{\tc}=\prod_{v\in\tc_{0}}\mm_{\sfr_{\tc}\left(v\right)}}$,
making $\ev_{\sstar'_{i}}|_{_{\mm_{\tc}}},\ev_{\sstar''_{i}}|_{\mm_{\tc}}$
into $\left(G^{r+1}\to G\right)$-equivariant maps for each $i$,
with respect to a suitable projection $\left(G^{r+1}\to G\right)$.
This implies that $\ed_{\basic}^{r}:\mm_{\basic}^{r}\to\left(L\times L\right)^{r}$
is b-transverse to $\bu_{\Delta}^{r}:\left(\widetilde{L\times L}\right)^{r}\to\left(L\times L\right)^{r}$
(recall that this means that the restrictions of these two maps to
corner strata are also transverse, see Remark \ref{rem:easy b-transversality}),
and we use Lemma \ref{lem:fibp in orb} to construct the cartesian
square

\begin{equation}
\xymatrix{\widetilde{\mm}_{\basic}^{r}\ar[r]^{\bu_{\basic}^{r}}\ar[d]_{\widetilde{\ed}_{\basic}^{r}} & \mm_{\basic}^{r}\ar[d]^{\ed{}_{\basic}^{r}}\\
\left(\widetilde{L\times L}\right)^{r}\ar[r]_{\bu_{\Delta}^{r}} & \left(L\times L\right)^{r}
}
.\label{eq:blow up as single pullback}
\end{equation}
Since the right and bottom map are $G\times\Sym\left(r\right)$ equivariant,
there's a natural $G\times\Sym\left(r\right)$ action on $\widetilde{\mm}_{\basic}^{r}$
making the entire square equivariant.

By Lemma \ref{lem:fibp in orb}(c), since $\bu_{\Delta}^{r}$ is
b-normal so is $\bu_{\basic}^{r}$. In particular we have a decomposition
${\partial\widetilde{\mm}_{\basic}^{r}=\partial_{+}^{\bu_{\basic}^{r}}\widetilde{\mm}_{\basic}^{r}\coprod\partial_{+}^{\sim}\widetilde{\mm}_{\basic}^{r}\coprod\partial_{-}^{\sim}\widetilde{\mm}_{\basic}^{r}}$,
where 
\[
\partial_{\pm}^{\sim}\widetilde{\mm}_{\basic}^{r}=\left(\bu_{\basic}^{r}\right)_{-}^{-1}\left(\partial_{\pm}^{For_{\basic}^{r}}\mm_{\basic}^{r}\right).
\]
Writing $\bu_{\basic}^{r}$ as a composition of maps blowing up a
single edge at a time, using (\ref{eq:multi composition decomposition})
to further break down $\partial_{+}^{\bu_{\basic}^{r}}\widetilde{\mm}_{\basic}^{r}$,
we obtain the following decomposition of the blow up boundary:

\[
\partial\widetilde{\mm}_{\basic}^{r}=\partial_{+}^{\sim}\widetilde{\mm}_{\basic}^{r}\coprod\partial_{-}^{\sim}\widetilde{\mm}_{\basic}^{r}\coprod\coprod_{j=1}^{r}\left(\widetilde{\ed}{}_{j}^{\basic,r}\right)^{-1}S\left(N_{\Delta}\right).
\]

We turn to discuss orientations. Clearly $d\bu_{\basic}^{r}$ is a
diffeomorphism away from the boundary of $\partial\widetilde{\mm}_{\basic}^{r}$,
so it induces (see Lemma \ref{lem:ls extensions}(a)) a $G\times\Sym\left(r\right)$
equivariant local system map
\[
\Or\left(T\widetilde{\mm}_{\basic}^{r}\right)\to\Or\left(T\mm_{\basic}^{r}\right).
\]
By a slight abuse of notation we denote this map by $\Or\left(d\bu_{\basic}^{r}\right)$.

We set 
\[
\widetilde{\mathcal{J}}_{\basic}^{r}:\Or\left(T\widetilde{\mm}_{\basic}^{r}\right)\to\Or\left(TL\right)^{\boxtimes\left(\kf\coprod\sstar''_{\left[r\right]}\right)}
\]
to be the composition
\begin{equation}
\widetilde{\jc}_{\basic}^{r}=\jc_{\basic}^{r}\circ\Or\left(d\bu_{\basic}^{r}\right);\label{eq:tilde J}
\end{equation}
it follows that
\begin{equation}
\widetilde{\mathcal{J}}_{\basic}^{r}\circ\Or\left(d\tau.\right)=\sgn\left(\tau\right)\widetilde{\mathcal{J}}_{\basic}^{r}.\label{eq:tilde J Sym(r)-equivariance}
\end{equation}

Since $\partial_{+}^{\sim}\widetilde{\mm}_{\basic}^{r}$ is the fiber
product of $\ed_{\basic}^{r}|_{\partial_{+}\mm_{\basic}^{r}}$ with
$\bu_{\Delta}^{r}$ the pair of maps 
\[
\left(\id_{\left(\widetilde{L\times L}\right)^{r}},\inv_{\basic}^{r}\right)
\]
define an involution 
\[
\tilde{\inv}_{\basic}^{r}:\partial_{+}^{\sim}\widetilde{\mm}_{\basic}^{r}\to\partial_{+}^{\sim}\widetilde{\mm}_{\basic}^{r}
\]
such that $\bu_{\basic}^{r}\circ i_{\widetilde{\mm}_{\basic}^{r}}^{\partial_{+}^{\sim}}\circ\tilde{\inv}_{\basic}^{r}=\inv_{\basic}^{r}\circ\bu_{\basic}^{r}\circ i_{\mm_{\basic}^{r}}^{\partial_{+}}$.
It follows that 
\begin{equation}
\widetilde{\mathcal{J}}\circ\iota_{\widetilde{\mm}_{\basic}^{r}}^{\partial_{+}^{\sim}}\circ\Or\left(d\tilde{\inv}_{\basic}^{r}\right)=\left(-1\right)\widetilde{\mathcal{J}}\circ\iota_{\widetilde{\mm}_{\basic}^{r}}^{\partial_{+}^{\sim}}.\label{eq:tilde tau reverses orientation}
\end{equation}

\subsection{\label{subsec:Equivariant-homotopy-kernel}Equivariant homotopy kernel}

We fix an \emph{equivariant homotopy kernel} \linebreak{}
${\Lambda\in\Omega\left(\widetilde{L\times L};\tilde{\mbox{pr}}_{2}^{*}\left(Or\left(TL\right)\right)\otimes_{\zz}\rr\left[\vec{\lambda}\right]\right)^{\tb}}$for
$L$. Namely, 
\[
\Lambda=\sigma\left(\mbox{pr}_{S\left(N_{\Delta}\right)}^{\widetilde{N}_{\Delta}}\right)^{*}\phi+\bu_{\Delta}^{*}\Upsilon
\]
for $\phi$ an equivariant angular form for $S\left(N_{\Delta}\right)$,
$\sigma:[0,\infty)\to[0,1]$ a smooth, compactly supported cutoff
function with $\sigma\left(0\right)=+1$, and $\Upsilon$ chosen so
that 
\begin{equation}
D\Lambda=\widetilde{\pr}_{2}^{*}\rho_{0}\in\im\left(\widetilde{\pr}{}_{2}^{*}\right)\subset\im\left(\bu_{\Delta}^{*}\right),\label{eq:D Lambda}
\end{equation}
where $\rho_{0}$ is an equivariant form representing the point class.
It follows that 
\begin{equation}
\left(i_{\widetilde{L\times L}}^{\partial}\right)^{*}\Lambda=\phi+\left(\pi_{\Delta}^{S\left(N_{\Delta}\right)}\right)^{*}\Upsilon|_{\Delta}.\label{eq:Lambda asymptotics}
\end{equation}

\begin{rem}
Compare this to Definition 55 and Proposition 56 in \cite{twA8}.
First, there is a minus sign introduced in $\sigma$ for convenience.
More importantly, here we require only conditions (\ref{eq:D Lambda},
\ref{eq:Lambda asymptotics}), whereas in \cite{twA8} $\Lambda$
(denoted $\Lambda'$ there) depended on a particular choice of form
$\rho$ representing the point class, and the associated homotopy
operator $h'$ was modified further in order to satisfy the side conditions
(see Definition 23 \emph{ibid.}). If one can construct a unital cyclic
retraction that is represented by a smooth kernel $\Lambda$ as above,
then the open Gromov-Witten invariants we define here also encode
the unital cyclic homotopy type of the twisted equivariant Fukaya
$A_{\infty}$ algebra of $\rr\pp^{2m}\hookrightarrow\cc\pp^{2m}$.
See \cite[\S 1.6]{fp-loc-OGW} for a detailed discussion.
\end{rem}

\subsection{Integration of extended forms}

\begin{defn}
\label{def:integration}Let $\omega=\left\{ \check{\omega}_{r}\right\} \in\Omega_{\basic}$
be an extended form. We define
\[
\int_{\basic}\omega=\sum_{r\geq0}\frac{1}{r!}\int_{\widetilde{\mm}_{\basic}^{r}}\left(\bu_{\basic}^{r}\right)^{*}\omega_{r}\cdot\left(\widetilde{ed}_{b}^{r}\right)^{*}\Lambda^{\boxtimes r},
\]
where $\omega_{r}:=\left(\For{}_{\basic}^{r}\right)^{*}\check{\omega}_{r}$
and $\left(\widetilde{\ed}_{b}^{r}\right)^{*}\Lambda^{\boxtimes r}$
are forms with values in ${\left(\bu_{\basic}^{r}\right)^{-1}\mathcal{E}_{\basic}^{r}\otimes_{\zz}\rr\left[\lambda_{1},...,\lambda_{m}\right]}$
and in ${\left(\bu_{\basic}^{r}\right)^{-1}\mathcal{I}_{\basic}^{r}\otimes_{\zz}\rr\left[\lambda_{1},...,\lambda_{m}\right]}$,
respectively, so that the integrand takes values in 
\[
{\left(\bu_{\basic}^{r}\right)^{-1}\left(\mathcal{E}_{\basic}^{r}\otimes\mathcal{I}_{\basic}^{r}\right)\otimes_{\zz}\rr\left[\vec{\alpha}\right]=\bigotimes_{x\in k\coprod\sstar''_{\rho}}\left(\widetilde{\ev}{}_{x}^{\basic,r}\right)^{-1}Or\left(TL\right)}
\]
and the integral is computed using $\widetilde{\mathcal{J}}_{\basic}^{r}$.
More precisely, integration of real-valued forms is defined as pushforward
along the horizontally-submersive map $\widetilde{\mm}_{\basic}^{r}\to\pt$,
see \cite[Eq (24)]{mod2hom}, which becomes an oriented map using
$\widetilde{\mathcal{J}}_{\basic}^{r}$. The integral is then extended
$\rr\left[\vec{\alpha}\right]$-linearly to define integration of
equivariant forms.
\end{defn}

\begin{proof}
[Proof of Theorem \ref{thm:stokes' theorem} (Stokes' Theorem).]The
computation goes as follows.
\begin{multline*}
\int_{\basic}D\omega=\sum_{r\geq0}\frac{1}{r!}\int_{\widetilde{\mm}_{\basic}^{r}}\left(\bu_{\basic}^{r}\right)^{*}D\left(\omega_{r}\right)\cdot\left(\widetilde{\ed}_{b}^{r}\right)^{*}\Lambda^{\boxtimes r}\overset{\left(1\right)}{=}\\
=\sum_{r\geq0}\frac{1}{r!}\int_{\widetilde{\mm}_{\basic}^{r}}D\left[\left(\bu_{\basic}^{r}\right)^{*}\omega_{r}\cdot\left(\widetilde{\ed}_{b}^{r}\right)^{*}\Lambda^{\boxtimes r}\right]\overset{\left(2\right)}{=}\sum_{r\geq0}\frac{1}{r!}\int_{\partial\widetilde{\mm}_{\basic}^{r}}\left(\bu_{\basic}^{r}\right)^{*}\omega_{r}\cdot\left(\widetilde{\ed}_{b}^{r}\right)^{*}\Lambda^{\boxtimes r}\overset{\left(3\right)}{=}\\
=\sum_{r\geq0}\frac{1}{r!}\int_{\partial_{-}^{\sim}\mm_{\basic}^{r}}\left(\bu_{\basic}^{r}\right)^{*}\omega_{r}\cdot\left(\widetilde{\ed}_{b}^{r}\right)^{*}\Lambda^{\boxtimes r}+\\
+\sum_{\begin{array}{c}
r\geq0\\
0\leq i\leq r+1
\end{array}}\frac{1}{\left(r+1\right)!}\int_{\left(\widetilde{\ed}_{i}^{\basic,r+1}\right)^{-1}\left(S\left(N_{\Delta}\right)\right)}\left(\bu_{\basic}^{r+1}\right)^{*}\omega_{r+1}\cdot\left(\widetilde{\ed}_{b}^{r}\right)^{*}\Lambda^{\boxtimes\left(r+1\right)}\overset{\left(4\right)}{=}\\
=\sum_{r\geq0}\frac{1}{r!}\bigg[\int_{\partial_{-}^{\sim}\mm_{\basic}^{r}}\left(\bu_{\basic}^{r}\right)^{*}\omega_{r}\cdot\left(\widetilde{\ed}_{b}^{r}\right)^{*}\Lambda^{\boxtimes r}+\\
\int_{\left(\widetilde{\ed}_{r+1}^{\basic,r+1}\right)^{-1}\left(S\left(N_{\Delta}\right)\right)}\left(\bu_{\basic}^{r+1}\right)^{*}\omega_{r+1}\cdot\left(\ed_{1,...,r}^{\basic,r+1}\right)^{*}\Lambda^{\boxtimes r}\,\left(\widetilde{\ed}_{r+1}^{\basic,r+1}\right)^{*}\phi\bigg]\\
\overset{\left(5\right)}{=}0.
\end{multline*}

The equality marked (1) is justified by Lemma \ref{lem:D hits Lambda vanishes}
below. The equality marked (2) is the usual Stokes' theorem. To justify
the equality (3), we argue that 
\[
\int_{\partial_{+}^{\sim}\mm_{\basic}^{r}}\left(\bu_{\basic}^{r}\right)^{*}\omega_{r}\cdot\left(\widetilde{\ed}_{b}^{r}\right)^{*}\Lambda^{\boxtimes r}=0.
\]
For this use (\ref{eq:tilde tau reverses orientation}) and observe
that the integrand is $\tilde{\inv}_{\basic}^{r}$-invariant since
\[
{\For_{b}^{r}\circ\bu_{\basic}^{r}\circ i_{\mm_{\basic}^{r}}^{\partial_{+}^{\sim}}\circ\widetilde{\inv}_{\basic}^{r}=\For_{b}^{r}\circ\bu_{\basic}^{r}\circ i_{\mm_{\basic}^{r}}^{\partial_{+}^{\sim}}}
\]
and 
\[
\widetilde{\ed}_{\basic}^{r}\circ i_{\mm_{\basic}^{r}}^{\partial_{+}^{\sim}}\circ\tilde{\inv}_{\basic}^{r}=\widetilde{\ed}_{\basic}^{r}\circ i_{\mm_{\basic}^{r}}^{\partial_{+}^{\sim}}.
\]
We justify the equality (4). Write the expression on the fourth line
as 
\[
\sum_{r\geq0}\sum_{i=0}^{r+1}\frac{I_{r+1}\left(i\right)}{\left(r+1\right)!},\;I_{r+1}\left(i\right):=\int_{\left(\widetilde{\ed}_{i}^{\basic,r+1}\right)^{-1}\left(S\left(N_{\Delta}\right)\right)}\left(\bu_{\basic}^{r+1}\right)^{*}\omega_{r+1}\cdot\left(\widetilde{\ed}_{b}^{r}\right)^{*}\Lambda^{\boxtimes\left(r+1\right)}.
\]
We claim that 
\begin{equation}
I_{r+1}\left(i\right)=I_{r+1}\left(r+1\right),\label{eq:i to r+1}
\end{equation}
so $\sum_{i=0}^{r+1}\frac{I_{r+1}\left(i\right)}{\left(r+1\right)!}=\frac{I_{r+1}\left(r+1\right)}{r!}$,
which immediately gives the equality (4). To prove (\ref{eq:i to r+1}),
pullback the integrand by some $\tau\in\Sym\left(r+1\right)$ with
$\tau\left(i\right)=r+1$. $\widetilde{\jc}_{\basic}^{r}$ picks up
a sign, $\sgn\left(\tau\right)$, by (\ref{eq:tilde J Sym(r)-equivariance}),
which cancels the sign of permuting the odd-degree $\Lambda$'s.

Eq (\ref{eq:check omega coherence}) implies $\omega_{r+1}=\left(g_{\basic}^{r+1}\right)^{*}\omega_{r}$.
Using this and Corollary \ref{cor:J coherence}, we obtain equality
(5) by integrating out $\phi$.
\end{proof}
\begin{lem}
\label{lem:D hits Lambda vanishes}$\int_{\widetilde{\mm}_{\basic}^{r}}\left(\bu_{\sfr}^{\rho}\right)^{*}\omega_{r}\cdot\left(\widetilde{ed}_{b}^{r}\right)^{*}\left(\Lambda^{\boxtimes\left(j-1\right)}\boxtimes D\Lambda\boxtimes\Lambda^{\boxtimes\left(r-j\right)}\right)=0.$
\end{lem}
\begin{proof}
By (\ref{eq:D Lambda}), the integrand is pulled back from the orbifold
$\mm'$, obtained from $\mm_{\basic}^{r}$ by forgetting $\sstar'_{j}$
and blowing up $\left(\ed_{i}^{\basic,r}\right)^{-1}\Delta$ for $i\neq j$.
Since 
\[
\dim\mm'=\dim\widetilde{\mm}_{\basic}^{r}-1,
\]
the integral vanishes.
\end{proof}

\section{\label{sec:appendix}Appendix: Orbifolds with Corners}

This appendix summarizes briefly some definitions and results from
\cite[\S 3]{mod2hom}. The reader should consult that reference for
the proofs and more detail. 

\subsection{Manifolds with corners}

We refer the reader to Joyce \cite[\S 2]{joyce-generalized} for the
terminology we use regarding manifolds with corners. The manifolds
we'll consider have ``ordinary'' corners (as opposed to generalized
corners), which are modeled on $\rr_{k}^{n}:=[0,\infty)^{k}\times\rr^{n-k}$. 

A \emph{weakly smooth} map $f:U\to V$ between open subsets $U\subset\rr_{k}^{m}$
and $V\subset\rr_{l}^{n}$ is a continuous map $f=\left(f_{1},...,f_{n}\right)$
such that all the partial derivatives \linebreak{}
${\frac{\partial^{a_{1}+\cdots+a_{m}}}{\partial u_{1}^{a_{1}}\cdots\partial u_{m}^{a_{m}}}f_{j}:U\to\rr}$
exist and are continuous (including one-sided derivatives where applicable). 

An $n$-dimensional \emph{manifold with corners} $X$ is a second
countable Hausdorff space equipped with a maximal $n$-dimensional
atlas of charts $\left(U,\phi\right)$ where $U\subset X$ is open
and $\phi:U\to\rr_{k}^{n}$ is a homeomorphism ($n$ is fixed, $k$
may vary), with weakly smooth transitions. A \emph{weakly smooth map
$f:X\to Y$} between manifolds with corners is a continuous map which
is of this form in every coordinate patch. A weakly smooth map $f:X\to Y$
is said to be \emph{smooth}, \emph{strongly smooth, interior, b-normal,
simple}, or a \emph{b-fibrations} as in \cite[Definitions 2.1, 4.3]{joyce-generalized}.
``A map'' between manifolds with corners will always be assumed
to be smooth unless specifically stated otherwise, and we denote by
$\manc$ the category of manifolds with corners with smooth maps.

The \emph{depth }of a point ${x=\left(x_{1},...,x_{n}\right)\in\rr_{k}^{n}}$
is defined by ${\mbox{depth}\left(x\right)=\#\left\{ 1\leq i\leq k|x_{i}=0\right\} }$.
It is easy to see that the transitions preserve the depth, so we can
speak of the depth of a point $x\in X$. We define $S^{k}\left(X\right)=\left\{ x\in X|\mbox{depth}\left(x\right)=k\right\} $.
A \emph{local k-corner component $\gamma$ }of $X$ at $x$ is a local
choice of connected component of $S^{k}\left(X\right)$ near $x$
(cf. \cite[Definition 2.7]{joyce-generalized}); a local 1-corner
component is also called a \emph{local boundary component}. 

We have manifolds with corners
\[
\partial X=C_{1}\left(X\right)=\left\{ \left(x,\beta\right)|x\in X,\,\mbox{\ensuremath{\beta}\ is a local boundary component of \ensuremath{X}at \ensuremath{x}}\right\} 
\]
and, for every $k\geq0$,
\[
C_{k}\left(X\right)=\left\{ \left(x,\gamma\right)|x\in X,\,\gamma\text{ is a local \ensuremath{k}-corner component of \ensuremath{X} at \ensuremath{x}}\right\} .
\]
Letting $\partial^{k}X$ denote the iterated boundary, we find that
$C_{k}\left(X\right)\simeq\partial^{k}X/\Sym\left(k\right)$ where
$\Sym\left(k\right)$ acts by permuting the local boundary components. 

We can consider $C\left(X\right)=\coprod_{k\geq0}C_{k}\left(X\right)$
as a \emph{local manifold with corners }(or ``manifold with corners
of mixed dimension'', in Joyce's terms). These form a category and
the various properties of maps can be used to describe maps between
local manifolds with corners. If $f:X\to Y$ is a smooth map of manifolds
with corners, there's an induced interior map 
\[
C\left(f\right):C\left(X\right)\to C\left(Y\right)
\]
We denote by $i_{X}^{\partial}:\partial X\to X$ the map defined by
$i_{X}^{\partial}\left(\left(x,\beta\right)\right)=x$. Even if $X$
is connected, $\partial X$ may be disconnected and $i_{X}^{\partial}$
may not be injective. Sometimes we abbreviate $i^{\partial}=i_{X}^{\partial}$. 

A strongly smooth map $f:X\to Y$ between manifolds with corners is
\emph{a submersion }if, whenever $x$ of depth $k$ maps to $y=f\left(x\right)$
of depth $l$, both $df|_{x}:T_{x}X\to T_{y}Y$ and $df|_{x}:T_{x}S^{k}\left(X\right)\to T_{y}S^{l}\left(Y\right)$
are surjective (see \cite[Definition 3.2]{joyce-fibered}; beware
that a ``smooth map'' there is what we call a strongly smooth map,
see \cite[Remark 2.4,(iii)]{joyce-generalized}). We say a map $f:X\to Y$
is \emph{perfectly simple }if it is simple and maps points of depth
$k$ to points of depth $k$, and is \emph{étale }if it is a local
diffeomorphism.

If $X$ is a manifold with corners its tangent bundle $TX$ is defined
in the obvious way. In addition, one can consider the \emph{b-tangent
bundle} $^{b}TX$ . It is a vector bundle on $X$ whose sections can
be identified with sections $v\in C^{\infty}\left(TX\right)$ such
that $v|_{S^{k}\left(X\right)}$ is tangent to $S^{k}\left(X\right)$
for all $k$ (cf. \cite[Definition 2.15]{joyce-generalized}). If
$f:X\to Y$ is an interior map of orbifolds with corners, there's
an induced map $^{b}df:^{b}TX\to^{b}TY$. Two interior maps $f:X\to Z$
and $g:Y\to Z$ are called \emph{b-transverse }if for any $x\in S^{j}\left(X\right),\,y\in S^{k}\left(Y\right)$
such that $f\left(x\right)=g\left(y\right)=z$, the map 
\[
^{b}df\oplus^{b}dg:{}^{b}T_{x}X\oplus^{b}T_{y}Y\to^{b}T_{z}Z
\]
is surjective. 
\begin{rem}
\label{rem:easy b-transversality}In case $\partial Z=\emptyset$,
$f,g$ are b-transverse if and only if for every $x\in S^{j}\left(X\right),y\in S^{k}\left(Y\right)$
with $f\left(x\right)=g\left(y\right)=z$ the map
\[
df|_{TS^{k}\left(X\right)}\oplus dg|_{TS^{l}\left(Y\right)}:TS^{k}\left(X\right)\oplus TS^{l}\left(Y\right)\to T_{z}Z
\]
is surjective.
\end{rem}
\begin{lem}
\label{prop:fibered product in manc}Let $X,Y,Z$ be manifolds with
corners and let $f:X\to Z$ and $g:Y\to Z$ be continuous. Consider
the topological fiber product 
\[
P=X\fibp{f}{g}Y=\left\{ \left(x,y\right)\in X\times Y|f\left(x\right)=g\left(y\right)\right\} .
\]
Suppose at least one of the following conditions holds.

(i) $f$ is a b-normal submersion and $g$ is strongly smooth and
interior,

(ii) $f$ is étale, $g$ is a smooth map,

(iii) $f$ is a b-submersion, $g$ is perfectly simple, or

(iv) $\partial Z=\emptyset$, $f,g$ are b-transverse and smooth.

Then $P$ admits a unique structure of a manifold with corners making
it the fiber product in $\manc$, and we have
\begin{equation}
C_{i}\left(W\right)=\coprod_{j,k,l\geq0;i=j+k-l}C_{j}^{l}\left(X\right)\times_{C_{l}\left(Z\right)}C_{k}^{l}\left(Y\right)\label{eq:corners as fibered prods}
\end{equation}
where $C_{j}^{l}\left(X\right)=C_{j}\left(X\right)\cap C\left(f\right)^{-1}\left(C_{l}\left(Z\right)\right)$
and $C_{k}^{l}\left(Y\right)=C_{k}\left(Y\right)\cap C\left(g\right)^{-1}\left(C_{l}\left(Z\right)\right)$,
and the fiber product is taken over $C\left(f\right),C\left(g\right)$. 

Moreover, if $X\xrightarrow{f}Z$ (respectively, $Y\xrightarrow{g}Z$)
is b-normal then so is $P\xrightarrow{f'}Y$ (resp., $P\xrightarrow{g'}X$).
\end{lem}

In what follows the discussion diverges from \cite{joyce-generalized}
(see more specifically $\S$4.2 there). More precisely we introduce
a stronger notion of a closed immersion, that has the implicit function
theorem built into it. This is the only kind of closed immersion that
we need to consider, and makes the discussion considerably simpler.
\begin{defn}
\label{def:cl im}A map $f:X\to Y$ of manifolds with corners is called
a \emph{closed immersion} if for every $p\in X$ there exists an open
neighborhood $p\in U\subset X$, an open neighborhood $f\left(U\right)\subset V\subset Y$,
and a strongly smooth submersion $h:V\to\rr^{N}$ for some integer
$N\geq0$ such that the following square is cartesian 
\[
\xymatrix{U\ar[r]^{f|_{U}}\ar[d] & V\ar[d]^{h}\\
0\ar[r] & \rr^{N}
}
\]
(it follows that $N=\dim Y-\dim X$). The fiber product exists by
Lemma \ref{prop:fibered product in manc} since $h$ is (vacuously)
b-normal, and $0\to\rr^{N}$ is strongly smooth and interior.
\end{defn}
\begin{rem}
\label{rem:b-submersive is enough}Any b-submersion to a manifold
without boundary is automatically a strongly smooth submersion, so
in Definition \ref{def:cl im} it suffices to assume that $h$ is
a b-submersion. 
\end{rem}

\begin{defn}
A map $f:X\to Y$ of manifolds with corners is called \emph{a closed
embedding }if it is a closed immersion, has a closed image, and induces
a homeomorphism on its image.
\end{defn}

\begin{defn}
A map $f:X\to Y$ of manifolds with corners is an \emph{open embedding}
if it is étale and injective.
\end{defn}
\begin{defn}
(a) Let $f:X\to Y$ be a map of manifolds with corners. We say $f$
is \emph{horizontally submersive }if for every $\tilde{x}\in X$ the
germ $f_{\tilde{x}}$ is isomorphic to the projection $\rr_{k}^{n}\to\rr_{k'}^{n'}$,
\[
\left(x_{1},...,x_{n}\right)\mapsto\left(x_{1},...,x_{k'},x_{k+1},...,x_{k+n'}\right).
\]

(b) Let $f:X\to Y$ be a b-normal map. We call
\[
C_{k}^{hor}\left(X\right):=\left(C\left(f\right)^{-1}\left(C_{0}\left(Y\right)\right)\cap C_{k}\left(X\right)\right)
\]
the \emph{horizontal k-corners }of $X$ with respect to $f$. 
\end{defn}
\begin{lem}
A map $f:X\to Y$ is horizontally submersive if and only if it is
b-normal and the induced map $C_{k}^{hor}\left(X\right)\xrightarrow{C\left(f\right)}Y$is
a submersion for every $k$; that is, 
\[
T_{x}C_{k}^{hor}\left(X\right)\xrightarrow{dC\left(f\right)}T_{y}Y
\]
is surjective for all $x\in C_{k}^{hor}\left(X\right)$.
\end{lem}

Suppose now $X,Y$ are manifolds with corners, $f$ is horizontally
submersive with oriented fibers, and let $\omega$ be a compactly
supported differential form on $X$. In this case we can define $f_{*}\omega$
by integration along the fiber.

\subsection{Orbifolds with corners}
\begin{defn}
A groupoid $\left(G_{0},G_{1},s,t,e,i,m\right)$ is a category where
every arrow is invertible. Namely, $G_{0}$ is a class of \emph{points}
and $G_{1}$ is a class of \emph{arrows}. The\emph{ }maps $s,t:G_{1}\to G_{0}$
take an arrow to its \emph{source }and \emph{target }objects, respectively.
The \emph{composition} map ${m:\left\{ \left(f,g\right)\in G_{1}\times G_{1}|t\left(f\right)=s\left(g\right)\right\} \to G_{1}}$
takes a pair of composable arrows to their composition. The \emph{identity
map} $e:G_{0}\to G_{1}$ takes an object to the identity arrow and
the \emph{inverse map} $i:G_{1}\to G_{1}$ takes an arrow to its inverse. 
\end{defn}
The equivalence classes of the equivalence relation ${\im\left(s\times t\right)\subset G_{0}\times G_{0}}$
are called \emph{the orbits} of the groupoid; the class of all orbits
is denoted $G_{0}/G_{1}$. We will use different notations for groupoids,
depending on how much of the structure we want to label:
\[
\left(G_{0},G_{1},s,t,e,i,m\right)=G_{\bullet}=G_{1}\overset{s,t}{\rightrightarrows}G_{0}.
\]

\begin{defn}
\label{def:groupoids in ManC}A groupoid $\left(X_{0},X_{1},s,t,e,i,m\right)$
will be called \emph{étale }if $X_{0},X_{1}$ are objects of $\manc$,
and the maps $s,t,e,i,m$ are all étale (in fact, it suffices to require
that $s:X_{1}\to X_{0}$ is étale). An étale groupoid will be called
\emph{proper }if the map $s\times t:X_{1}\to X_{0}\times X_{0}$ is
proper. We will mostly be interested in proper étale groupoids, or
PEG's for short.

Let $X_{\bullet}$ be a PEG. The set of orbits $X_{0}/X_{1}$, taken
with the quotient topology, forms a locally compact Hausdorff space.
$X_{\bullet}$ is called \emph{compact }if $X_{0}/X_{1}$ is compact. 
\end{defn}
Let $X_{\bullet},Y_{\bullet}$ be two PEG's. A\emph{ smooth functor}
$X_{\bullet}\xrightarrow{F_{\bullet}}Y_{\bullet}$ consists of a pair
of  smooth maps $F_{0}:X_{0}\to Y_{0}$ and $F_{1}:X_{1}\to Y_{1}$
which is a functor between the underlying categories. If $F_{\bullet},G_{\bullet}:X_{\bullet}\to Y_{\bullet}$
are two functors a\emph{ smooth transformation }$\alpha:F_{\bullet}\Rightarrow G_{\bullet}$\textbf{
}is a smooth map $X_{0}\to Y_{1}$ which is a natural transformation
between the underlying functors. In this way we obtain a bicategory
(see \cite{benabou}) $\mathbf{PEG}$, whose objects, or \emph{0-cells},
are proper étale groupoids, morphisms (or \emph{1-cells}) are smooth
functors, and 2-cells are natural transformations. A \emph{refinement
}$R_{\bullet}:X_{\bullet}\to X'_{\bullet}$ is a smooth functor which
is an equivalence of categories and such that $R_{0}$ (hence also
$R_{1}$) is an étale map.
\begin{lem}
As a subset of the 1-cells of $\mathbf{PEG}$ the refinements admit
a right calculus of fractions, in the sense of \cite[\S 2.1]{pronk}.
\end{lem}
We define the category $\mathbf{Orb}$ of orbifolds (always with corners,
unless specifically mentioned otherwise) to be the 2-localization
of $\mathbf{PEG}$ by the refinements. We usually denote orbifolds
by calligraphic letters $\xx,\yy,\mm$... They are given by proper
étale groupoids. Maps $\xx\to\yy$ are given by fractions $F_{\bullet}|R_{\bullet}$
with $X_{\bullet}\xleftarrow{R_{\bullet}}X_{\bullet}'$ a refinement
and $X'_{\bullet}\xrightarrow{F_{\bullet}}Y_{\bullet}$ a smooth functor.
We refer the reader to \cite{pronk} for further details, including
the definition of the 2-cells, the composition operations, etc.
\begin{defn}
\label{def:properties of orbi-maps}We say $f$ is\emph{ strongly-smooth,
étale, interior, b-normal, submersive, b-submersive, horizontally
submersive, simple }or\emph{ perfectly simple }if $F_{0}$ has the
corresponding property as a map of manifolds with corners. It is
easy to check that these properties are preserved by 2-cells (and
thus are properties of the homotopy class of $f$). The map $f$ is
called a\emph{ b-fibration} if it is b-normal and b-submersive (cf.
\cite[Definition 4.3]{joyce-generalized}).

For $i=1,2$ let $f^{i}=F^{i}|R^{i}:\xx^{i}\to\yy$ be an interior
map. We say $f^{1}$ and $f^{2}$ are \emph{b-transverse }if $F_{0}^{1},F_{0}^{2}$
are b-transverse (as maps of manifolds with corners). 
\end{defn}
An equivalence in $\mathbf{Orb}$ is called \emph{a diffeomorphism.}
We say $f=F|R:\xx\to\yy$ is \emph{full}, \emph{essentially surjective},
or \emph{faithful} if $F$ is full, essentially surjective, or faithful,
respectively.

If $\xx=X_{1}\rightrightarrows X_{0}$ is an orbifold with corners,
$\partial\xx=\partial X_{1}\rightrightarrows\partial X_{0}$ is naturally
an orbifold with corners and the smooth functor $\left(i_{X_{1}}^{\partial},i_{X_{0}}^{\partial}\right)$
induces a map $i_{\xx}^{\partial}:\partial\xx\to\xx$. We denote 
\[
i_{\xx}^{\partial^{c}}:=i_{\xx}^{\partial}\circ i_{\partial\xx}^{\partial}\circ\cdots\circ i_{\partial^{c-1}\xx}^{\partial}:\partial^{c}\xx\to\xx.
\]
 Since the maps $s,t,e,i,m$ are étale, they preserve the depth and
we obtain orbifolds with corners
\[
C_{k}\left(\xx\right)=C_{k}\left(X_{1}\right)\rightrightarrows C_{k}\left(X_{0}\right)
\]
for all $k$. A \emph{local orbifold with corners }$\xx=\coprod\xx_{n}$
(or just an $l$-orbifold) is a disjoint union of orbifolds with corners
with $\dim\xx_{n}=n$. It is obvious how to turn this into a category
and extend the definitions of various types of maps to this situation.
If $\xx$ is an orbifold with corners, we can consider $C\left(\xx\right)=\coprod_{k\geq0}C_{k}\left(\xx\right)$
as an l-orbifold. A smooth map $f:\xx\to\yy$ induces an interior
map $C\left(f\right):C\left(\xx\right)\to C\left(\yy\right)$.

We turn to a discussion of the weak fibered product in $\text{\textbf{Orb}}$. 
\begin{lem}
\label{lem:fibp in orb}Let
\begin{equation}
f:\xx\xleftarrow{R}\xx'\xrightarrow{F}\zc\text{ and }g:\yy\xleftarrow{S}\yy'\xrightarrow{G}\zc\label{eq:f and g fracs}
\end{equation}
be two 1-cells in $\text{\textbf{Orb}}$. Suppose at least one of
the following conditions holds.

(i) $F$ is a b-normal submersion and $G$ is strongly smooth and
interior,

(ii) $F$ is étale, $G$ is a smooth map,

(iii) $F$ is a b-submersion, $G$ is perfectly simple, or

(iv) $\partial\zc=\emptyset$, $F$ and $G$ are b-transverse (see
Remark \ref{rem:easy b-transversality} for an equivalent condition)
and smooth. 

Then 

(a) The weak fiber product $\mathcal{P}=\xx\fibp{f}{g}\yy$ exists
in $\text{\textbf{Orb}}$. In fact, we can take 
\[
\mathcal{P}=\xx'\fibp{F}{G}\yy'
\]
the weak fiber product in $\text{\textbf{PEG}}$, given by the groupoid
$P_{1}\rightrightarrows P_{0}$ where 
\begin{eqnarray*}
P_{0}=X_{0}'\fibp{F_{0}}{s}Z_{1}\fibp{t}{G_{0}}Y_{0}',\\
P_{1}=X_{1}'\fibp{s\circ F_{1}}{s}Z_{1}\fibp{t\circ G_{1}}{s}Y_{1}'.
\end{eqnarray*}
Here an element of $P_{1}$ specifies the three solid arrows in the
diagram below, 
\[
\xymatrix{x^{1}\ar[d]_{a} & F_{0}\left(x^{1}\right)\ar@{-->}[d]_{F_{1}\left(a\right)}\ar[r] & G_{0}\left(y^{1}\right)\ar@{-->}[d]^{G_{1}\left(b\right)} & y^{1}\ar[d]^{b}\\
x^{2} & F_{0}\left(x^{2}\right)\ar@{-->}[r] & G_{0}\left(y^{2}\right) & y^{2}
}
.
\]
The horizontal dashed arrow is uniquely determined by requiring the
square to be commutative; $s,t:P_{1}\to P_{0}$ are the projections
on the top and bottom rows of the diagram, respectively, and the
other structure maps are computed similarly.

(b) We have
\begin{equation}
C_{i}\left(\mathcal{P}\right)=\coprod_{j,k,l\geq0;i=j+k-l}C_{j}^{l}\left(\xx\right)\times_{C_{l}\left(Z\right)}C_{k}^{l}\left(\yy\right)\label{eq:corners as fibp in orb}
\end{equation}
where $C_{j}^{l}\left(\xx\right)=C_{j}\left(\xx\right)\cap C\left(f\right)^{-1}\left(C_{l}\left(\zc\right)\right)$
and $C_{k}^{l}\left(\yy\right)=C_{k}\left(\yy\right)\cap C\left(g\right)^{-1}\left(C_{l}\left(\zc\right)\right)$,
and the weak fiber product is taken over $C\left(f\right),C\left(g\right)$. 

(c) If we assume, in addition, that $\xx\xrightarrow{f}\zc$ (respectively,
$\yy\xrightarrow{g}\zc$) is b-normal then so is $\mathcal{P}\xrightarrow{f'}\yy$
(resp., $\mathcal{P}\xrightarrow{g'}\xx$).
\end{lem}

\begin{defn}
A map $F|R:\xx\to\yy$ of orbifolds with corners is a \emph{closed
immersion} if $F_{0}$ is a closed immersion. In this case, the same
holds for any map homotopic to $F|R$.
\end{defn}

A manifold with corners $M$ specifies an orbifold $\underline{M}=M\rightrightarrows M$
with only identity morphisms, and this extends to a 2-fully-faithful
pseudofunctor $\manc\to\text{\textbf{Orb}}$ (namely, it restricts
to an equivalence $\manc\left(X,Y\right)\simeq\text{\textbf{Orb}}\left(\underline{X},\underline{Y}\right)$
for any pair $X,Y$ of objects of $\manc$). We say an orbifold ``is''
a manifold with corners if it is in the essential image of this functor.
\begin{defn}
Let $\xx$ be an orbifold with corners. \emph{An atlas} \emph{for
$\xx$ }is a map $p:\underline{M}\to\xx$ where $M$ is some manifold
with corners, such that for any other map $f:\underline{N}\to\xx$
from a manifold with corners, $\underline{M}\times_{\xx}\underline{N}$
is a manifold with corners and the projection $\underline{M}\times_{\xx}\underline{N}\overset{p'}{\longrightarrow}\underline{N}$
is étale and surjective (as a map of $\manc$).
\end{defn}
The obvious map $\underline{X_{0}}\to\left(X_{1}\rightrightarrows X_{0}\right)$
is an atlas. Conversely, any atlas $\underline{M}\to\xx$ defines
an orbifold equivalent to $\xx$, whose objects are $\underline{M}$
and morphisms are $\underline{M}\times_{\xx}\underline{M}$.
\begin{defn}
\label{def:closed embedding}A map $f:\xx\to\yy$ of orbifolds with
corners is a \emph{closed} (respectively, \emph{open}) \emph{embedding}
if for some (hence any) atlas $p:\underline{M}\to\yy$, the 2-pullback
$\underline{M}\fibp{p}{f}\xx$ is a manifold with corners  and the
map $\underline{M}\fibp{p}{f}\xx\to\underline{M}$ is a closed (resp.
open) embedding of manifolds with corners.
\end{defn}
If $f:\xx\to\yy$ is a closed embedding we may refer to $\xx$ as
a \emph{suborbifold} of $\yy$.

The notion of a sheaf on an orbifold $\xx$ is the same as the notion
of a sheaf on the underlying topological orbifold (see \cite{moerdijk-classifying-groupoids,pronk}).
A vector bundle $E$ on an orbifold with corners $\xx=X_{1}\overset{s,t}{\rightrightarrows}X_{0}$
is given by $\left(E_{0},\phi\right)$ where $E_{0}$ is a smooth
vector bundle on $X_{0}$ and 
\[
\phi:s^{*}E_{0}\to t^{*}E_{0}
\]
is an isomorphism satisfying some obvious compatibility requirements
with the groupoid structure. The sections of $\left(E_{0},\phi\right)$
form a sheaf over $\xx$. A \emph{local system }on an orbifold $\xx$
is a sheaf which is locally isomorphic to the constant sheaf $\underline{\zz}$.
We extend the conventions set forth in \cite[\S1.1, \S 6.1]{twA8}
to proper étale groupoids with corners in the obvious way\footnote{Note there we had to work with $\cc$-valued local systems, but for
the purposes of this paper we can work with $\zz$-valued local systems.}. In particular, for every vector bundle $E$ on $X_{\bullet}$ there's
a local system $\Or\left(E\right)$ on $X_{\bullet}$. The \emph{orientation
local system }of $X_{\bullet}$ is $\Or\left(TX_{\bullet}\right)$.
We have a local system isomorphism 
\begin{equation}
\iota_{\xx}^{\partial}:\Or\left(T\partial X{}_{\bullet}\right)\to\Or\left(TX_{\bullet}\right)\label{eq:boundary ls iso}
\end{equation}
lying over $i_{X_{\bullet}}^{\partial}:\partial X_{\bullet}\hookrightarrow X_{\bullet}$,
defined by appending the outward normal vector at the beginning of
the oriented base. Given a short exact sequence of vector bundles
\[
0\to E_{1}\xrightarrow{f}E\xrightarrow{q}E_{2}\to0
\]
on $\xx$, we obtain a local system isomorphism 
\begin{equation}
\Or\left(E_{1}\right)\otimes\Or\left(E_{2}\right)\to\Or\left(E\right),\label{eq:ls map from ses}
\end{equation}
which, using oriented bases to represent orientation, can be expressed
by
\[
\left[e_{1}^{1},...,e_{1}^{n_{1}}\right]\otimes\left[e_{2}^{1},...,e_{2}^{n_{2}}\right]\mapsto\left[f\left(e_{1}^{1}\right),...,f\left(e_{1}^{n_{1}}\right),g\left(e_{2}^{1}\right),...,g\left(e_{2}^{n_{2}}\right)\right]
\]
where $g:E_{2}\to E$ is any section of $q$.

Maps of local systems are \emph{always} assumed to be cartesian,
so to specify a local system map ${\lc_{1}\xrightarrow{\ff}\lc_{2}}$
over ${\xx_{1}\xrightarrow{f}\xx_{2}}$ is equivalent to giving an
isomorphism $\lc_{1}\to f^{-1}\lc_{2}$. In this case we may say that
 $\ff$ \emph{lies over }$f$.
\begin{lem}
\label{lem:ls extensions}Let $\xx$ be an orbifold with corners.
We denote by $\mathring{\xx}:=S^{0}\left(\xx\right)$ the orbifold
(without boundary or corners) consisting of points of depth zero,
and by $j:\mathring{\xx}\hookrightarrow\xx$ be the inclusion. 

(a) The pushforward and inverse image functors $j_{*},j^{-1}$ form
an adjoint equivalence of groupoids between local systems on $\mathring{\xx}$
and local systems on $\xx$.

(b) $Or\left(dj\right):Or\left(T\mathring{\xx}\right)\to j^{-1}Or\left(T\xx\right)$
is an isomorphism.

Let $f:\xx\to\yy$ be a b-normal map of orbifolds with corners. 

(c) There exists a unique map $\mathring{f}:\mathring{\xx}\to\mathring{\yy}$
with $f\circ j_{\xx}=j_{\yy}\circ\mathring{f}$.

Let $\lc$ be a local system on $\xx$ and let $\lc'$ be a local
system on $\yy$, and denote by $\mathring{\lc}=j_{\xx}^{-1}\lc$,
$\mathring{\lc}':=j_{\yy}^{-1}\lc'$ their restrictions to $\mathring{\xx},\mathring{\yy}$,
respectively. Define a map taking a map of sheaves $\ff:\lc\to\lc'$
over $f$ to the map $\mathring{\ff}:\mathring{\lc}\to\mathring{\lc}'$
over $\mathring{f}$ given by the composition 
\[
j_{\xx}^{-1}\lc\overset{j_{\xx}^{-1}\ff}{\longrightarrow}j_{\xx}^{-1}f^{-1}\lc'\simeq\mathring{f}^{-1}j_{\yy}^{-1}\lc'.
\]

(d) $\ff\mapsto\mathring{\ff}$ is a bijection
\[
\left\{ \mbox{maps }\ff:\lc\to\lc'\mbox{ over }f\right\} \simeq\left\{ \mbox{maps }\mathring{\ff}:\mathring{\lc}\to\mathring{\lc}'\mbox{ over }\mathring{f}\right\} .
\]
and together with $\lc\mapsto\mathring{\lc}$ forms a functor from
the category of sheaves (respectively, local systems) over orbifolds
with corners with b-normal maps to the category of sheaves (resp.
local systems) over orbifolds.
\end{lem}
Let $\xx$ be an orbifold with corners and $\lc$ a local system on
$\xx$. We define the complex of \emph{differential forms on $\xx$
with values in $\lc$}
\[
\Omega\left(\xx;\lc\right)=\Gamma\left(C^{\infty}\left(\bigwedge T\xx\right)\otimes_{\zz}\lc\right)
\]
as the global sections of the sheaf of sections of the vector bundle
$\bigwedge T\xx$, twisted by $\lc$. 

Suppose $\xx,\yy$ are compact orbifolds with corners, $\mathcal{K},\lc$
are local systems on $\xx$ and on $\lc$, respectively, and $f:\left(\xx,\mathcal{K}\right)\to\left(\yy,\lc\right)$
is an \emph{oriented }map, which means it is a local system map $\mathcal{K}\to\lc$
lying over a smooth map of orbifolds with corners $\xx\to\yy$. We
have a pullback operation

\begin{equation}
\Omega\left(\yy;\lc\right)\xrightarrow{f^{*}}\Omega\left(\xx;\mathcal{K}\right).\label{eq:pullback}
\end{equation}
If, in addition, we assume that $f$ is horizontally submersive, then
there's a pushforward operation

\begin{equation}
\Omega\left(\xx;\kk\otimes Or\left(T\xx\right)^{\vee}\right)\xrightarrow{f_{*}}\Omega\left(\yy;\lc\otimes Or\left(T\yy\right)^{\vee}\right).\label{eq:pushforward}
\end{equation}
We now sketch how these operations are constructed. Define the complex
of compactly supported differential forms\emph{ }on $\xx$ by
\[
\Omega_{c}\left(\xx;\lc\right):=\coker\left(t_{*}-s_{*}:\Omega_{c}\left(X_{1};s^{*}\lc_{0}\right)\to\Omega_{c}\left(X_{0};\lc_{0}\right)\right),
\]
where on the right hand side, $\Omega_{c}$ denotes the usual complex
of compactly supported forms on a manifold with corners. In case
$f=\left(F_{0},F_{1}\right):\xx\to\yy$ is a smooth functor, $F_{0}^{*}$
induces a pullback map (\ref{eq:pullback}) and (if $f$ is horizontally
submersive) $\left(F_{0}\right)_{*}$ induces a pushforward map of
compactly supported forms, 
\begin{equation}
\Omega_{c}\left(\xx;\kk\otimes Or\left(T\xx\right)^{\vee}\right)\xrightarrow{f_{*}}\Omega_{c}\left(\yy;\lc\otimes Or\left(T\yy\right)^{\vee}\right).\label{eq:pushforward compact support}
\end{equation}
In defining the operations $F_{0}^{*},\left(F_{0}\right)_{*}$ (for
forms on manifolds with corners) we follow the conventions in \cite{twA8}.
A \emph{partition of unity} \emph{for $\xx$} is a smooth map\emph{
}$\rho:X_{0}\to\left[0,1\right]$ such that $\operatorname{supp}\left(s^{*}\rho\right)\cap t^{-1}\left(K\right)$
is compact for every compact subset $K\subset X_{0}$ and $t_{*}s^{*}\rho\equiv1$
(the fiber of $t$ is discrete, hence canonically oriented). Partitions
of unity always exist; since $\xx$ is assumed to be compact we can
require that $\rho$ has compact support in $X_{0}$, and use this
to construct an isomorphism
\begin{equation}
\Omega\left(\xx;\lc\right)\simeq\Omega_{c}\left(\xx;\lc\right),\label{eq:Poincare duality}
\end{equation}
see Behrend \cite{behrend}. The isomorphism (\ref{eq:Poincare duality})
allows us to define (\ref{eq:pushforward}) using (\ref{eq:pushforward compact support}).
Now if $f=\xx\xleftarrow{R}\xx'\xrightarrow{F}\yy$ is a general oriented
map, we define (\ref{eq:pullback}) by
\[
f^{*}=R_{*}F^{*},
\]
pulling back along the smooth functor $F$ and then pushing forward
along the refinement $R$ ($R$ is a horizontally submersive since
it is étale; moreover, any refinement defines an equivalence between
the categories of local systems on $\xx$ and on $\xx'$, so orientations
for $f$ are in natural bijection with orientations for $F$). If
$f$ is oriented and horizontally submersive we define the pushforward
(\ref{eq:pushforward}) by 
\[
f_{*}=F_{*}R^{*}.
\]
By construction, the operations (\ref{eq:pullback}, \ref{eq:pushforward})
extend the operations defined in \cite{twA8} for the case $\xx,\yy$
are manifolds, and they satisfy the same relations.

To make the paper more readable, outside of this appendix we will
sometimes abuse notation and refer to maps which have a specified
isomorphism as being equal. For example, if $G$ acts on $\xx$ (see
$\S$\ref{subsec:Group-actions} below) we may write 
\[
g.h.=\left(gh\right).
\]
even though in general the two sides differ by a (specified) 2-cell.
The same goes for orbifolds which are canonically equivalent (that
is, with a given equivalence, or with an equivalence which is specified
up to a unique 2-cell). For example we may write
\[
\left(\mm_{1}\times\mm_{2}\right)\times\mm_{3}=\mm_{1}\times\left(\mm_{2}\times\mm_{3}\right).
\]
When we write $p\in\xx$ we mean $p\in X_{0}$, where $\xx=X_{1}\rightrightarrows X_{0}$.

\subsection{\label{subsec:Hyperplane-Blowup}Hyperplane Blowup}

An important step in the construction of the moduli spaces of discs
from the moduli spaces of curves, is the notion of a hyperplane blowup,
which we now discuss.

\subsubsection{Hyperplane blowup of manifolds}
\begin{defn}
\label{def:hyper}(a) Let $h:W\to X$ be a proper closed immersion
between manifolds without boundary. Write $h^{-1}\left(x\right)=\left\{ w_{1},...,w_{r}\right\} $
(this is finite since $h$ is proper), and let $N_{w_{i}}^{\vee}=\ker\left(T_{x}^{\vee}X\xrightarrow{dh|_{w_{i}}^{\vee}}T_{w_{i}}^{\vee}W\right)$
denote the conormal bundle to $h$. We say $h$ has \emph{transversal
self-intersection} \emph{at $x\in X$} if the induced map 
\[
\bigoplus_{i=1}^{r}N_{w_{i}}^{\vee}\to T_{x}^{\vee}X
\]
is injective. We say $h$ has transversal self-intersection\emph{
}if it has transversal self intersection at every $x\in X$. 

(b) Let $h:W\to X$ be a proper closed immersion which has transversal
self intersection. Suppose further that $h$ is \emph{codimension
one}, i.e. $\dim X-\dim W=1$. In this case we call $E=\im h$ a \emph{hyper
subset}, and call $h$ a \emph{hyper map}. Since the conditions on
$h$ can be checked locally on the codomain $X$, being a hyper subset
is a local property. Moreover, it follows from Proposition \ref{prop:orthant charts exist}
below that the map $h$ is essentially unique: if $W\xrightarrow{h}X,W'\xrightarrow{h'}X$
are two hyper maps with $\im h=\im h'$ then there's a unique diffeomorphism
$W\xrightarrow{\phi}W'$ such that $h=h'\circ\phi$.

(c) Let $Y\to X$ be a smooth map of manifolds without boundary, and
let $E\subset X$ be a hyper subset. We say $f$ is \emph{multi-transverse
to $E$ }if for some (hence any) hyper map $h$ such that $E=\im h$,
$f$ is transverse to $h$ and the pullback $f^{-1}W\xrightarrow{f^{-1}h}Y$
has transversal self intersection (so in fact, since $f^{-1}h$ is
necessarily a codimension one proper closed immersion, $f^{-1}E\subset Y$
is a hyper subset).
\end{defn}

\begin{defn}
(a) Let $X$ be a manifold without boundary, let $U\subset X$ be
an open subset. Consider the set of germs of connected components,\emph{
}
\[
I\left(X,U\right)=\bigcup_{x\in X}\left\{ x\right\} \times\lim_{x\in V\subset X}\pi_{0}^{x}\left(V\cap U\right)
\]
where for $V$ an open neighborhood of $x\in X$, $\pi_{0}^{x}\left(V\cap U\right)$
denotes the set of connected components $C\subset V\cap U$ with $x\in\overline{C}$
in the closure. If $V_{1}\subset V_{2}$ are two such neighborhoods,
there's an induced map $\pi_{0}^{x}\left(V_{1}\cap U\right)\to\pi_{0}^{x}\left(V_{2}\cap U\right)$,
and $\lim_{x\in V\subset X}\pi_{0}^{x}\left(V\cap U\right)$ denotes
the \emph{inverse} limit of this system of sets.

(b) If $\left(X_{1},U_{1}\right)\to\left(X_{2},U_{2}\right)$ is a
map of pairs there's an induced map $I\left(X_{1},U_{1}\right)\to I\left(X_{2},U_{2}\right)$
making $I$ a functor; there's an obvious natural transformation $I\left(X,U\right)\to X$.

(c) Let $E\subset X$ be a hyper subset. As a set, \emph{the blow
up of $X$ along $E$ }is given by\emph{ 
\[
B\left(X,E\right)=I\left(X,X\backslash E\right).
\]
}The associated natural transformation is denoted $B\left(X,E\right)\xrightarrow{\beta_{\left(X,E\right)}}X$,
and if $Y\xrightarrow{f}X$ is multi-transverse to $E$ write 
\[
B\left(Y,f^{-1}E\right)\xrightarrow{B\left(f\right)}B\left(X,E\right)
\]
for the induced map.
\end{defn}

\begin{prop}
\label{prop:hyper bu functor ManC}Let $\text{\textbf{Man}}^{+}$
denote the category of \emph{marked manifolds,} whose objects are
pairs $\left(X,E\right)$ where $X$ is a manifold without boundary
and $E$ is a hyper subset of $X$, and where an arrow $\left(X_{1},E_{1}\right)\to\left(X_{2},E_{2}\right)$
is given by a map $X_{1}\xrightarrow{f}X_{2}$ which is multi-transverse
to $E_{2}$ and such that $f^{-1}E_{2}=E_{1}$. Let $\manc_{ps}$denote
the category of manifolds with corners with perfectly simple maps.
Then blowing up gives a faithful functor
\[
B:\text{\textbf{Man}}^{+}\to\manc_{ps}
\]
together with a natural transformation $B\left(X,E\right)\xrightarrow{\beta_{\left(X,E\right)}}X$.

Moreover, if $\left(X_{1},E_{1}\right),\left(X_{2},E_{2}\right)$
are any two objects of $\text{\textbf{Man}}^{+}$, any étale map $f:X_{1}\to X_{2}$
is a morphism of $\text{\textbf{Man}}^{+}$ and $B\left(f\right)$
is also étale in this case.
\end{prop}
The following definition characterizes the manifold with corners structure
on the blow up. More precisely, $B\left(X,E\right)$ will be equipped
with the unique manifold with corners structure on the set $B\left(X,E\right)$
making the map $\beta_{\left(X,E\right)}$ \emph{rectilinear}:
\begin{defn}
\label{def:rectilinear map}Let $C$ be a manifold with corners, $M$
a manifold without boundary. A map $f:C\to M$ will be called \emph{rectilinear
}if the restriction of $f$ to interior points is an injective map
$\mathring{C}\to M$, and for every $c\in C$ there exist a non-negative
integer $k$ and coordinate charts $U\xrightarrow{\varphi}\rr_{k}^{n},c\in U,\varphi\left(c\right)=0$
and $V\xrightarrow{\psi}\rr^{n},f\left(c\right)\in V,\psi\left(f\left(c\right)\right)=0$
such that $f\left(U\right)\subset V$ and $\psi\circ f\circ\varphi^{-1}$
is the standard embedding of $\rr_{k}^{n}$ to $\rr^{n}$, restricted
to $\varphi\left(U\right)$.
\end{defn}

\subsubsection{\label{subsec:orbifold hyper bu}Hyperplane blowup of orbifolds}

We consider the bicategory $\text{\textbf{PEG}}^{+}$ of \emph{marked
proper étale groupoids}. The objects of $\text{\textbf{PEG}}^{+}$
are pairs $\left(X_{1}\rightrightarrows X_{0},E\right)$ where $X_{1}\rightrightarrows X_{0}$
is a proper étale groupoid without boundary, and $E\subset X_{0}$
is a hyper subset which is a union of orbits, $s^{-1}E=t^{-1}E$.

If $\left(X_{\bullet}^{\left(1\right)},E^{\left(1\right)}\right),\left(X_{\bullet}^{\left(2\right)},E^{\left(2\right)}\right)$
are two objects, a 1-cell of $\text{\textbf{PEG}}^{+}$ consists of
a smooth functor 
\[
X_{\bullet}^{\left(1\right)}\xrightarrow{F=\left(F_{0},F_{1}\right)}X_{\bullet}^{\left(2\right)}
\]
such that $F_{0}$ is multi-transverse to $E^{\left(2\right)}$ and
$E^{\left(1\right)}=F_{0}^{-1}E^{\left(2\right)}$. Every étale map,
and in particular every refinement, satisfies this condition. The
2-cells in $\text{\textbf{PEG}}^{+}$ are all the 2-cells of $\text{\textbf{PEG}}$
spanned by the 1-cells specified above. 

For emphasis, in this subsection we denote the bicategory of proper
étale groupoids and orbifolds in the category $\manc$ by $\text{\textbf{PEG}}^{c}$
and $\text{\textbf{Orb}}^{c}$, respectively. We denote by $\text{\textbf{PEG}}_{ps}^{c},\text{\textbf{Orb}}_{ps}^{c}$
the subcategories whose maps are perfectly simple maps. 

If $X_{1}\rightrightarrows X_{0}$ is a groupoid, we write 
\[
X_{2}=X_{1}\fibp{t}{s}X_{1}
\]
for the manifold with corners parameterizing composable arrows
\[
x_{1}\xrightarrow{a}x_{2}\xrightarrow{b}x_{3}
\]
and, for $i=1,2,3$, we denote by $p_{i}:X_{2}\to X_{0}$ the map
sending a composable arrow as above to $x_{i}$.
\begin{thm}
The functor $B$ extends to a strict 2-functor
\[
B:\text{\textbf{PEG}}^{+}\to\text{\textbf{PEG}}_{ps}^{c}
\]
which takes
\[
\left(\xx=X_{2}\xrightarrow{m}\overset{\overset{i}{\curvearrowleft}}{X_{1}}\begin{array}{c}
\overset{s,t}{\rightrightarrows}\\
\underset{e}{\leftarrow}
\end{array}X_{0},E\right)
\]
to 
\[
B\left(\xx\right)=B\left(X_{2},p_{1}^{-1}E\right)\overset{B\left(m\right)}{\to}B\overset{\overset{B\left(i\right)}{\curvearrowleft}}{\left(X_{1},s^{-1}E\right)}\begin{array}{c}
\overset{B\left(s\right),B\left(t\right)}{\rightrightarrows}\\
\underset{B\left(e\right)}{\leftarrow}
\end{array}B\left(X_{0},E\right)
\]
together with the obvious strict natural transformation $B\left(\xx\right)\xrightarrow{\beta_{\left(\xx,E\right)}}\xx$.

This functor takes refinements to refinements, and thus there's an
induced functor between the 2-localization of these categories 
\[
B:\text{\textbf{Orb}}^{+}\to\text{\textbf{Orb}}_{ps}^{c}.
\]
\end{thm}

\subsubsection{A hyper map between orbifolds}

There's a natural way to construct objects and arrows in $\text{\textbf{Orb}}^{+}$.
Let $h:\mathcal{W}\to\xx$ be a map of orbifolds without boundary,
given by a pair of smooth functors
\begin{equation}
\wc\xleftarrow{S}\left(\tilde{\wc}=\tilde{W}_{1}\rightrightarrows\tilde{W}_{0}\right)\xrightarrow{H=\left(H_{1},H_{0}\right)}\left(\xx=X_{1}\rightrightarrows X_{0}\right)\label{eq:h as a fraction}
\end{equation}
with $S$ a refinement. 

Let $W_{0}'=X_{1}\fibp{t}{H_{0}}\tilde{W}_{0}$ . Since $t$ is étale
this fiber product exists. We let $H_{0}':W_{0}'\to X_{0}$ denote
the composition $X_{1}\fibp{t}{H_{0}}\tilde{W}_{0}\to X_{1}\xrightarrow{s}X_{0}$.
We call the image of $H_{0}'$ the\emph{ essential image} of $h$,
and denote it $\im h$. Fix some point $x\in X_{0}$. The \emph{essential
fiber }of $h$ over $x$ is a topological groupoid, with object space
$\left(H_{0}'\right)^{-1}\left(x\right)$ and with arrows between
$\left(x\xrightarrow{\alpha}H_{0}\left(w\right),w\right)$ and $\left(x\xrightarrow{\alpha'}H_{0}\left(w'\right),w'\right)$
consisting of the arrows in $\tilde{W}_{1}$ between $w$ and $w'$
(this is a special case of the weak fiber product, see Lemma \ref{lem:fibp in orb}).
If $R$ is a refinement, the essential fiber of $h$ over $x$ and
of $R\circ h$ over $R\left(x\right)$ are equivalent. The essential
image and, up to bijection the essential fiber, depend only on the
homotopy class of $H$ (in particular, they do not depend on $S$).
\begin{defn}
\label{def:orbimaps properies 2}We say that $h$ is \emph{hyper}
if the following five conditions are met (cf. Definition \ref{def:hyper})
\begin{itemize}
\item $h$ is \emph{faithful}, which means $H_{0}$ is faithful. This implies
the essential fiber over every point is equivalent to a set (with
a topology).
\item $h$ is a \emph{closed immersion}, which means $H_{0}$ is a closed
immersion. This implies the orbit space of each essential fiber has
the discrete topology.
\item $h$ is \emph{proper}, which means the essential fibers have compact
orbit spaces. Given our previous assumptions this means the\emph{
}essential fiber is equivalent to a finite set\emph{,} and and we
fix representatives $\left\{ q_{i}=x\xrightarrow{\alpha_{1}}H_{0}\left(w_{1}\right)\right\} _{i=1}^{r}$.
\item We require that $h$ has \emph{transversal self-intersection},\emph{
}that is, we require the map
\[
\bigoplus_{i=1}^{r}N_{w_{i}}^{\vee}\to T_{x}^{\vee}X_{0}
\]
be injective, where $N_{w_{i}}^{\vee}=\ker\left(T_{w_{i}}^{\vee}W_{0}\to T_{x}^{\vee}X_{0}\right)$;
this is independent of the choice of representatives.
\item $h$ has \emph{codimension one}, meaning $\dim\xx-\dim\wc=1$.
\end{itemize}
If $\yy$ is another orbifold without boundary, we say a map $\yy\xrightarrow{f}\xx$
is \emph{multi-transverse} to $h$ if $h$ is (b-)transverse to $f$
and the 2-pullback $f^{-1}h$ has transversal self-intersection (it
is automatically a proper, faithful closed immersion). 
\end{defn}
\begin{lem}
(a) Let $\wc\xrightarrow{h}\xx$ be a hyper map. Then $\left(\xx,\im h\right)$
is an object of $\text{\textbf{Orb}}^{+}$.

(b) If $f$ is multi-transverse to $h$, $\left(\yy,\im f^{-1}h\right)\xrightarrow{f}\left(\xx,\im h\right)$
is an arrow in $\text{\textbf{Orb}}^{+}$.
\end{lem}

\subsection{\label{subsec:Group-actions}Group actions}

Let $G$ be a compact lie group, with multiplication \linebreak{}
${m:G\times G\to G}$ and identity $e:\pt\to G$. Given a bicategory
of spaces $\cl$ such as\footnote{More precisely, we need to be able to consider $G$ as a group object
in $\cl$ and certain products to exist. For $\cl=\text{\textbf{Orb}}^{+}$
we consider $G$ as having the trivial marking $\emptyset$, and we
have 
\[
\left(\xx,E\right)\times\left(\yy,F\right)=\left(\xx\times\yy,E\times\yy\coprod\xx\times F\right).
\]
} $\manc,\text{\textbf{Orb}},\text{\textbf{Orb}}^{+}$, we construct
a category $G-\cl$ of $G$-equivariant objects following Romagny
\cite{stacky-action}. We briefly explain how to translate his definitions
to our setup, and refer the reader to \cite{stacky-action} for more
details. A 0-cell of $G-\cl$ is given by a 4-tuple $\left(\xx,\mu,\alpha,a\right)$
where $\xx$ is a 0-cell of $\cl$, $\mu:G\times\xx\to\xx$ is a 1-cell,
and $\alpha$ and $a$ are 2-cells filling in, respectively, the following
square and triangle:
\[
\xymatrix{G\times G\times\xx\aru{r}{m\times\id_{\xx}}\ar[d]_{\id_{G}\times\mu} & G\times\xx\ar[d]^{\mu}\\
G\times\xx\ar[r]_{\mu} & \xx
}
\;\xymatrix{G\times\xx\ar[r]^{\mu} & \xx\\
\xx\ar[ur]_{\id_{\xx}}\ar[u]^{e\times\id_{\xx}}
}
.
\]
A 1-cell (or \emph{$G$-equivariant map})
\[
\left(\xx,\mu,\alpha,a\right)\to\left(\xx',\mu',\alpha',a'\right)
\]
is given by a pair $\left(\xx\xrightarrow{f}\xx',\sigma\right)$ where
$\sigma$ is a 2-cell filling in the square 
\[
\xymatrix{G\times\xx\ar[r]^{\mu}\ar[d]_{\id_{G}\times f} & \xx\ar[d]^{f}\\
G\times\xx'\ar[r]_{\mu'} & \xx'
}
.
\]
A 2-cell $\left(f,\sigma\right)\Rightarrow\left(f',\sigma'\right)$
is given by a 2-cell $f\xRightarrow{\beta}f'$. As usual, the 2-cells
$\alpha,a,\sigma,\beta$ are required to satisfy some coherence conditions,
cf. \cite[Definition 2.1]{stacky-action}. 

\bibliographystyle{amsabbrvc}
\bibliography{localization}

\end{document}